\documentclass[reqno, centertages, 12pt, twoside]{article}

\usepackage[top=1in, bottom=1in, left=1in, right=1in]{geometry}
\usepackage{amsfonts, amssymb, latexsym, amscd, epsf, cite, gensymb}
\usepackage{amsmath, amsthm}
\usepackage[colorlinks=false, pdfstartview=FitV,pagebackref=false, pdfborder={0 0 0}]{hyperref}
\usepackage[mathscr]{eucal}
\usepackage{booktabs}
\usepackage{multirow}
\usepackage{longtable}
\usepackage[bottom]{footmisc} 
\usepackage{MnSymbol}
\usepackage{color}
\usepackage{graphicx}
\usepackage{amscd}
\usepackage{xcolor}
\usepackage{bbm}
\usepackage{dsfont}
\usepackage{enumerate}
\usepackage{hyperref}
\usepackage{subcaption}
\usepackage{dsfont}
\usepackage{authblk}
\usepackage{titlesec}
\usepackage{setspace}
\usepackage{lipsum}
\usepackage{fancyhdr}
\usepackage{esint}
\usepackage{mathtools}
\usepackage{microtype}
\usepackage{enumitem}
\usepackage{float}
\usepackage{tikz}
\usetikzlibrary{shapes.geometric}

\tikzset{three sided/.style={
        draw=none,
        append after command={
            ([shift={( -0.3cm,0.3cm)}]\tikzlastnode.north west)
        edge([shift={( -0.3cm,-0.3cm)}]\tikzlastnode.south west)
            ([shift={( -0.3cm,-0.3cm)}]\tikzlastnode.south west)
        edge([shift={(0.3cm,-0.3cm)}]\tikzlastnode.south east)
            ([shift={(0.3cm,-0.3cm)}]\tikzlastnode.south east)
        edge([shift={( 0.3cm,0.3cm)}]\tikzlastnode.north east) 
        }
    }
}

\allowdisplaybreaks

\theoremstyle{plain}
\newtheorem{thm}{Theorem}
\newtheorem{cor}{Corollary}

\newtheorem{lem}{Lemma}
\newtheorem{prop}{Proposition}
\newtheorem{assumption}{Assumption}
\theoremstyle{definition}
\newtheorem{defn}{Definition}
\newtheorem*{remark}{Remark}

\newcommand{\ep}{\epsilon}

\newcommand*{\rttensortwo}[1]{\bar{\bar{#1}}}

\definecolor{normalgreen}{rgb}{0.1,.65,.1}
\definecolor{mygreen}{RGB}{10, 140, 10}

\providecommand{\keywords}[1]{\hspace*{-4mm}\small\textbf{\textit{Keywords}---}#1}
\title{{\textbf{A Queueing Model of Dynamic Pricing and Dispatch Control for Ride-Hailing Systems Incorporating Travel Times}}}
\author{Amir Anastasios Alwan, Baris Ata, and Yuwei Zhou}
\affil{{\small {\textit{The University of Chicago Booth School of Business}}}}
\date{\vspace{-3mm}{{\small{\today}}}}

\titleformat*{\section}{\large\bfseries}
\titleformat*{\subsection}{\normalsize\bfseries}
\usepackage{etoolbox}
\newcommand{\zerodisplayskips}{%
	\setlength{\abovedisplayskip}{5.7pt}%
	\setlength{\belowdisplayskip}{5.7pt}%
	\setlength{\abovedisplayshortskip}{5.7pt}%
	\setlength{\belowdisplayshortskip}{5.7pt}}
\appto{\normalsize}{\zerodisplayskips}
\appto{\small}{\zerodisplayskips}
\appto{\footnotesize}{\zerodisplayskips}

\usepackage{natbib}
\bibpunct[, ]{(}{)}{,}{a}{}{,}%

\begin{document}
	\vspace*{-17mm}
	{\let\newpage\relax\maketitle}
	\vspace{-2em}
	\begin{abstract}
		A system manager makes dynamic pricing and dispatch control decisions in a queueing network model motivated by ride hailing applications. A novel feature of the model is that it incorporates travel times. Unfortunately, this renders the exact analysis of the problem intractable. Therefore, we study this problem in the heavy traffic regime. Under the assumptions of complete resource pooling and common travel distribution, we solve the problem in closed form by analyzing the corresponding Bellman equation. Using this solution, we propose a policy for the queueing system and illustrate its effectiveness in a simulation study.
	\end{abstract}
	
	
	\doublespacing
	
	\section{Introduction}\label{sec:1}
	
	This paper studies a dynamic control problem for a queueing model motivated by taxi and ride hailing systems. In those systems, customers and drivers can be matched centrally by a platform using web or mobile applications. In addition, the platform can adjust the prices dynamically over time. We consider a city partitioned into a set of geographical regions. Each such region should be thought of as a pick-up or drop-off location. Simultaneously, cars reside in these regions waiting to pick up customers. We use a queueing model to study this problem, following a growing number of papers in the operations research literature. However, much of the relevant literature assumes away the travel times between the pick-up and drop-off locations, see for example \citet{AtaBarjestehKumar_SPN} and the references therein. A key novelty of our model is that it incorporates travel times, but this leads to a significantly more challenging analysis.
	
	We assume that the platform, also referred to as the system manager hereafter, has two levers: pricing and dispatch controls. She seeks an effective policy that makes both dynamic pricing and dynamic dispatch control decisions in order to maximize the long-run average profit. We allow the prices to depend on time and the customer location. Dynamically adjusting prices elicits two competing effects. On the one hand, increasing prices increase the per-ride revenue for the platform. On the other hand, customers are price sensitive, so higher prices result in lower customer demand. Dispatching refers to the process of matching a car with a customer requesting a ride and constitutes an important operational decision for the platform. 

	We model a ride-hailing or taxi system as a closed queueing network with a fixed number of jobs, denoted by $n$. There are $I$ buffers, $I$ single-server nodes, and an infinite-server node in the SPN. The terms ``server'' and ``resource'' will be used interchangeably to refer to a single-server node. Similarly, the terms ``buffer'' and ``class'' will be used interchangeably. As such, jobs in buffer $i$ will be referred to as class $i$ jobs, for $i=1,\dots, I$. In addition to choosing prices dynamically, the system manager can engage in $J$ possible (dispatch) activities, where each activity corresponds to a server serving jobs in a buffer. Following service at a single-server node, jobs are routed to the infinite-server node. Jobs then continue their service at the infinite-server node, after which they are probabilistically routed back to the buffers. The infinite-server node models the travel times. This process continues indefinitely.
	
	In the context of our motivating application, jobs correspond to cars that circulate in the system perpetually. The $I$ buffers correspond to $I$ city regions where cars wait to get matched with a customer. In addition, the service rates at a single-server node can be thought of as the customer arrival rate to the corresponding region, which depends on the price. As a result, customer demand dynamically changes over time as the platform varies the prices of rides. An activity corresponds to dispatching a car from one region to serving an arriving customer possibly in another region. Thus, a service completion at a single-server node corresponds to a car getting matched with a customer. After getting matched with a customer, the car must travel to pick up the customer and bring him to his destination. We assume that all customer requests that are not met immediately are lost. In the queueing model, this corresponds to jobs getting routed to and served at the infinite-server node. That is, the infinite-server node models the travel time of a car from its initial dispatch time to the drop off time of the customer. Upon completing service at the infinite server node, the job is routed to the buffer that is associated with the customer's destination. This is modeled through a probabilistic routing structure as is usually done in the queueing literature. Although the SPN we study is motivated by the ride-hailing and taxi systems, in what follows we use the queueing terminology that is standard in the literature. However, we will occasionally make reference to our motivating applications when intuition or interpretation are needed.

    As mentioned above, incorporating travel times makes the problem significantly more challenging. To ease the analysis, we assume there is a single travel node. This assumption has two implications: First, the travel times between any two regions have the same distribution. Second, upon completing service at the infinite-server node all job classes share the same probabilistic routing structure. Admittedly, this is a restrictive assumption, but it simplifies the analysis and allows us to incorporate the travel times into the model. We view our model as an important first step in the analysis of ride-sharing network models that incorporate travel times. 

    However, even under the single travel node assumption, the problem is not amenable to exact analysis. As such, we consider a diffusion approximation to it in the heavy traffic asymptotic regime. In that regime, under the so called complete resource pooling condition, see \citet{HarrisonLopez1999}, we solve the problem analytically and derive a closed-form solution for the optimal dynamic prices.

    Notwithstanding these restrictive assumptions, the paper makes two contributions. First, it incorporates the travel times in the model and solves the resulting dynamic pricing and dispatch control problem analytically in the heavy traffic regime. Second, it makes a methodological contribution by solving a drift-rate control problem on an unbounded domain, which could be of interest in its own right.
	
	The rest of the paper is structured as follows. Section \ref{sec:2} reviews the literature. Section \ref{sec:3} presents the control problem for the ride-hailing platform, and the associated Brownian control problem is derived formally in Section \ref{sec:4}. The equivalent workload formulation is formulated in Section \ref{sec:5} and it is solved in Sections \ref{sec:6} and \ref{sec:7} by studying a related Bellman equation. Section \ref{sec:8} interprets the solution of the equivalent workload formulation in the context of the original control problem and proposes a pricing and dispatch policy. Section \ref{sec:SimulationStudy} conducts a simulation study to illustrate the effectiveness of the proposed policy. Section \ref{sec:10} concludes the paper. There are two appendices: Appendix \ref{app:A} provides a formal derivation of the Brownian control problem, and additional proofs are given in Appendix \ref{app:B}.
	
	\section{Literature Review}\label{sec:2}
	
	Our paper is related to two streams of literature: the modeling and analysis of ride-hailing and taxi systems and the dynamic control of queueing networks. 
 
	In recent years several authors have modeled ride-hailing and taxi systems using queueing networks. A majority of this literature has focused on how pricing, dispatch (matching), and relocation decisions can improve system performance. From a modeling perspective, \citet{AtaBarjestehKumar_SPN} and \citet{Braverman} are most closely related to ours. \citet{AtaBarjestehKumar_SPN} model a ride-hailing system closed stochastic processing network with dispatch and relocation control. Under heavy traffic conditions, they approximate the original control problem by a Brownian control problem (BCP). After reducing the BCP to an equivalent workload formulation, they propose an algorithm to solve it numerically. However, their model does not include travel times, whereas ours does. Incorporating travel times leads to a significantly more challenging problem in the heavy traffic limit under the diffusion scaling. On the other hand, \citet{Braverman} model a ride-hailing system as a closed queueing networks with travel times and relocation control. By solving a suitable linear program, they propose a static routing policy and prove that it is asymptotically optimal in a large market asymptotic regime under fluid scaling. \citet{Hosseini2021} extends the analysis of \citet{Braverman} by designing a dynamic relocation that outperforms the asymptotically optimal static policy in realistic problem instances. In a related study, \citet{ZhangPavone2016} uses a combination of single-server and infinite-server queueing model to study the control of autonomous vehicles. The authors derive an open loop policy by solving a linear program. Building on this solution, they also propose an effective dynamic rebalancing policy. 
	
	Several other papers are at the intersection of ride-hailing and queueing, but differ more in their modeling choices and analysis. \citet{Banerjee15} study pricing on a single-region model with a single travel time node and show that an optimal static pricing policy performs well. \citet{Freund} develop an approximation framework to study vehicle sharing systems under pricing, matching, and repositioning policies for several objective functions and under various system constraints. In particular, they develop algorithms and show that the approximation ratio of the resulting policy improves as the number of cars in each region grows. \citet{BanerjeeKanoriaQian2020} study matching for a general closed queueing network that can be used to model ride-hailing systems. They propose a family of state-dependent matching policies that do not use any demand arrival rate information. Under a complete resource pooling assumption, they show that the proportion of dropped demand under any such policy decays exponentially as the number of supply units in the network grows. \citet{Afeche18} develop a game-theoretic fluid model to study admission control and repositioning in a ride-hailing system with strategic drivers. Their analysis provides insights into spatial demand imbalances and how demand admission control can impact the strategic behavior of drivers in the network. \citet{Afeche18} studies the optimal dynamic pricing and dispatch control under demand shocks. \citet{OzkanWard2020} model a ride-hailing system as an open queueing network model with impatient customers. They propose a matching policy and prove asymptotically optimality in the fluid scale in a large market regime. \citet{Ozkan2020} studies a fluid model with strategic drivers that incorporates both pricing and matching decisions, highlighting the importance of looking at multiple controls simultaneously. \citet{Besbes2021b} study the effect of pick up and travel times on capacity planning for a ride-hailing system by modeling it as a spatial multi-server queue. \citet{Chen2020} proposes static and dynamic policies that are asymptotically optimal. \citet{Varma2022} studies an open network model and proposes an asymptotically optimal policy. Examples of other papers that use spatial models for pricing include \citet{Yang2018}, \citet{JacobRoetGreen2021}, and \citet{HuHuZhu2022}. 
	
	Several other researchers focused on different aspects of the ride-hailing and taxi systems without using queueing theoretic models. These include \citet{Wang17}, \citet{AtaBarjestehKumar_Empirical}, \citet{Bertsimas2019}, \citet{Besbes2021a}, \citet{Candogan}, \citet{Cachon}, \citet{Castillo}, \citet{Chen16}, \citet{GargNazerzadeh2021}, \citet{Gokpinar}, \citet{Guda2019}, \citet{HeHuZhang2020}, \citet{HuHuZhu2022}, \citet{HuZhou2021}, \citet{Korolko2020}, and \citet{Lu2018}.

	This paper also contributes to the broader literature on dynamic control of queueing systems. Two prominent approaches in that literature are: (i) Markov decision process (MDP) formulations, and (ii) heavy traffic approximations. Intuitively, the workload problem studied in Sections \ref{sec:5}--\ref{sec:7} relates to the service rate and admission control problems studied using MDP formulations, see for example \citet{StidhamWeber1989} and references therein. The most closely related papers are \citet{GeorgeHarrison2001} and \citet{Ata2005}. These papers study the service rate control problems for an $M/M/1$ queue and provide closed-form solutions; also see \citet{AtaShneorson2006}, \citet{AtaZachariadis2007}, \citet{Adusumilli2010}, and \citet{Kumar_et_al_2013}.

	The second approach is pioneered by \citet{Harrison1988}, also see \citet{Harrison2000, Harrison2003}. In particular, a number of papers studied drift rate control problems for one-dimensional diffusions arising under heavy traffic approximations, see \citet{AtaHarrisonShepp2005}, \citet{Ata2006}, \citet{Ghosh2007, Ghosh2010}, \citet{AtaRubino2009}, \citet{KimWard2013}, and \citet{AtaTongarlak2013}. More recently, \citet{AtaBarjesteh2020} and \citet{Ata_Volunteer2} studied drift-rate control problems arising in different contexts such as volunteer capacity management and make-to-stock manufacturing. The analysis of the drift-rate control problem solved in this paper differs significantly from the analysis in those papers because it involves a quadratic cost of drift rate, unbounded set of feasible drift rates, and an unbounded state space. The combination of these features lead to a more challenging analysis. Our paper also makes a modeling contribution by formulating the dynamic dispatch and pricing control problem that incorporates travel times. Furthermore, it proposes an analytically tractable approximation in the heavy traffic limit and solves that in closed form.

	Lastly, our paper draws on the literature of the asymptotic analysis of closed queueing networks with infinite-server queues, see for example \citet{KoganLipsterSmorodinskii}, \citet{Smorodinskii}, \citet{KoganLipster}, and \citet{Krichagina}.
	
	\section{Model} \label{sec:3} 
	
	Motivated by the taxi and ride-hailing application described in the introduction, we consider a closed queueing network with $n$ jobs, $I$ buffers, $I$ single-server nodes, and one infinite-server node. Figure \ref{fig:TaxiNetwork} displays an illustrative network with $I=4$ and $J=10$, also see Section \ref{sec:SimulationStudy} for the motivation behind this example.

    \begin{figure}[h!]
		\centering
        \begin{tikzpicture}
            \node[three sided] (b_1) at (0,3) {$1$}; 
            \node[three sided] (b_2) at (3,3) {$2$}; 
            \node[three sided] (b_3) at (6,3) {$3$}; 
            \node[three sided] (b_4) at (9,3) {$4$}; 
            \node[shape=circle, draw] (s_1) at (0,0) {$1$}; 
            \node[shape=circle, draw] (s_2) at (3,0) {$2$}; 
            \node[shape=circle, draw] (s_3) at (6,0) {$3$}; 
            \node[shape=circle, draw] (s_4) at (9,0) {$4$}; 
            \node[shape=ellipse, minimum width=2cm, minimum height=1cm, draw] (inf_s) at (4.5, -3) {$ $};
            \draw[->] (s_1) -- (inf_s);
            \draw[->] (s_2) -- (inf_s);
            \draw[->] (s_3) -- (inf_s);
            \draw[->] (s_4) -- (inf_s);
            \draw[->] ([yshift=-0.6cm]b_1.center) -- (s_1) node [midway,left] {1};
            \draw[->] ([yshift=-0.6cm]b_2.center) -- (s_2) node [midway,left] {2};
            \draw[->] ([yshift=-0.6cm]b_3.center) -- (s_3) node [midway,left] {3};
            \draw[->] ([yshift=-0.6cm]b_4.center) -- (s_4) node [midway,left] {4};
            \draw[->] ([yshift=-0.6cm]b_1.center) -- (s_2) node [pos=0.4,above] {6};
            \draw[->] ([yshift=-0.6cm]b_2.center) -- (s_1) node [pos=0.4,above] {5};
            \draw[->] ([yshift=-0.6cm]b_2.center) -- (s_3) node [pos=0.4,above] {8};
            \draw[->] ([yshift=-0.6cm]b_3.center) -- (s_2) node [pos=0.4,above] {7};
            \draw[->] ([yshift=-0.6cm]b_3.center) -- (s_4) node [pos=0.4,above] {10};
            \draw[->] ([yshift=-0.6cm]b_4.center) -- (s_3) node [pos=0.4,above] {9};
            \draw (4.5,-3.5) -- (4.5,-4.5) -- (-1,-4.5) node [midway,above] {$q_1$} -- (-1,4) -- (0,4);
            \draw[->] (0,4) -- (0,3.7);
            \draw (4.5,-4.5) -- (4.5,-5) --
            (-1.5,-5) node [midway,below] {$q_2$} -- (-1.5,4.5) -- (3,4.5);
            \draw[->] (3,4.5) -- (3,3.7);
            \draw (4.5,-4.5) -- (10,-4.5) node [midway,above] {$q_4$} -- (10,4) -- (9,4);
            \draw[->] (9,4) -- (9,3.7);
            \draw (4.5,-5) -- (10.5,-5) node [midway,below] {$q_3$} -- (10.5,4.5) -- (6,4.5);
            \draw[->] (6,4.5) -- (6,3.7);
        \end{tikzpicture}
    \caption{A network with four regions and ten dispatch activities. The open rectangles are the buffers, the circles are the single servers, and the oval is an infinite-server node. The ten activities are represented by the arrows between the buffers and servers. The numbers on the arrows indicate their index. Activities 1, 2, 3, and 4 are local dispatch activities while activities 5 through 10 are non-local dispatch activities. The arrows from the infinite-server to the buffers represent probabilistic rerouting of jobs in the network.
    }
    \label{fig:TaxiNetwork}
	\end{figure}
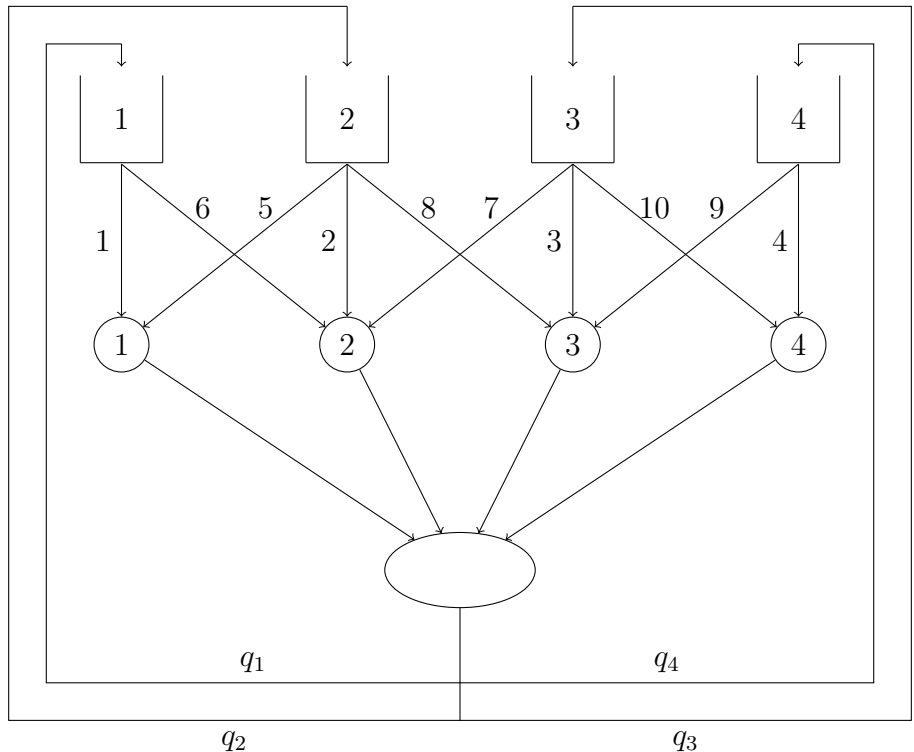
 
    As mentioned earlier, in addition to dynamic pricing decisions, the system manager also makes dispatch decisions dynamically. There are $J$ dispatch activities she can choose from. Each dispatch activity involves a unique buffer and a unique server--we use the terms single-server node and server interchangeably. Let $s(j)$ and $b(j)$ denote the server and the buffer, respectively, associated with activity $j$ for $j=1,\dots, J$. In other words, activity $j$ is undertaken by server $s(j)$ and it servers jobs in buffer $b(j)$. We describe the association between activities and resources by the capacity consumption matrix $A$ and the association between activities and buffers by the constituency matrix $C$. That is, $A$ is the $I\times J$ matrix given by
	\begin{align}
	A_{ij}=\Bigg\{
	\begin{array}{ll}
	1,&\text{if }s(j) = i,\\
	0,&\text{otherwise},\label{eq:1.1}
	\end{array}
	\end{align}and $C$ is the $I\times J$ matrix given by 
	\begin{align}
	C_{ij}=\Bigg\{
	\begin{array}{ll}
	1,&\text{if }b(j) = i,\\
	0,&\text{otherwise.}\label{eq:1.2}
	\end{array}
	\end{align}
	Let $\mathcal{A}_i$ denote the set of activities server $i$ undertakes. Similarly, let $\mathcal{C}_i$ denote the set of activities that serve buffer $i$. We have that
	\begin{align}
	\mathcal{A}_i &= \left\{j: A_{ij} = 1\right\},\label{eq:1.3}\\
	\mathcal{C}_i &= \left\{j: C_{ij} = 1\right\}.\label{eq:1.4}
	\end{align}
	
	For each activity $j=1,\dots, J$, we associate a unit rate Poisson process $N_j$. We also associate a unit rate Poisson process $N_0$ with the infinite-server node. The processes $N_0,N_1,\dots, N_J$ are mutually independent. The service rate at the infinite-server node is denoted by $\eta>0$. We denote the service rate for activity $j$ at time $t$ by $\mu_j(t)$ for $t\ge 0$ and $j=1,\dots, J$. The system manager chooses prices $p(t)= \left(p_i(t)\right)$ dynamically over time, where $p_i(t)$ denotes the price charged to customers who seek rides from region $i$ at time $t$. As the reader will see below, these prices ultimately determine activity service rates $\mu_j(t)$ for $j=1,\dots, J$ and $t\ge 0$. We assume $p_i(t)\in [\underline{p}_i, \overline{p}_i]$ for $t\geq 0$, where $0\leq \underline{p}_i< \overline{p}_i<\infty$. The price sensitivity of demand is captured by a nonnegative demand function $\Lambda:\mathcal{P}\rightarrow \mathbb{R}_+^I$, where ${\cal P}=\prod_{i=1}^I [\underline{p}_i, \overline{p}_i]$. Namely, the demand rate vector at time $t$, denoted by $\lambda(t)$, is given by\footnote{The customer demand rate in region $i$, $\lambda_i(t)$, depends only on the price $p_i(t)$.}
	\begin{align}
	\lambda(t) = \Lambda\left(p(t)\right)=\left(\Lambda_1(p_1(t)),\dots, \Lambda_I(p_I(t))\right)^{\prime},\quad t\ge 0.\label{eq:7}
	\end{align}
	We make the following monotonicity assumption to simplify the analysis:
    \begin{assumption}
        The demand rate function is strictly decreasing in price, i.e., $\Lambda_i(p_i)$ is strictly decreasing in $p_i$ for $i=1,\ldots, I$. 
    \end{assumption}
     From this monotonicity assumption, it follows that $\Lambda_i(\cdot)$ has an inverse function, denoted by $\Lambda_i^{-1}(\cdot)$.
     Moreover, the pricing decisions can be replaced with choosing the demand rate vector $\lambda(t)$ dynamically over time. This is convenient for our analysis. In order to proceed with that approach, we first define the set of admissible demand rate vectors $\mathcal{L}\subseteq \mathbb{R}_+^I$ as follows:
	\begin{align}
	    {\cal L}=\prod_{i=1}^I {\cal L}_i, \label{eq:8}
	\end{align} where ${\cal L}_i=[\Lambda_i(\overline{p}_i), \Lambda_i(\underline{p}_i)]$ for $i=1,\ldots,I$. Denoting $\Lambda^{-1}(x) = \left(\Lambda_1^{-1}(x_1),\dots, \Lambda_I^{-1}(x_I)\right)^{\prime}$ for $x\in \mathcal{L}$, it is easy to see that $\Lambda^{-1}$ is the inverse function of $\Lambda$. Viewing the demand rates as the platform's pricing control, we define the revenue rate function $\pi:\mathcal{L}\rightarrow \mathbb{R}$ as follows:
	\begin{align}
	\pi(x) = \sum_{i=1}^{I}x_i\Lambda_i^{-1}(x_i),\quad x\in \mathcal{L}.\label{eq:1.10}
	\end{align}
	We also make the following regularity assumptions for the revenue rate function:
	\begin{assumption}\label{ass:2} 
		The revenue rate function $\pi$ is: (a) three-times continuously differentiable and strictly concave on $\mathcal{L}$, and (b) has a maximizer in the interior of $\mathcal{L}$.
	\end{assumption} 
	
	Upon completing service at a single-server node, each job goes next to the infinite-server node. Once its service there is complete, the job next joins buffer $i$ with probability $q_i>0$ for $i=1,\dots, I$ where $\sum_{i=1}^{I}q_i=1$. The routing probability vector $q=\left(q_i\right)$ does not depend on the single-server node the job departed from prior to joining the infinite-server node. In other words, customers' destination distribution is identical across different origins. This is a restrictive assumption, but it simplifies the analysis significantly and enables us to incorporate travel times into the model. As discussed in the Introduction, we view this as an important first step in the analysis of ride-sharing network models that incorporate travel times. In order to model this probabilistic routing structure mathematically, we let $\psi=\left\{\psi(l),\,l\ge 1\right\}$ denote a sequence of $I$-dimensional i.i.d. random vectors with $P\left(\psi(1)=e_i\right)=q_i$ for $i=1,\dots, I$, where $e_i$ is an $I$-dimensional vector with one in the $i$th component and zeros elsewhere. Then letting
	\begin{align}
	\Psi(m)  = \sum_{l=1}^{m}\psi(l)\quad\text{for}\quad m\ge 1,
	\end{align}
	we note that the $i$th component of $\Psi(m)$, denoted by $\Psi_i(m)$, represents the total number of jobs routed to buffer $i$ among the first $m$ jobs that have finished service at the infinite-server node.
	
	As discussed earlier, there are two types of control decisions that the system manager must make. First, she must choose an $I$-dimensional demand rate process $\lambda=\left\{\lambda(t),\,t\ge 0\right\}$.  This is equivalent to making dynamic pricing decisions. Recall that the customer arrival process at single-server node $i$ corresponds to its service process. Because these customers can be transported by cars in regions corresponding to activities $j\in\mathcal{A}_i$, we let
	\begin{align}
	\mu_j(t) = \lambda_i(t)\quad\text{for}\quad j\in \mathcal{A}_i,\quad i=1,\dots, I,\quad\text{and}\quad t\ge 0.\label{eq:9}
	\end{align}
	This defines the $J$-dimensional service rate process $\mu=\left\{\mu(t),\,t\ge 0\right\}$, where $\mu(t)=\left(\mu_j(t)\right)$. Second, she must decide on how servers allocate their time to various (dispatch) activities. This decision takes the form of cumulative allocation processes $T_j=\left\{T_j(t),\,t\ge 0\right\}$ for $j=1,\dots, J$. In particular, $T_j(t)$ represents the cumulative amount of time server $s(j)$ devotes to activity $j$ (serving class $i(j)$ jobs) during $[0,t]$.
	
	Next, we introduce the system dynamics equations that govern the movement of jobs in the network. To that end, we let $Q_0(t)$ and $Q_i(t)$ denote the number of jobs in the infinite-server node and in buffer $i$ at time $t$, respectively, for $i=1,\dots, I$. We also let $A_0(t)$ and $A_i(t)$ be the total number of jobs that have arrived to the infinite-server node and to buffer $i$ by time $t$, respectively, for $i=1,\dots, I$. Then we have that
	\begin{alignat}{2}
	A_0(t) &= \sum_{j=1}^{J}N_j\left(\int_{0}^{t}\mu_j(s)\,dT_j(s)\right),\quad&& t\ge 0,\label{eq:2.1}\\
	A_i(t) &= \Psi_i\left(N_0\left(\eta\int_{0}^{t}Q_0(s)\,ds\right)\right),\quad &&t\ge 0.\label{eq:2.2}
	\end{alignat}
	Moreover, letting $D_0(t)$ and $D_i(t)$ denote the total number of jobs that have left the infinite-server node and buffer $i$ by time $t$, respectively, for $i=1,\dots, I$, we have that
	\begin{alignat}{2}
	D_0(t) &= N_0\left(\eta\int_{0}^{t}Q_0(s)\,ds\right),\quad &&t\ge 0,\label{eq:2.3}\\
	D_i(t) &= \sum_{j\in\mathcal{C}_i}N_j\left(\int_{0}^{t}\mu_j(s)\,dT_j(s)\right),\quad &&t\ge 0.\label{eq:2.4}
	\end{alignat}
	 We refer to the $(I+1)$-dimensional process $Q=(Q_0,Q_1,\dots, Q_I)^{\prime}$ as the queue length process, whose dynamics is given next:
	\begin{alignat}{2}
	Q_i(t) &= Q_i(0) + A_i(t) - D_i(t)\quad &&\text{for}\quad i=0,1,\dots, I\quad \text{and}\quad t\ge 0,\label{eq:2.5}
	\end{alignat}
	where $Q(0)$ is the vector of initial queue lengths such that $\sum_{i=0}^{I}Q_i(0)=n$. Letting $I_i(t)$ denote the cumulative amount of time that server $i$ is idle during the interval $[0,t]$ for $i=1,\dots, I$, we have that
	\begin{align}
	I_i(t) = t- \sum_{j\in\mathcal{A}_i}T_j(t),\quad t\ge 0,\label{eq:2.6}
	\end{align}
	or in matrix notation, $I(t)=et - AT(t)$ for $t\ge 0$. Note that Equations (\ref{eq:2.1})--(\ref{eq:2.5}) imply that 
	\begin{align*}
	\sum_{i=0}^{I}Q_i(t) = \sum_{i=0}^{I}Q_i(0)=n\quad \text{for}\quad t\ge 0,
	\end{align*}
	expressing the fact that the total number of jobs in the system remains fixed in a closed network.
	
	In order to state the platform's objective and its control problem formally, we introduce two vectors of cost parameters $h=\left(h_0, h_1,\dots, h_I\right)^{\prime}\in\mathbb{R}_+^{I+1}$ and $c=\left(c_1,\dots, c_I\right)^{\prime}\in\mathbb{R}_+^I$. In the context of the ride-hailing system, the platform incurs a fuel cost at a rate of $h_0$ per traveling car. Moreover, for $i=1,\dots, I$, there is a holding cost at a rate of $h_i$ for each car waiting for a ride in region $i$, reflecting the fact that no driver likes sitting idle. We assume that $h_i> h_0$ for all $i=1,\dots, I$.
 Finally, for $i=1,\dots, I$, there is an idleness cost at the rate of $c_i$ per unit of time server $i$ is idle. This represents the lost revenue from picking up customers arriving to region $i$ and goodwill loss.\footnote{One can assume $c_i \ge p_i^* = \Lambda_i^{-1}(\lambda_i^\ast)$ naturally, where $\lambda^*$ is defined in Equation (\ref{eq:1.11}) below.} A control policy is denoted by $(T, \lambda)$ and
	must satisfy the following conditions: 
	\begin{align}
	&\text{$T$, $\lambda$ are nonanticipating with respect to $Q$,}\label{eq:2.8}\\
	&\text{$T$, $I$ are nondecreasing and continuous with $T(0)=I(0)=0$,}\label{eq:2.9}\\
	&\text{$\lambda(t)\in\mathcal{L}$ for all $t\ge 0$,}\label{eq:2.10}\\
	&\text{$Q_i(t)\ge 0$ for all $t\ge 0$, $i=0,1,\dots, I$.} \label{eq:2.11}
	\end{align}
	Equation (\ref{eq:2.8}) expresses the fact that the policy can only depend on observable quantities, Equation (\ref{eq:2.9}) is natural given the interpretations of the processes $T$ and $I$. Equation (\ref{eq:2.10}) requires that $\lambda$ come from the set of achievable demand rates. Equation (\ref{eq:2.11}) expresses the fact that queue lengths are nonnegative. The arriving customer demand is allocated to cars waiting in various buffers through the dispatch activities $j=1,\ldots, J$, see for example Equations (\ref{eq:2.1}) and (\ref{eq:2.4}). Given a control policy $(T,\lambda)$, we define the cumulative profit collected up to time $t$ as
	\begin{align}
	V(t) = \int_{0}^{t} \left[\pi\left(\lambda(s)\right) - h^{\prime}Q(s)\right]\,ds - c^{\prime}I(t),\quad t\ge 0.\label{eq:2.12}
	\end{align}
	The platform's control problem is to choose a policy $(T,\lambda)$ so as to 
	\begin{alignat}{2}
	&\text{maximize}\quad &&\liminf_{t\rightarrow \infty}\,\,\frac{1}{t}\,\,E\left[V(t)\right]\label{eq:2.13}\\
	&\text{subject to} && (\ref{eq:2.1})\text{--}(\ref{eq:2.12}).\label{eq:2.14}
	\end{alignat}
	Because control problem (\ref{eq:2.13})--(\ref{eq:2.14}) in its original form is not amenable to exact analysis, the next section considers a related control problem in an asymptotic regime where the number of cars gets large and derives the approximating Brownian control problem. The Brownian control problem is an approximation to the original problem, yet it is far more tractable.
	
	\section{Brownian Control Problem}\label{sec:4}
	Following an approach that is similar to the one taken in \citet{Harrison1988}, this section develops a Brownian approximation to the control problem presented in Section \ref{sec:3}. Many authors have proved heavy traffic limit theorems to rigorously justify such Brownian approximations---see for example \citet{Harrison1998}, \citet{Williams1998}, \citet{Kumar2000}, \citet{BramsonDai2001}, \citet{Stolyar2004}, \citet{BellWilliams2001, BellWilliams2005}, \citet{AtaKumar05}, \citet{AtaOlsen2009, AtaOlsen2013} and references therein. We do not attempt to prove a rigorous convergence theorem in this paper, but refer the reader to \citet{Harrison1988, Harrison2000, Harrison2003} for elaborate and intuitive justifications of the approximation procedure we follow.
	
	The approximation procedure starts by solving the following static pricing problem (existence of the optimal solution is guaranteed by Assumption \ref{ass:2}), which helps us articulate the heavy traffic assumption that underlies the mathematical development to follow. We set
	\begin{align}
	\lambda^* = \underset{\lambda\in\mathcal{L}}{\text{arg\,max}}\,\,\,\pi(\lambda).\label{eq:1.11}
	\end{align}
    Recall from Assumption \ref{ass:2} that we assume $\lambda^\ast$ is in the interior of ${\cal L}$, i.e., $\lambda^\ast\in {\rm int}({\cal L})$.
	The vector $\lambda^*$ represents the average demand rates that would result in the largest revenue rate ignoring variability in the system. Note that the corresponding nominal service rates for the various activities are given by\footnote{In particular, for all $j$, $\mu_j^* = \sum_{i=1}^{I}\lambda_i^* A_{ij}$. This is true because there exists only one $i$ such that $A_{ij} \ne 0$ for each $j=1,\dots, J$. That is, an activity only uses one server. In matrix notation, $\mu^* = A^{\prime}\lambda^*$, where $A^{\prime}$ is the transpose of $A$.\vspace{0.5em}}
	\begin{align}
	\mu_j^* = \lambda_i^*\quad\text{for}\quad j\in \mathcal{A}_i.\label{eq:1.12}
	\end{align}
	Using these nominal service rates, we define an $I\times J$ input-output matrix $R$ as follows:
	\begin{align}
	R_{ij} = \mu_j^* C_{ij},\quad i=1,\dots, I,\quad j=1,\dots, J.\label{eq:1.7}
	\end{align}
	Following \citet{Harrison1988, Harrison2000}, we interpret $R_{ij}$ as the long-run average rate of class $i$ material consumed per unit of activity $j$ under the nominal service rates $\mu_j^*$ for $j=1,\dots, J$. We also define the $I$-dimensional input vector $\nu$ as
	\begin{align}
	\nu_i = q_i\eta,\quad i=1,\dots, I.\label{eq:1.8*}
	\end{align}
	We interpret $\nu_i$ as the long-run average rate of input into buffer $i$ from the infinite-server node. As a preliminary to stating the heavy traffic assumption, we introduce the notion of local activities. In the context of the motivating application, it corresponds to a customer in a region being picked up by a car in the same region. Using the terminology that is standard in queueing theory, it corresponds to a server processing its own buffer. Without loss of generality, we assume that the first $I$ activities are local. That is, \[s(j)=b(j)=j \quad  \text{for} \quad j=1,\ldots, I.\] This is equivalent to assuming that the first $I$ columns of matrices $A$ and $C$ constitute the $I\times I$ dimensional identity matrix. 
	The following is the heavy traffic assumption:
	\begin{assumption}\label{ass:3}
		There exists a unique $x^*\in\mathbb{R}^J$ such that 
		\begin{align}
        x_j^\ast &= \min \{\lambda_j^*, \nu_j\}, \quad j=1,\ldots, I,  \label{eq:1.14} \\
		Rx^*&=\nu,\label{eq:1.13}\\
		Ax^*&= e,\label{eq:1.15}\\
		x^*&\ge 0.\label{eq:1.16}
		\end{align}
	\end{assumption}
	The vector $x^*$ is referred to as the nominal processing plan and the component $x_j^*$ can be interpreted as the long-run average rate at which activity $j$ is undertaken. Equation (\ref{eq:1.16}) says that all nominal activity levels must be non-negative. Equation (\ref{eq:1.15}) means that under the nominal processing plan, servers are fully utilized. Equation (\ref{eq:1.13}) is a flow balance condition which says that the rate of jobs leaving the buffers equals the rate of jobs entering the buffers under the nominal processing plan. Note that by Equations (\ref{eq:1.7})--(\ref{eq:1.13}) we have $\sum_{j\in\mathcal{C}_i}\mu_j^* x_j = q_i\eta$ for each $i$, which then implies that $(\mu^*)^{\prime}x^* =\eta$ by summing over $i$. We interpret $(\mu^*)^{\prime}x^*$ as the rate of jobs entering the infinite-server node under the nominal processing plan, and Assumption \ref{ass:3} ensures that this equals the service rate at the infinite-server node.\footnote{Based on intuition from the classical $M/M/\infty$ queue, this condition implies that the steady-state fraction of jobs in the infinite-server node under the nominal processing plan is equal to one as the number of jobs in the system grows, i.e. as $n\rightarrow\infty$.\vspace{0.5em}}
    Equation (\ref{eq:1.14}) ensures that local activities are used at maximal rates. In the context of the motivating application, this means customer demand is met by cars in the same region as much as possible.

	Following \citet{Harrison2000}, we call activity $j$ basic if $x_j^*>0$, whereas it is called nonbasic if $x_j^*=0$. We let $b$ denote the number basic activities. After possibly relabeling, we assume without loss of generality that activities $1,\dots, b$ are basic and that activities $b+1,\dots, J$ are nonbasic. Recall that the first $I$ of them are the local activities. As done in \citet{Harrison2000}, we partition the matrices $R$ and $A$ as follows:
	\begin{align}
	R=\left[H\quad K\right]\qquad\text{and}\qquad A=\left[B\quad N\right],\label{eq:30}
	\end{align}
	where $H,B\in\mathbb{R}^{I\times b}$ and $K, N\in\mathbb{R}^{I\times (J-b)}$. The submatrices $H$ and $B$ correspond to the basic activities of $R$ and $A$, respectively, while the submatrices $K$ and $N$ correspond to the nonbasic activities.
	
	In order to derive the approximating Brownian control problem, we consider a sequence of closely related systems indexed by the total number of jobs $n$. The formal limit of this sequence as $n\rightarrow\infty$ is the approximating Brownian control problem.
We attach a superscript of $n$ to quantities associated with the $n$th system in the sequence. To be specific, we define the scaled demand rate function $\Lambda^{n}:\mathcal{P}\rightarrow \mathbb{R}_+^I$  by 
	\begin{align}
	\Lambda^n(x) = n\Lambda(x),\quad x\in\mathcal{P}.\label{eq:31}
	\end{align}
	Then we define the set of admissible scaled demand rate vectors $\mathcal{L}^n$ as the following:
	\begin{align}
	\mathcal{L}^n = \left\{\lambda^n\in\mathbb{R}_+^I: \lambda^n = \Lambda^n(p)\text{ for some }p\in\mathcal{P}\right\}.\label{eq:32}
	\end{align}
	We note from Equations (\ref{eq:7})--(\ref{eq:8}) and (\ref{eq:31})--(\ref{eq:32}) that $\mathcal{L}^n= n\mathcal{L}$, and that $\Lambda^n$ has the inverse function $\left(\Lambda^n\right)^{-1}(x)=\left(\left(\Lambda_1^n\right)^{-1}(x_1),\dots, \left(\Lambda_I^n\right)^{-1}(x_I)\right)^{\prime}$ for $x\in \mathcal{L}^n$. We define the scaled revenue rate function $\pi^n$ as follows:
	\begin{align}
	\pi^n(x) = \sum_{i=1}^{I}x_i\left(\Lambda_i^n\right)^{-1}(x_i),\quad x\in\mathcal{L}^n.\label{eq:3.3}
	\end{align}
	Observing that $nx\in\mathcal{L}^n$ if and only if $x\in\mathcal{L}$, it can equivalently be shown that\footnote{The first equality in (\ref{eq:3.4*}) is proved by applying (\ref{eq:3.3}) and noting that $\left(\Lambda^n\right)^{-1}(nx) = \Lambda^{-1}(x)$ for $x\in\mathcal{L}$. The second equality in (\ref{eq:3.4*}) then follows by (\ref{eq:1.10}).\vspace{0.5em}}
	\begin{align}
	\pi^n(nx) = n \pi(x) = n\sum_{i=1}^{I}x_i\Lambda_i^{-1}(x_i),\quad x\in\mathcal{L}.\label{eq:3.4*}
	\end{align}
	Therefore, in the $n$th system, the revenue rate process is simply scaled by $n$. We also scale the holding cost rates $h^n$ and the idleness cost rates $c^n$ as follows:
	\begin{alignat}{2}
	h_i^n&= \frac{h_i}{\sqrt{n}},&&\quad i=0,1,\dots, I,\label{eq:3.5**}\\
	c_i^n &= \sqrt{n} c_i,&&\quad i=1,\dots, I.\label{eq:3.5*}
	\end{alignat}
    Lastly, we allow the mean travel time to vary with $n$ as follows: \begin{align}
        \eta^n &= \eta+\frac{\hat{\eta}}{\sqrt{n}},  \label{eq:TravelTime}
    \end{align}
    where $\hat{\eta}\in\mathbb{R}$.
	As observed in \citet{KoganLipster} and \citet{Ata_Volunteer2}, under our heavy traffic assumption we expect that the queue lengths at the buffers to be of order $\sqrt{n}$ and that the number of jobs in the infinite-server node be of order $n$. Therefore, we define the centered and scaled queue length processes as follows:
	\begin{alignat}{2}
	Z_0^n(t) &= \frac{1}{\sqrt{n}}\left(Q_0^n(t)-n\right)\quad\text{and}\quad
	Z_i^n(t)  = \frac{1}{\sqrt{n}}Q_i^n(t)\quad \text{for}\quad i=1,\dots, I,\quad t\ge 0.\label{eq:38}
	\end{alignat}
	Observe that since $\sum_{i=0}^{I}Q_i^n(t)=n$ for all $t\ge 0$, it follows that $\sum_{i=0}^{I}Z_i^n(t)=0$ for all $t\ge 0$. 
	
	As argued in \citet{Harrison1988} (see also \citet{Harrison2000, Harrison2003}), any policy worthy of consideration satisfies $T^n(t)\approx x^*t$, for all $t\ge 0$ and large $n$. That is, the nominal allocation rate $x^*$ given in Assumption \ref{ass:3} should give a first-order approximation to the allocation rates of the various activities under policy $T^n$. However, the system manager can choose the second-order, i.e., order $1/\sqrt{n}$, deviations from that. In order to capture such deviations from the nominal rates, we define the centered and scaled processes as follows:
	\begin{align}
	Y_j^n(t) =\sqrt{n}\left(x_j^*t - T_j^n(t)\right), \quad j = 1,\dots, J,\quad t\ge 0,\label{eq:39}
	\end{align}
	Similarly, in the heavy traffic regime, we expect the servers to be always busy to a first-order approximation, but they may incur idleness on the second order, i.e., order $1/\sqrt{n}$. As such, we define the scaled idleness processes as follows:
	\begin{align}
	U_i^n(t)&=\sqrt{n} I_i^n(t),\quad i=1,\dots, I,\quad t\ge 0.\label{eq:40}
	\end{align}
	Then, it follows from Equations (\ref{eq:2.6}) and (\ref{eq:1.15}) that
	\begin{align}
	U_i^n(t)&=\sum_{j\in\mathcal{A}_i}Y_j^n(t),\quad i=1,\dots, I,\quad t\ge 0.\label{eq:41}
	\end{align}
	In addition, we define the centered and scaled demand and service rate processes, respectively, as follows:
	\begin{alignat}{3}
	\zeta_i^n(t) & =\frac{1}{\sqrt{n}} \left(\lambda_i^n(t) - n\lambda_i^*\right), &&\quad i=1,\dots, I,&&\quad t\ge 0,\label{eq:41*}\\
	\kappa_j^n(t) &= \frac{1}{\sqrt{n}} \left(\mu_j^n(t) - n\mu_j^*\right), &&\quad j=1,\dots, J,&&\quad t\ge 0.\label{eq:42}
	\end{alignat}
	Note that by Equation (\ref{eq:9}) we have $\kappa_j^n(\cdot)= \zeta_i^n(\cdot)$ for each $j\in \mathcal{A}_i$. Finally, we define the centered cumulative profit function. To do so, we first introduce the auxiliary function $\tilde{V}^n$ that will serve as the centering function. To be specific, we define
	\begin{align}
	\tilde{V}^n(t) = n\left[\pi\left(\lambda^*\right)-h_0^n\right]t = n\pi\left(\lambda^*\right)t - \sqrt{n}h_0 t,\quad t\ge 0,\label{eq:3.18*}
	\end{align}
	where the second equality follows from the definition of $h_0^n$, see Equation (\ref{eq:3.5**}). Note that $\tilde{V}^n(t)$ does not depend on the system manager's control. Therefore, instead of maximizing the average profit, she can focus on minimizing the average cost, where the cumulative cost up to time $t$, denoted by $\hat{V}^n(t)$, is defined as follows:
	\begin{align}
	\hat{V}^n(t) = \tilde{V}^n(t) - V^n(t),\quad t\ge 0.\label{eq:44}
	\end{align}
	We then proceed with replacing the processes $Z^n$, $Y^n$, $U^n$, $\zeta^n$, $\kappa^n$, and $\hat{V}^n$ with their formal limits $Z$, $Y$, $U$, $\zeta$, $\kappa$, and $\xi$, respectively, as $n\rightarrow\infty$. In particular, the cost process $\xi$ in the Brownian approximation is given by
	\begin{align}
	\xi(t) = \int_{0}^{t}\left(\sum_{i=1}^I \alpha_i\zeta_i^2(s) + \sum_{i=0}^{I}h_iZ_i(s)\right)\,ds +c^{\prime}U(t),\quad t\ge 0,\label{eq:45}
	\end{align} 
	where $\alpha_i = -\left(\Lambda_i^{-1}\right)^{\prime}(\lambda_i^*) -(\lambda^*_i/2)\times \left(\Lambda_i^{-1}\right)^{\prime\prime}(\lambda_i^*)>0$ for $i=1,\dots, I$. The steps outlining the formal derivation of the Brownian Control Problem and of Equation (\ref{eq:45}) are given in Appendix \ref{app:A}.
	
	The Brownian control problem (BCP) is given as follows: Choose processes $Y=\left(Y_j\right)$ and $\zeta =\left(\zeta_i\right)$ that are nonanticipating with respect to $B$ so as to
	\begin{align}
	& \text{minimize}\quad \limsup_{t\rightarrow \infty}\,\,\frac{1}{t}\,\,E\left[\xi(t)\right]\label{eq:3.16}\\
	& \text{subject to}\notag\\
	& Z_i(t) = B_i(t) - q_i\eta\int_{0}^{t}\sum_{i=1}^{I}Z_i(s)\,ds-\sum_{j\in\mathcal{C}_i}\int_{0}^{t}x_j^* \kappa_j(s)\,ds + \sum_{j\in\mathcal{C}_i}\mu_j^* Y_j(t), \quad i=1,\dots, I,\quad t\ge 0,\label{eq:3.17}\\
	& Z_0(t)= - \sum_{i=1}^{I}Z_i(t),\quad t\ge 0,\label{eq:3.18}\\
	& U(t) = AY(t),\quad t\ge 0,\label{eq:3.20}\\
	& \kappa_j(t) = \zeta_i(t) \quad \text{for} \quad j\in\mathcal{A}_i,\quad i=1,\dots, I,\quad t\ge 0,\label{eq:3.22}\\
	& Z_i(t)\ge 0,\quad i=1,\dots, I,\quad t\ge 0,\label{eq:3.19}\\
	& \text{$U$ is nondecreasing with $U(0)=0$,}\label{eq:3.21}
	\end{align}
	where $B=\left\{B(t),\,t\ge 0\right\}$ is an $I$-dimensional Brownian motion with starting state $B(0)\ge 0$ that has drift rate vector $\gamma=(\gamma_i)$ where $\gamma_i=\hat{\eta}q_i$ and covariance matrix $\Sigma$ given by
	\begin{alignat}{2}
	\Sigma_{ii} = q_i \eta + \sum_{j\in\mathcal{C}_i} \mu_j^* x_j^*\qquad \text{and}\qquad \Sigma_{ii^{\prime}} = q_i q_{i^{\prime}}\eta\qquad\text{for}\qquad  i,i^{\prime}=1,\dots, I,\quad i\ne i^{\prime}.  \label{eq:DriftCovariance}
	\end{alignat}Although the BCP (\ref{eq:3.17})--(\ref{eq:3.22}) is simpler than the original control problem that it approximates, it is not easy to solve because it is a multidimensional stochastic control problem. Thus, we further simplify it in Section \ref{sec:5} and derive an equivalent workload fomulation that is one-dimensional under the complete resource pooling condition which we solve analytically in Section \ref{sec:6}.

	\section{Equivalent Workload Formulation}\label{sec:5} 
	As a preliminary to the derivation of the workload problem, letting $Z=(Z_1,\ldots,Z_I)'$ and using Equation (\ref{eq:1.7}), we first rewrite Equation (\ref{eq:3.17})
	in vector form as follows:
	\begin{align}
	Z(t) &= B(t) -\eta q\int_{0}^{t}e^{\prime} Z(s)\,ds -C\,\text{diag}(x^*)\int_{0}^{t}\kappa(s)\,ds + RY(t),\quad t\ge 0,  \label{eq:EWF}
	\end{align}
	where $e$ is an $I$-dimensional vector of ones and $\text{diag}(x^*)$ is the $J\times J$ diagonal matrix whose $(j,j)$th element is $x^*_j$. 
 
	Motivated by the development in \citet{HarrisonVanMieghem1997} and \citet{Harrison2000}, we define the space of reversible displacements as follows:
	\begin{align}
	\mathcal{N}= \left\{Hy_B: By_B=0,\, y_B\in\mathbb{R}^b\right\},
	\end{align}
	where $y_{B}\in\mathbb{R}^{b}$ is the vector consisting of the components of $y$ indexed by the basic activities $j=1,\dots, b$. We let $\mathcal{M}= \mathcal{N}^{\perp}$ be the orthogonal complement of the space $\mathcal{N}$ and call $d=\dim(\mathcal{M})$ the workload dimension.
	Any $d\times I$ matrix $M$ whose rows form a basis for $\mathcal{M}$ is called a workload matrix. 
	Lemma \ref{lem:2} provides a canonical choice of the workload matrix $M$ based on the notion of communicating buffers, which is defined next, see \citet{AtaBarjestehKumar_SPN}. Also see \citet{HarrisonLopez1999} for a related definition of communicating servers.
	\begin{defn}
		Buffers $i$ and $i^{\prime}$ are said to \textit{communicate directly} if there exist basic activities $j$ and $j^{\prime}$ such that $i=b(j)$, $i^{\prime}=b(j^{\prime})$, and $s(j) = s(j^{\prime})$. That is, buffers $i$ and $i^{\prime}$ are served by a common server using basic activities. Buffers $i$ and $i^{\prime}$ are said to \textit{communicate} if there exist buffers $i_1,\dots, i_l$ such that $i_1=i$, $i_l=i^{\prime}$, and buffer $i_s$ communicates directly with buffer $i_{s+1}$ for $s=1,\dots, l-1$.
	\end{defn}
	Buffer communication is an equivalence relation.
 Thus, the set of buffers can be partitioned into $L$ disjoint subsets where all buffers in the same subset communicate with each other. We call each subset a buffer pool and denote the $l$th buffer pool by $\mathcal{P}_l$, $l=1,\dots, L$. Associated with each buffer pool is a server pool. The $l$th server pool $\mathcal{S}_l$ is defined as follows: 
	\begin{align}
		\mathcal{S}_l = \left\{k: \exists j\in\left\{1,\dots, b\right\}\text{ s.t. }s(j)=k\text{ and }b(j)\in\mathcal{P}_l\right\},\quad l=1,\dots, L.\label{eq:4.5}
	\end{align}
	In words, server pool $l$ consists of all servers that can serve a buffer in buffer pool $l$ using a basic activity. Note that since the buffer pools partition the buffers, it follows from Equation (\ref{eq:4.5}) that the server pools partition the servers. Thus, the buffer pools and the server pools are in a one-to-one correspondence. As a result, there is an equivalent notion of server communication, but we stick with the definition of buffer communication for mathematical convenience. The following lemma characterizes the workload dimension and the workload matrix, see Appendix \ref{app:B} for its proof.
	
	\begin{lem}\label{lem:2}
		The workload dimension equals the number of buffer pools, i.e., $d=L$. Furthermore, the $L\times I$ matrix $M$ given by 
		\begin{align}
		M_{li}=\Bigg\{
		\begin{array}{ll}
		1,&\text{if $i\in\mathcal{P}_l$},\\
		0,&\text{otherwise,}
		\label{eq:4.6*}\end{array}
		\end{align}
		for $l=1,\dots, L$ and $i=1,\dots, I$ constitutes a canonical workload matrix.
	\end{lem}
	To facilitate the derivation of the workload state dynamics, we define the $L\times I$ matrix $G$ as follows:
	\begin{align}
	G_{lk} = \lambda_{k}^*\, \mathbf{1}_{\left\{k\,\in\,\mathcal{S}_l\right\}},\quad l=1,\dots, L,\quad k=1,\dots, I.\label{eq:4.7*}
	\end{align}
	That is, the $l$th row of $G$, $\left(G_{l1},\dots, G_{lI}\right)$ contains the nominal service rates for those servers in server pool $l$ and zeros for the rest of the servers. The next lemma provides a useful result that helps us derive the workload problem. It is proved in Appendix \ref{app:B}.
	\begin{lem}\label{lem:3}
		We have that $MR=GA$.
	\end{lem}
	
	We define the $L$-dimensional workload process $W=\left\{W(t),\,t\ge 0\right\}$ as 
	\begin{align}
	W(t) = MZ(t),\quad t\ge 0,\label{eq:4.7}
	\end{align}
	whose $l$th component represents the total number of jobs for the $l$th server pool at time $t$ for $l=1,\dots, L$. By Equation (\ref{eq:4.7}) and Lemma \ref{lem:3}, we arrive at the following equation which describes the evolution of the workload process:
	\begin{align}
	W(t) = \chi(t) -M\eta q\int_{0}^{t}e^{\prime}Z(s)\,ds - MC\text{diag}(x^*)\int_{0}^{t}\kappa(s)\,ds + GU(t),\quad t\ge 0,\label{eq:4.8}
	\end{align}
	where $\chi(t)= MB(t)$, so that $\chi=\left\{\chi(t),\,t\ge 0\right\}$ is a $L$-dimensional Brownian motion with drift vector $M\gamma$, covariance matrix $M\Sigma M^{\prime}$, and starting state $\chi(0)=MB(0)\ge 0$. 
	
	Next, we introduce a closely related control problem referred to as the reduced Brownian control problem (RBCP). Its state descriptor is the workload process $W$. To be more specific, the RBCP involves choosing a policy $(Z, U, \zeta)$ that is nonanticipating with respect to $\chi$ so as to
	\begin{align}
	&\text{minimize}\quad \limsup_{t\rightarrow \infty}\,\,\frac{1}{t}\,\,E\left[\int_{0}^{t}\left(\sum_{i=1}^I \alpha_i \zeta_i^2(s) + \sum_{i=1}^{I}(h_i-h_0)Z_i(s)\right)\,ds +c^{\prime}U(t)\right]\label{eq:4.9}\\
	&\text{subject to} \notag\\
	& W(t) = MZ(t),\quad t\ge 0, \label{eq:RBCP_ReduceW} \\
	& W(t) = \chi(t) -M\eta q\int_{0}^{t}e^{\prime}Z(s)\,ds -MC\text{diag}(x^*)\int_{0}^{t}\kappa(s)\,ds + GU(t),\quad t\ge 0, \label{eq:RBCP_W} \\
	& Z(t) \ge 0\text{ for } t\ge 0, \label{eq:RBCP_ZPositive}\\
	& \text{$U$ is nondecreasing with $U(0)=0$},  \label{eq:RBCP_UNondecreasing} \\
	&\kappa(t) = A^{\prime}\zeta(t)\text{ for } t\ge 0.\label{eq:4.14}
	\end{align}
	The BCP (\ref{eq:3.16})--(\ref{eq:3.22}) and the RBCP (\ref{eq:4.9})--(\ref{eq:4.14}) are equivalent as shown by the next proposition, see Appendix \ref{app:B} for its proof.
	\begin{prop}\label{prop:1}
		Every admissible policy $(Y,\zeta)$ for the BCP (\ref{eq:3.16})--(\ref{eq:3.22}) yields an admissible policy $(Z,U,\zeta)$ for the RBCP (\ref{eq:4.9})--(\ref{eq:4.14}) and these two policies have the same cost. On the other hand, for every admissible policy $(Z,U,\zeta)$ of the RBCP, there exists an admissible policy $(Y,\zeta)$ for the BCP whose cost is equal to that of the policy $(Z,U,\zeta)$ for the RBCP.
	\end{prop}
	
	Hereafter, we make the complete resource pooling assumption that corresponds to having a single resource pool in our context, see Assumption \ref{ass:4} below. \citet{HarrisonLopez1999} observes that the complete resource pooling assumption leads to a one-dimensional workload formulation, also see \citet{AtaKumar05}. Similarly, Assumption \ref{ass:4} allows us to formulate a one-dimensional workload formulation that is equivalent to the RBCP formulated in Equations (\ref{eq:4.9})--(\ref{eq:4.14}).
	
	\begin{assumption}\label{ass:4}
		All buffers communicate under the nominal processing plan, i.e., $L=1$.
	\end{assumption}
	
	 This assumption says that servers have sufficiently overlapping capabilities under the nominal processing plan; see \citet{HarrisonLopez1999} for further details. The following lemma allows us to simplify the RBCP under Assumption \ref{ass:4}, see Appendix \ref{app:B} for its proof.
	
	\begin{lem}\label{lem:4}
		Under Assumption \ref{ass:4}, we have $M=e^{\prime}$ and $G=\left(\lambda^*\right)^{\prime}$. Moreover, we have that
		\begin{align}
		M\eta q =\eta\qquad \text{and}\qquad 
		MC \text{diag}(x^*) A^{\prime} =e^{\prime}.\label{eq:5.1*}
		\end{align}
	\end{lem}
	Using Lemma \ref{lem:4}, the RBCP can be equivalently written as follows: Choose a policy $(Z, U,\zeta)$ that is nonanticipating with respect to $\chi$ so as to 
	\begin{align}
	&\text{minimize}\quad \limsup_{t\rightarrow \infty}\,\,\frac{1}{t}\,\,E\left[\int_{0}^{t}\left(\sum_{i=1}^I \alpha_i \zeta_i^2(s) + \sum_{i=1}^{I}(h_i-h_0)Z_i(s)\right)\,ds +c^{\prime}U(t)\right]\label{eq:5.2*}\\
	&\text{subject to} \notag\\
	&W(t) = \sum_{i=1}^{I}Z_i(t),\quad t\ge 0,\\
	&W(t) = \chi(t) -\eta\int_{0}^{t}W(s)\,ds  -\int_{0}^{t}\sum_{i=1}^{I}\zeta_i(s)\,ds + \sum_{i=1}^{I}\lambda_i^* U_i(t),\quad t\ge 0,  \label{eq:EWF_W} \\
	&Z(t) \ge 0\text{ for } t\ge 0,\\
	&U\text{ is nondecreasing with }U(0)=0,\label{eq:5.6*}
	\end{align}
	where $\chi$ is a one-dimensional Brownian motion with drift rate parameter $a=e'\gamma$ and variance parameter $\sigma^2 = e^{\prime}\Sigma e$ and starting state $\chi(0) = \sum_{i=1}^{I}B_i(0)\ge 0$.
	
	To further simplify the RBCP, we define the cost function $c$ by
	\begin{alignat}{2}
	c(x) &= \min\left\{\sum_{i=1}^I \alpha_i\zeta_i^2: e^{\prime}\zeta = x,\,\zeta\in\mathbb{R}^I\right\},\quad x\in\mathbb{R},\label{eq:5.9*}
	\end{alignat}
	and the optimal (state-dependent) drift rate function $\zeta^*$ by
	\begin{align}
	\zeta^*(x) = \text{argmin}\left\{ \sum_{i=1}^I \alpha_i\zeta_i^2: e^{\prime}\zeta = x,\,\zeta\in\mathbb{R}^I\right\},\quad x\in\mathbb{R}.\label{eq:5.11*}
	\end{align}
	Defining $\hat{\alpha}=\sum_{i=1}^I 1/\alpha_i$, the following lemma characterizes these functions---similar results are found in \citet{CelikMaglaras2008} and \citet{AtaBarjesteh2020}.
	\begin{lem}\label{lem:5}
	   We have that $c(x)=\frac{1}{\hat{\alpha}}x^2$ and $\zeta_i^\ast(x)=\frac{1}{\alpha_i\hat{\alpha}}x$ for $i=1,\ldots,I$ and $x\in\mathbb{R}$.
	\end{lem}
	In the workload formulation, it is optimal to keep all workload in the buffer with the lowest holding cost, i.e., buffer $i^*$ where 
	\begin{align}
	i^*= \underset{i=1,\dots, I}{\text{arg\,min}}\,\,h_i,   \label{eq:LowestHoldingCostBuffIndex}
	\end{align} 
	with holding cost $h = h_{i^*} - h_0>0$.
	Moreover, the system manager will only idle the server that is cheapest to idle, i.e., server $k^*$ where
	\begin{align}
	k^*=  \underset{i=1,\dots, I}{\text{arg\,min}} \,\,\frac{c_i}{\lambda_i^*},\label{eq:76}
	\end{align}
	with idling cost $r = c_{k^*}/\lambda^*_{k^*}$.
	
	The workload formulation can now be stated as follows: Choose a policy $\theta:[0,\infty)\rightarrow \mathbb{R}$
	that is nonanticipating with respect to $\chi$ so as to
	\begin{align}
	&\text{minimize}\quad\limsup_{t\rightarrow \infty}\,\,\frac{1}{t}\,\,E\left[\int_{0}^{t}\left[c\left(\theta(s)\right) + h\,W(s)\right]\,ds +r L(t)\right]\label{eq:5.12*}\\
	&\text{subject to} \notag\\
	&W(t) = \chi(t) -\eta\int_{0}^{t}W(s)\,ds  -\int_{0}^{t}\theta(s)\,ds + L(t),\quad t\ge 0,\label{eq:5.13}\\
	& W(t)\ge 0\text{ for } t\ge 0,\label{eq:5.14}\\
	&  L \text{ is nondecreasing with }L(0)=0,\label{eq:5.15*}
	\end{align}
 The RBCP (\ref{eq:4.9})--(\ref{eq:4.14}) and the EWF (\ref{eq:5.12*})--(\ref{eq:5.15*}) are equivalent as proved by the following proposition, see Appendix \ref{app:B} for its proof.
	\begin{prop}\label{prop:2}
		Every admissible policy $\theta$ for the EWF (\ref{eq:5.12*})--(\ref{eq:5.15*}) yields an admissible policy $(Z,U,\zeta)$ for the RBCP (\ref{eq:5.2*})--(\ref{eq:5.6*}) and these two policies have the same cost. On the other hand, for every admissible policy $(Z,U,\zeta)$ of the RBCP, there exists an admissible policy $\theta$ for the EWF whose cost is less than or equal to that of the policy $(Z,U,\zeta)$ for the RBCP.
	\end{prop}
In what follows, we add two additional constraints to the equivalent workload formulation. First, we require that \begin{align}
    \int_0^{\infty} \mathbf{1}_{\left\{W(t)>0\right\}}\,dL(t) &= 0, \label{eq:ProcessL}
\end{align} which requires that the process $L$ can increase only when $W=0$. That is, the control policy must be work conserving. 
 We include this restriction because its optimality is intuitive from the cost structure, i.e., there are both holding and idleness costs, and that the workload process is one dimensional. 
Second, we impose the following regularity condition: \begin{align*}
	\lim_{t\rightarrow \infty}\frac{ E\left[W(t)\right]}{t}=0.
	\end{align*} To repeat, we further require a policy $\theta$ to satisfy these conditions to be admissible.

	\section{Solving the Equivalent Workload Formulation}\label{sec:6}
	This section solves the EWF (\ref{eq:5.12*})--(\ref{eq:5.15*}). In order to minimize technical complexity, we restrict attention to stationary Markov policies. That is, the drift chosen at time $t$ will be a function of the current workload only, and so we write it as $\theta\left(W(t)\right)$. To facilitate the analysis, we next consider the Bellman equation for the workload formulation 
	which is the following second-order nonlinear differential equation: Find a function $f\in\mathcal{C}^2[0,\infty)$ and a constant $\beta\in\mathbb{R}$ satisfying
	\begin{align}
	\beta &= \min_{x\in\mathbb{R}}\left\{\frac{1}{2}\sigma^2 f^{\prime\prime}(w)-\eta wf^{\prime}(w) - xf^{\prime}(w) +af^{\prime}(w) + c(x) + hw\right\}\notag\\
	&= \min_{x\in\mathbb{R}}\left\{\frac{1}{ \hat{\alpha}}x^2 - xf^{\prime}(w)\right\} + \frac{1}{2}\sigma^2f^{\prime\prime}(w)-\eta wf^{\prime}(w) +af^{\prime}(w) +hw,\quad w\ge 0,\label{eq:6.1*}
	\end{align}  
	subject to the boundary conditions
	\begin{align}
	f^{\prime}(0) = -r\quad \text{and}\quad f^{\prime}\text{ is increasing with }\lim\limits_{w\rightarrow\infty}f^{\prime}(w)=\frac{h}{\eta}.\label{eq:6.2*}
	\end{align}
	The optimization problem on the right hand side of Equation (\ref{eq:6.1*}) is convex. Therefore, its solution is easily seen to be 
	\begin{align}
	x^* = \frac{\hat{\alpha}}{2}f^{\prime}(w).\label{eq:6.3*}
	\end{align}
	The Bellman equation can then be simplified as follows: Find a function $f\in\mathcal{C}^2[0,\infty)$ and a constant $\beta\in\mathbb{R}$ satisfying
	\begin{align}
	\beta = -\frac{\hat{\alpha}}{4}\left[f^{\prime}(y)\right]^2 + \frac{1}{2}\sigma^2f^{\prime\prime}(y)-\eta yf^{\prime}(y)+af^{\prime}(y) +hy,\quad  y\ge 0,
	\end{align}
	subject to the boundary conditions
	\begin{align}
	f^{\prime}(0) = -r\quad \text{and}\quad f^{\prime}\text{ is increasing with }\lim\limits_{w\rightarrow\infty}f^{\prime}(w)=\frac{h}{\eta}.  \label{eq:fBoundary}
	\end{align}
	Setting $v= f^{\prime}$, the Bellman equation can be written as follows: find a function $v\in \mathcal{C}^{1}[0,\infty)$ and a constant $\beta \in \mathbb{R}$ satisfying
	\begin{align}
	\beta = -\frac{\hat{\alpha}}{4}v^2(y) + \frac{1}{2}\sigma^2v^{\prime}(y)-\eta yv(y)+av(y)+hy,\quad y\ge 0,\label{eq:6.5*}
	\end{align}
	subject to the boundary conditions
	\begin{align}
	v(0) = -r\quad\text{and}\quad v\text{ is increasing with }\lim\limits_{y\rightarrow\infty}v(y) =\frac{h}{\eta}.\label{eq:6.6*}
	\end{align}
	This expresses the Bellman equation as a first-order differential equation. The following theorem provides its solution. Its proof is given at the end of Section \ref{sec:7}.
	\begin{thm}\label{thm:1}
		The Bellman equation (\ref{eq:6.5*})--(\ref{eq:6.6*}) has a solution $\left(\beta^*, v\right)$ with $\beta^*>0$.
	\end{thm}
	With $\beta^*>0$ and $v$ given by Theorem \ref{thm:1}, we define 
	\begin{align*}
	f(y) = \int_{0}^{y}v(x)\,dx,\quad y\ge 0.
	\end{align*} 
	The next result is immediate from Theorem \ref{thm:1} and provides a solution to the original Bellman equation:
	\begin{cor}
		The pair $(\beta^*, f)$ solves the Bellman equation (\ref{eq:6.1*})--(\ref{eq:6.2*}). 
	\end{cor}
	Define the following candidate policy $\theta^*:[0,\infty)\rightarrow \mathbb{R}$ by
	\begin{align}
	\theta^*(w) = \frac{\hat{\alpha}}{2}v(w),\quad w\ge 0.\label{eq:6.7}
	\end{align}
The following proposition facilitates the proof of our main result, Theorem \ref{thm:2}; see Appendix \ref{app:B} for its proof.
	\begin{prop}\label{prop:3}
		The candidate policy $\theta^*$ is admissible for the equivalent  workload formulation. That is, letting $W^*=\left\{W^*(t),\,t\ge 0\right\}$ denote the workload process under the candidate policy $\theta^*$, we have $$\lim_{t\rightarrow \infty}\frac{ E\left[W^*(t)\right]}{t}=0.$$
	\end{prop}
	The following result establishes that the candidate policy is optimal:
	\begin{thm}\label{thm:2}
		The candidate policy $\theta^*$ is optimal for the equivalent workload formulation (\ref{eq:5.12*})--(\ref{eq:5.15*}), and its long-run average cost is $\beta^*$.
	\end{thm}

	Next, we state an auxiliary lemma used in the proof of Theorem \ref{thm:2}.

	\begin{lem}\label{lem:6}
		Let $W$ be the workload process defined by (\ref{eq:5.13})--(\ref{eq:5.15*}) under an arbitrary admissible policy. Then the following hold:
		\begin{enumerate}
			\setlength{\itemsep}{0.5em}
			\item[(i)] $E{\displaystyle \int_{0}^{t} }f^{\prime}(W(s))\,d\chi(s) = 0$,\quad $t\ge 0$,
			\item[(ii)] $\underset{t\rightarrow \infty}{\limsup}\,\,\,{\displaystyle\frac{E\left[f\left(W(t)\right)\right]}{t}} = 0$.
		\end{enumerate}
	\end{lem}
	\begin{proof}
		By Proposition 4.7 in \citet{Harrison2013}, to prove part (i) it suffices to show that
		\begin{align*}
		E\int_{0}^{t}\left[f^{\prime}\left(W(s)\right)\right]^2\,ds<\infty \quad\text{for each}\quad t\ge 0.
		\end{align*}
		Because $f^{\prime}(w) \in [-r, h/\eta]$ for all $w\ge 0$ by Equation (\ref{eq:6.2*}) and because $W(t)\ge 0$ for all $t\ge 0$ by Equation (\ref{eq:5.14}), it follows that
		\begin{align*}
		E\int_{0}^{t}\left[f^{\prime}\left(W(s)\right)\right]^2\,ds\le t\left(r+\frac{h}{\eta}\right)^2<\infty,\quad \text{for}\quad t\ge 0,
		\end{align*}
		proving part (i). 
		
		In order to prove part (ii), note that it suffices to show that 
		\begin{align*}
		\limsup_{t\rightarrow \infty}\frac{\left\vert E\left[f\left(W(t)\right)\right]\right\vert}{t}= 0.
		\end{align*}
		We also note that 
		\begin{align*}
		\left\vert E\left[f\left(W(t)\right)\right]\right\vert & \le E\left\vert f\left(W(t)\right)\right\vert = E\left\vert \int_{0}^{W(t)}f^{\prime}(s)\,ds\right\vert \le E\int_{0}^{W(t)}\left\vert f^{\prime}(s)\right\vert\,ds \le \left(r+\frac{h}{\eta}\right)E\left[W(t)\right].
		\end{align*}
		Thus, by definition of an admissible policy, it follows that
		\begin{align*}
		\limsup_{t\rightarrow \infty}\frac{\left\vert E\left[f\left(W(t)\right)\right]\right\vert}{t} \le \left(r+\frac{h}{\eta}\right)\limsup_{t\rightarrow \infty}\frac{ E\left[W(t)\right]}{t}=0,
		\end{align*}
		proving part (ii).
	\end{proof}

We conclude this section with a proof of Theorem \ref{thm:2}.

	\begin{proof}[Proof of Theorem \ref{thm:2}]
		By Equation (\ref{eq:5.13}), note that for an admissible policy $\theta$,
		\begin{align}
		dW(s) = d\chi(s) -\eta W(s)\,ds - \theta(W(s))\,ds+dL(s).\label{eq:7.21}
		\end{align}
		Furthermore, since $L(s)$ is nondecreasing in $s$, the processes is a VF function almost surely; see Section B.2 in Harrison (2013). Therefore,
		\begin{align}
		\left[dW(s)\right]^2 &=\left[d\chi(s)\right]^2 + 2\,d\chi(s) \left[-\eta W(s)\,ds - \theta(W(s))\,ds+dL(s)\right] \notag\\&\qquad\qquad + \left[-\eta W(s)\,ds - \theta(W(s))\,ds+dL(s)\right]^2\label{eq:7.22}\\
		&=\sigma^2 \,ds.\notag
		\end{align}
		Note that the last two terms on the right hand side of Equation (\ref{eq:7.22}) are zero; see Chapter 4 in Harrison (2013). Then, for $f\in C^2[0,\infty)$, It\^{o}'s Lemma gives 
		\begin{align}
		df(W(s)) = f^{\prime}(W(s)) dW(s) + \frac{1}{2}f^{\prime\prime}(W(s)) \left[dW(s)\right]^2.\label{eq:7.23}
		\end{align}
		Define the differential operator $\Gamma_{\theta}:C^2[0,\infty)\rightarrow C[0,\infty)$ by
		\begin{align}
		\left(\Gamma_{\theta}f\right)(w) = \frac{1}{2}\sigma^2f^{\prime\prime}(w)- \left[\eta w + \theta(w) -a \right]f^{\prime}(w),\quad w\ge 0.\label{eq:7.24}
		\end{align}
		Then, combining Equations (\ref{eq:7.21})--(\ref{eq:7.24}) gives 
		\begin{align}
		df(W(s)) &= f^{\prime}(W(s))\,d\chi(s) + \Gamma_{\theta}f(W(s))\,ds+ f^{\prime}(W(s))\,dL(s).\label{eq:7.25}
		\end{align}
		Integrating both sides of Equation (\ref{eq:7.25}) over $[0,t]$ gives 
		\begin{align}
		f(W(t)) = f(W(0)) + \int_{0}^{t}f^{\prime}(W(s))\,d\chi(s) + \int_{0}^{t}\Gamma_{\theta}f(W(s))\,ds + \int_{0}^{t}f^{\prime}(W(s))\,dL(s).\label{eq:7.29}
		\end{align}
		Recall that by Equation (\ref{eq:ProcessL}) the process $L$ increases only when $W=0$. Thus, for $f\in C^2[0,\infty])$ satisfying $f^{\prime}(0)=-r$ we have
		\begin{align}
		\int_{0}^{t}f^{\prime}(W(s))\,dL(s) = f^{\prime}(0)L(t) = -rL(t).\label{eq:7.30}
		\end{align} 
		By Lemma \ref{lem:6} and Equations (\ref{eq:7.29})--(\ref{eq:7.30}), it follows that
		\begin{align}
		f(W(t)) = f(W(0))  + \int_{0}^{t}\Gamma_{\theta}f(W(s))\,ds -rL(t).\label{eq:7.31}
		\end{align}
		In particular, for the solution $(\beta^*, f)$ of the Bellman equation (\ref{eq:6.1*})--(\ref{eq:6.2*}) it follows that
		\begin{align}
		\beta^* - c\left(\theta(w)\right) - hw \le \frac{1}{2}\sigma^2 f^{\prime\prime}(w) -\left[\eta w + \theta(w) -a\right]f^{\prime}(w),\quad w\ge 0,\label{eq:7.32}
		\end{align}
		with equality holding when $\theta = \theta^*$. Therefore, by Equations (\ref{eq:7.24}) and (\ref{eq:7.31})--(\ref{eq:7.32}) we have
		\begin{align}
		f(W(t)) - f(W(0)) + rL(t)& = \int_{0}^{t}\Gamma_{\theta}f(W(s))\,ds\notag\\
		&\ge \int_{0}^{t}\left[\beta^* - c\left(\theta(W(s))\right) - hW(s)\right]\,ds,\label{eq:7.33}
		\end{align}
		with equality holding when $\theta = \theta^*$. Rearranging terms in Equation (\ref{eq:7.33}), taking expectations, and dividing by $t$ gives 
		\begin{align}
		\frac{1}{t}E\left[\int_{0}^{t}\left[c\left(\theta(s)\right)+hW(s)\right]\,ds  + rL(t)\right]\ge \beta^* - \frac{1}{t}Ef(W(t)) + \frac{1}{t}Ef(W(0)),\label{eq:7.34}
		\end{align}
		with equality holding when $\theta=\theta^*$. Finally, taking limits on both sides of Equation (\ref{eq:7.34}) and applying Lemma \ref{lem:6} gives 
		\begin{align*}
		\limsup_{t\rightarrow \infty}\frac{1}{t}E\left[\int_{0}^{t}\left[c\left(\theta(s)\right)+hW(s)\right]\,ds  + rL(t)\right]\ge \beta^*,
		\end{align*}
		with equality holding when $\theta= \theta^*$. Therefore, the policy $\theta^*$ is optimal for the equivalent workload formulation and its long-run average cost is $\beta^*$. 
	\end{proof}
	
	\section{Solution to the Bellman Equation}\label{sec:7}
	In this section we prove Theorem \ref{thm:1} by considering an initial value problem that is closely related to the Bellman equation. Namely, for each fixed $\beta\ge 0$ consider the following initial value problem, denoted by IVP($\beta$): Find a function $v\in C^1[0,\infty)$ such that 
	\begin{align}
	&\frac{\sigma^2}{2}v^{\prime}(y) = \beta +\frac{\hat{\alpha}}{4}v^2(y) + \eta y\left(v(y) - \frac{h}{\eta}\right) -av(y),\quad y\ge 0,\label{eq:1}\\
	&v(0)= -r.\label{eq:2}
	\end{align}
	
	The following result is standard and its proof is provided in Appendix \ref{app:B}.
	\begin{lem}\label{lem:vUnique}
		For $\beta\ge 0$, there exists a unique solution $v_{\beta}\in C^{1}[0,\infty)$  to (\ref{eq:1})-(\ref{eq:2}).
	\end{lem}
	For the remainder of this section, we analyze the (unique) solution to Equations (\ref{eq:1})--(\ref{eq:2}), focusing on how the behavior of the solution varies with the parameter $\beta$. Using this approach, we ultimately find a $\beta^*>0$, with corresponding solution $v_{\beta^*}$, such that the pair $(\beta^*, v_{\beta^*})$ solves the original Bellman equation. Namely, we look for $\beta^*$ such that $v_{\beta^*}$ satisfies the second condition in Equation (\ref{eq:fBoundary}) that $v_{\beta^*}$ is increasing with $\lim_{y\rightarrow\infty} v_{\beta^*}(y)=h/\eta$.
	

For much of our analysis, we consider parameters that satisfy one of the two cases, given in Assumption \ref{ass:ParamAssumption}. To state the assumption, let \[\underline{\beta}_1=0, \text{ and } \underline{\beta}_2=-ar-\frac{\hat{\alpha} r^2}{4}.\]
\begin{assumption}
    \label{ass:ParamAssumption} 
    One of the following holds:
    \begin{enumerate}[label=(\alph*)]
        \item Case 1: $a>-\frac{\hat{\alpha}}{4}r$ and $\beta\geq \underline{\beta}_1$; 
        \item Case 2: $a\leq -\frac{\hat{\alpha}}{4}r$ and $\beta>\underline{\beta}_2$.
    \end{enumerate}
\end{assumption}
\begin{remark}
    Note that under Assumption \ref{ass:ParamAssumption}(b), we have that $\underline{\beta}_2\geq 0$.
\end{remark}
Lemmas \ref{lem:vUpperBound}-\ref{lem:vConstant} facilitate the analysis to follow.
\begin{lem} \label{lem:vUpperBound}
    If $y>0$ is a local maximizer of $v_\beta(y)$, then $v_\beta(y)\leq h/y$. 
\end{lem}
\begin{proof}
    Because $y$ is a local maximizer, we have that $v_\beta^{\prime}(y)=0$, and $v_\beta''(y)\leq 0$. Differentiating both sides of Equation (\ref{eq:1}) and using $v_\beta'(y)=0$, we write \[\frac{\sigma^2}{2}v_\beta''(y)=\eta\left( v_\beta(y)-\frac{h}{\eta}\right)\leq 0,\] from which it follows that $v_\beta(y)\leq h/y$.
\end{proof}

	\begin{lem}\label{lem:vIncreaseSup}
		Under Assumption \ref{ass:ParamAssumption}, $v_{\beta}$ increases to its supremum.
	\end{lem}
	\begin{proof}
		First, note that $v_{\beta}^{\prime}(0) = \frac{2\beta}{\sigma^2} + \frac{\hat{\alpha}}{2\sigma^2}r^2 +\frac{2ar}{\sigma^2}>0$ in either case of Assumption \ref{ass:ParamAssumption}. Aiming for a contradiction, suppose $v_{\beta}$ does not increase to its maximum. Then we must have $0\le x_1< x_2<x_3$ such that 
		\begin{align*}
		&v_{\beta}(x_1)= v_{\beta}(x_2) = v_{\beta}(x_3)= v, \\
		&v_{\beta}^{\prime}(x_1)>0,\,\,\,\,v_{\beta}^{\prime}(x_2)<0,\,\,\,\, v_{\beta}^{\prime}(x_3)>0.
		\end{align*}
		In particular, we have the following equations:
		\begin{align}
		v_{\beta}^{\prime}(x_1) = \frac{2\beta}{\sigma^2} + \frac{\hat{\alpha}}{2\sigma^2}v^2 +\frac{2\eta}{\sigma^2} x_1\left(v - \frac{h}{\eta}\right) -\frac{2av}{\sigma^2}>0,\label{eq:125}\\
		v_{\beta}^{\prime}(x_2) = \frac{2\beta}{\sigma^2} + \frac{\hat{\alpha}}{2\sigma^2}v^2 +\frac{2\eta}{\sigma^2} x_2\left(v - \frac{h}{\eta}\right) -\frac{2av}{\sigma^2}<0,\label{eq:126}\\
		v_{\beta}^{\prime}(x_3) = \frac{2\beta}{\sigma^2} + \frac{\hat{\alpha}}{2\sigma^2}v^2 +\frac{2\eta}{\sigma^2} x_3\left(v - \frac{h}{\eta}\right) -\frac{2av}{\sigma^2}>0.\label{eq:127}
		\end{align}
		On the one hand, subtracting (\ref{eq:126}) from (\ref{eq:125}) yields
		\begin{align}
		\frac{2\eta}{\sigma^2}\left(x_1-x_2\right)\left(v-\frac{h}{\eta}\right)>0.\label{eq:128}
		\end{align}
		Because $x_1-x_2<0$, we conclude from (\ref{eq:128}) that 
		\begin{align}
		v-\frac{h}{\eta}<0\label{eq:129}
		\end{align}
		On the other hand, subtracting (\ref{eq:126}) from (\ref{eq:127}) gives
		\begin{align}
		\frac{2\eta}{\sigma^2}\left(x_3-x_2\right)\left(v-\frac{h}{\eta}\right)>0.\label{eq:130}
		\end{align}
		But, we deduce from Equation (\ref{eq:129}) and from $x_3-x_2>0$ that the left hand side of Equation (\ref{eq:130}) is negative, which is a contradiction. This completes the proof.
	\end{proof}

\begin{lem}\label{lem:vConstant}
		Let $0\leq x_1<x_2$. Under Assumption \ref{ass:ParamAssumption}, the following condition is necessary for $v_\beta(x)$ to be constant on $(x_1,x_2)$: \begin{align}  \label{eq:BetaConstant}
		    \beta=a\frac{h}{\eta}-\frac{\hat{\alpha}}{4}\left(\frac{h}{\eta}\right)^2.
		\end{align}
  Moreover, if $v_\beta$ is constant on $(x_1,x_2)$, then $v_\beta(x)=h/\eta$ for $x\in (x_1,x_2)$, and letting $\hat{x}={\rm inf} \{x\geq 0: v_\beta(x)=h/\eta\}$, it follows that $v_\beta$ is nondecreasing on $[0, \hat{x}]$ and stays constant at value $h/\eta$ thereafter.

  On the other hand, if Equation (\ref{eq:BetaConstant}) does not hold, then there is no interval on which $v_\beta$ is constant, i.e., the set $\{y\geq 0: v_\beta'(y)=0\}$ has Lebesgue measure zero. 
	\end{lem}
	\begin{proof}
		Suppose the condition in Equation (\ref{eq:BetaConstant}) is violated, which implies 
		\begin{align}  \label{eq:BetaConstantContradict}
		\beta + \frac{\hat{\alpha}}{4}\left(\frac{h}{\eta}\right)^2-a\frac{h}{\eta} \not=0
		\end{align}
        Aiming for a contraction, suppose there exist an interval $(x_1,x_2)$ such that $v_\beta(y)=v$ on it. This implies $v_\beta'(y)=v_\beta''(y)=0$ on $(x_1,x_2)$. Differentiating both sides of Equation (\ref{eq:1}) and using $v_\beta'(y)=0$ on $(x_1,x_2)$ gives  
		\begin{align*}
		\frac{\sigma^2}{2} v_\beta''(y)=\eta\left(v_\beta(y)-\frac{h}{\eta}\right), \quad y\in (x_1,x_2).
		\end{align*}
		Thus, $v_\beta(y)=h/\eta$ on $(x_1,x_2)$. Substituting this into Equation (\ref{eq:1}) yields  
		\begin{align*}
		\frac{\sigma^2}{2} v_\beta'(y) = \beta+\frac{\hat{\alpha}}{4} \left(\frac{h}{\eta}\right)^2-a\frac{h}{\eta}\not=0,\quad y\in (x_1,x_2),
		\end{align*}
		which follows from (\ref{eq:BetaConstantContradict}) and contradicts that $v_\beta'(y)=0$ on $(x_1,x_2)$. Therefore, if (\ref{eq:BetaConstant}) does not hold, then there is no interval on which $v_\beta$ is constant. 

    Now, we turn to the first part of the lemma. If $v_\beta$ is constant on $(x_1,x_2)$, then $v_\beta'(x)=v_\beta''(x)=0$ on $(x_1,x_2)$. As argued above, these imply $v_\beta(x)=h/\eta$ on $(x_1,x_2)$. In addition, it follows from (\ref{eq:1}) and $v_\beta'(x)=-h/\eta$ on $(x_1,x_2)$ that \[\beta+\frac{\hat{\alpha}}{4}\left(\frac{h}{\eta}\right)^2-a\frac{h}{\eta}=0,\] proving the necessary condition (\ref{eq:BetaConstantContradict}). Building on these, because at any local maximum $v_\beta(x)\leq h/\eta$ by Lemma \ref{lem:vUpperBound} and $\hat{x}$ is the first time $v_\beta$ reaches to its maximum by Lemma \ref{lem:vIncreaseSup}, we conclude that $v_\beta$ is nondecreasing on $[0, \hat{x}]$. To conclude the proof, consider an auxiliary IVP involving (\ref{eq:1}) on $[\hat{x}, \infty)$ with the initial condition $v(\hat{x})=h/\eta$. Then setting $v(x)=h/\eta$ solves it. Moreover, combining that with $v_\beta$ on $[0, \hat{x})$ constitutes a solution to the IVP (\ref{eq:1})-(\ref{eq:2}). By Lemma \ref{lem:vUnique}, this is the unique solution.    
	\end{proof}
To facilitate the analysis below, we define the following four sets. First, consider Case 1 identified in Assumption \ref{ass:ParamAssumption} (i.e., Assumption \ref{ass:ParamAssumption}(a)) and let \begin{align*}
    {\cal I}_1 &= \left\{\beta\geq 0: v_\beta \text{ is nondecreasing on } (0, \infty) \right\},   \\
    {\cal D}_1 &= \left\{\beta\geq 0: \exists x_\beta \geq 0\text{ such that $v_\beta$ is nondecreasing on $(0,x_\beta)$ and decreasing on $(x_\beta, \infty)$}\right\}.
\end{align*}
Similarly, in Case 2 of Assumption \ref{ass:ParamAssumption} (Assumption \ref{ass:ParamAssumption}(b)), we define \begin{align*}
    {\cal I}_2 &= \left\{\beta>\underline{\beta}_2: v_\beta \text{ is nondecreasing on } (0, \infty) \right\},   \\
    {\cal D}_2 &= \left\{\beta > \underline{\beta}_2: \exists x_\beta \geq 0\text{ such that $v_\beta$ is nondecreasing on $(0,x_\beta)$ and decreasing on $(x_\beta, \infty)$}\right\}.
\end{align*}
	\begin{lem}\label{lem:12}
		We have the following: \begin{enumerate}[label=(\roman*)]
		    \item Under Assumption \ref{ass:ParamAssumption}(a), $\beta\in{\cal D}_1$ if and only if $\exists x_0\in (0, \infty)$ such that $v_\beta'(x_0)<0$.
            \item Under Assumption \ref{ass:ParamAssumption}(b), $\beta\in {\cal D}_2$ if and only if $\exists x_0\in (0, \infty)$ such that $v_\beta'(x_0)<0$.
		\end{enumerate}
	\end{lem}
	\begin{proof}
		First, note from Lemma \ref{lem:vConstant} that it is necessary that $v_\beta$ increases to $h/\eta$ and stay constant thereafter for it to be constant on any interval. In that case, we would have $\beta\in {\cal I}_i$ ($i=1$ under Assumption \ref{ass:ParamAssumption}(a) and $i=2$ under Assumption \ref{ass:ParamAssumption}(b)). Thus, for the remainder of the proof, we assume there is no interval on which $v_\beta$ is constant. 
  
        We prove Cases (i) and (ii) simultaneously because their proofs are identical. For $i=1,2$, let $\beta\in {\cal D}_i$. Aiming for a contradiction, assume there does not exist $x_0>0$ such that $v_\beta'(x_0)<0$. Then, $v_\beta'(x)\geq 0$ for all $x\geq 0$, i.e., $v_\beta$ is nondecreasing on $(0,\infty)$. Thus, $\beta\in {\cal I}_i$, a contradiction. Therefore, there exists $x_0>0$ such that $v_\beta'(x_0)<0$.
        
        For the other direction, suppose there exists $x_0>0$ such that $v_\beta'(x_0)<0$. Because $v_\beta$ increases to its maximum (Lemma \ref{lem:vIncreaseSup}), it is not constant on any interval (by the argument given in the opening paragraph of this proof) and $v_\beta'(x_0)<0$, it achieves its maximum at some $x^\ast<x_0$. Thus, by Lemma \ref{lem:vUpperBound}, we have that
		\begin{align}
		v_{\beta}(x) \le v_{\beta}(x^*) \leq\frac{h}{\eta},\quad x\ge 0.\label{eq:134}
		\end{align}
		Aiming for a contradiction, suppose that  $\beta\not\in\mathcal{D}_i$. Then $v_{\beta}$ cannot be decreasing over $[x^*,\infty)$. Thus, there exist $x_1$ and $x_2$ such that
		\begin{align*}
		&x^*<x_1<x_2,\\
		&v= v_{\beta}(x_1)= v_{\beta}(x_2)\leq \frac{h}{\eta},\\
		&v_{\beta}^{\prime}(x_1)<0<v_{\beta}^{\prime}(x_2).
		\end{align*}
		In particular, the following holds:
		\begin{align}
		v_{\beta}^{\prime}(x_1) = \frac{2\beta}{\sigma^2} + \frac{\hat{\alpha}}{2\sigma^2}v^2 +\frac{2\eta}{\sigma^2} x_1\left(v - \frac{h}{\eta}\right)-av<0,\label{eq:135}\\
		v_{\beta}^{\prime}(x_2) = \frac{2\beta}{\sigma^2} + \frac{\hat{\alpha}}{2\sigma^2}v^2 +\frac{2\eta}{\sigma^2} x_2\left(v - \frac{h}{\eta}\right)-av>0.\label{eq:136}
		\end{align}
		Subtracting (\ref{eq:135}) from (\ref{eq:136}) gives 
		\begin{align*}
		0<v_{\beta}^{\prime}(x_2) - v_{\beta}^{\prime}(x_1) = \eta\left(x_2- x_1\right)\left(v-\frac{h}{\eta}\right)\leq 0,
		\end{align*}
		where the last inequality follows because $x_2-x_1>0$ and $v \leq \frac{h}{\eta}$ by Equation (\ref{eq:134}), leading to a contradiction. Thus, $\beta\in {\cal D}_i$.
	\end{proof}
	
	\begin{cor}\label{cor:SetsPartition}
		Under Assumption \ref{ass:ParamAssumption}, we have the following: \begin{enumerate}[label=(\roman*)]
		    \item In Case 1 of Assumption \ref{ass:ParamAssumption} (Assumption \ref{ass:ParamAssumption} (a)), the sets ${\cal I}_1$ and ${\cal D}_1$ partition $[0, \infty)$;  
            \item In Case 2 of Assumption \ref{ass:ParamAssumption} (Assumption \ref{ass:ParamAssumption}(b)), the sets ${\cal I}_2$ and ${\cal D}_2$ partition $(\underline{\beta}_2, \infty)$.
		\end{enumerate}
	\end{cor}
	\begin{proof}
        Consider Case (i).
		For $\beta\ge 0$, if $v_{\beta}^{\prime}(x)<0$ for some $x>0$, then $\beta\in \mathcal{D}_1$ by Lemma \ref{lem:12}. Otherwise, $v_{\beta}^{\prime}(x) \ge 0$ for all $x>0$, in which case $\beta\in\mathcal{I}_1$ by definition. Proof of (ii) follows similarly. 
	\end{proof}
	\begin{cor}  \label{cor:vMaximumAndSup}
	    Under Assumption \ref{ass:ParamAssumption}, we have the following. In Case $i$ of Assumption \ref{ass:ParamAssumption} for $i=1,2$, if $\beta\in {\cal D}_i$, then $v_\beta$ achieves its maximum and \[\sup_{x\geq 0} v_\beta(x) < \frac{h}{\eta}. \]  
	\end{cor}
    \begin{proof}
        For $i=1,2$, by definition of ${\cal D}_i$, $\exists x^\ast\geq 0$ such that $v_\beta$ is nondecreasing on $(0, x^\ast)$ and it is decreasing on $(x^\ast, \infty)$. First, note that if $x^\ast=0$, then $v_\beta(x)$ is decreasing everywhere and $v_\beta(0)=-r$ and the result follows. Thus, we assume $x^\ast>0$. Note that $v_\beta$ achieves its maximum at $x^\ast$. Also, we conclude from Lemma \ref{lem:vUpperBound} that $v_\beta(x^\ast)\leq h/\eta$. Aiming for a contradiction, suppose $v_\beta(x^\ast)=h/\eta$. Note that $v_\beta'(x^\ast)=0$ because $x^\ast$ is the maximizer. From these, by differentiating both sides of Equation (\ref{eq:1}), we conclude that $v_\beta''(x^\ast)=0$. Then we can argue as in the proof of Lemma \ref{lem:vConstant} that $v_\beta(x)=h/\eta$ for $x\geq x^\ast$, implying $\beta\not\in {\cal D}_i$, a contradiction. Thus, $v_\beta(x^\ast)\not=h/\eta$, completing the proof.     
    \end{proof}
	\begin{lem}\label{lem:13}
        Under Case $i$ of Assumption \ref{ass:ParamAssumption} ($i=1,2$), we have that if $\beta\in{\cal D}_i$, 
		then $\lim\limits_{x\rightarrow\infty} v_{\beta}(x) = -\infty$.
	\end{lem}
	\begin{proof}
		It follows from Corollary \ref{cor:vMaximumAndSup} that $v_\beta$ has a maximizer $x^\ast$ such that 
		\begin{align}
		v_{\beta}(x) \le v_{\beta}(x^*)<\frac{h}{\eta},\quad x\ge 0.\label{eq:137}
		\end{align}
		Also define 
		\begin{align}
		\ep = \frac{h}{\eta} - v_{\beta}(x^*)>0.\label{eq:138}
		\end{align}
		To prove $\lim\limits_{x\rightarrow\infty}v_{\beta}(x) = -\infty$, we argue by contradiction. To that end, suppose there exists a $K_1>0$ such that $v_{\beta}(x) \ge -K_1$ for $x\ge 0$. Then we have that 
		\begin{align}
		\vert v_{\beta}(x) \vert \le K_2= \max\left\{K_1, \frac{h}{\eta}\right\}.\label{eq:139}
		\end{align}
		Recalling IVP($\beta$), we bound $v_{\beta}^{\prime}(\cdot)$ using (\ref{eq:137})--(\ref{eq:139}) as follows:
		\begin{align}
		\frac{\sigma^2}{2}v_{\beta}^{\prime}(y)&\le \beta + \frac{\hat{\alpha}}{4}K_2^2 + \eta y \left(v_{\beta}(x^*)- \frac{h}{\eta}\right) +|a|K_2 = \left[\beta + \frac{\hat{\alpha}}{4}K_2^2 +|a|K_2 \right]-\ep\eta y,\quad y\ge 0.\label{eq:140}
		\end{align}
		Integrating both sides of (\ref{eq:140}) over $[0,y]$ and using the initial condition $v_{\beta}(0) = -r$ gives
		\begin{align}
		\frac{\sigma^2}{2}v_{\beta}^{\prime}(y)\le -\frac{\sigma^2}{2}r + \left[\beta + \frac{\hat{\alpha}}{4}K_2^2 +|a|K_2 \right]y-\frac{\eta\ep}{2}y^2,\quad y\ge 0.\label{eq:141}
		\end{align}
		Since $\eta\ep/2>0$, the right hand side of (\ref{eq:141}) tends to $-\infty$ as $y\rightarrow\infty$, implying that $v(y)\rightarrow-\infty$ as $y\rightarrow\infty$, a contradiction.
	\end{proof}
	
	\begin{lem}\label{lem:14}
		Under Case $i$ of Assumption \ref{ass:ParamAssumption} ($i=1,2$), the following are equivalent:
		\begin{itemize}
			\setlength{\itemsep}{0em}
			\item[(i)] $\beta\in\mathcal{D}_i$,
			\item[(ii)] $\exists x>0$ such that $v_{\beta}^{\prime}(x) <0$,
			\item[(iii)] $\exists x>0$ such that $v_{\beta}(x)<-r$,
			\item[(iv)] $\lim\limits_{x\rightarrow\infty}v_{\beta}(x) = -\infty$.
		\end{itemize}
	\end{lem}
	\begin{proof}
		Parts (i) and (ii) are equivalent by Lemma \ref{lem:12}. Part (i) implies (iv) by Lemma \ref{lem:13}. Clearly, (iv) implies (iii). Therefore, it suffices to prove that (iii) implies (ii). To that end, let $x_0>0$ be such that $v_{\beta}(x_0)<-r$. Since $v_{\beta}(0)=-r$, it follows from the mean value theorem that there exists a $\hat{x}_0\in (0,x_0)$ such that 
		\begin{align*}
		v^{\prime}_{\beta}(\hat{x}_0) = \frac{v_{\beta}(x_0) - v_{\beta}(0)}{x_0-0}<0,
		\end{align*}
		proving part (ii).
	\end{proof}
	
	\begin{lem}\label{lem:vLimit}
        Under Assumption \ref{ass:ParamAssumption}, we have that $\lim_{x\rightarrow \infty}v_{\beta}(x)=\infty$  if and only if there exists an $x_0> 0$ such that $v_{\beta}(x_0)\ge \frac{h}{\eta}$.
	\end{lem}
	\begin{proof}
		First, if $\lim_{x\rightarrow \infty}v_{\beta}(x)=\infty$, then clearly, there exists an $x_0> 0$ such that $v_{\beta}(x_0) > \frac{h}{\eta}$. To prove the other direction, suppose there exists $x_0>0$ such that $v_{\beta}(x_0)>\frac{h}{\eta}$, and define
		\begin{align*}
		x_1 = \inf\left\{x>0: v_{\beta}(x) \geq  \frac{h}{\eta} \right\},
		\end{align*}
		Because $v_\beta(0)=-r<\frac{h}{\eta}<v_\beta(x_0)$, by the intermediate value theorem, $v_\beta(x_1)=\frac{h}{\eta}$. Next, we argue that $v_{\beta}(x) > \frac{h}{\eta}$ for all $x > x_1$. If not, then there exists an $x_2>x_1$ such that $v_{\beta}(x_2) \leq \frac{h}{\eta}$. Then let 
		\begin{align*}
		x_3 = \inf\left\{x>x_1: v_{\beta}(x) \le \frac{h}{\eta}\right\}.
		\end{align*}
		Note that $v_{\beta}(x_3)= \frac{h}{\eta}$ by continuity of $v_{\beta}$. Furthermore, note that $x_3>x_1$ since $v_{\beta}(x_1) = \frac{h}{\eta}$ and \begin{align} \label{eq:vLimitLem_vDerivativePositive}
		    v_{\beta}^{\prime}(x_1) = \frac{2\beta}{\sigma^2} + \frac{\hat{\alpha}}{2\sigma^2} \left(\frac{h}{\eta}\right)^2-a\frac{h}{\eta}>0,
		\end{align} where the last inequality holds because \begin{enumerate}[label=(\roman*)]
		    \item $v_\beta'(x_1)\geq 0$ by definition of $x_1$, 
            \item $x_1<x_0, v_\beta(x_0)>\frac{h}{\eta}$ and $v_\beta$ increases to its maximum, 
            \item we cannot have $v_\beta'(x_1)=0$ by Lemma \ref{lem:vConstant} because $v_\beta(x_0)>\frac{h}{\eta}$. 
		\end{enumerate} Thus, (\ref{eq:vLimitLem_vDerivativePositive}) follows. Consequently, we have that 
		\begin{align}
		v_{\beta}(x) > \frac{h}{\eta}\quad\text{for}\quad x\in (x_1, x_3).\label{eq:142}
		\end{align}
		By continuity, $v_\beta(x)$ achieves a local maximum at some $\hat{x}\in (x_1,x_3)$ and $v_\beta(\hat{x})>\frac{h}{\eta}$, but this contradicts Lemma \ref{lem:vUpperBound}. Therefore, we conclude that
		\begin{align}
		v_{\beta}(x) > \frac{h}{\eta}\quad \text{for}\quad x\ge x_1.\label{eq:143}
		\end{align}
        In particular, $\beta\not\in{\cal D}_i$. Rather, $\beta\in {\cal I}_i$ and $v_\beta$ is nondecreasing by Corollary (\ref{cor:SetsPartition}). So, we have that \begin{align}
            v_\beta(x) \geq v_\beta(x_0) > \frac{h}{\eta} \quad \text{ for } x> x_0.
        \end{align} To conclude the proof, we consider two cases: Case (i) $a\leq 0$, Case (ii) $a>0$. When $a\leq 0$, we note from (\ref{eq:1}) that
		\begin{align}
		\frac{\sigma^2}{2}v_{\beta}^{\prime}(y) \geq \beta + \frac{\hat{\alpha}}{4}\left(\frac{h}{\eta}\right)^2 \quad \text{for } y\ge x_0.\label{eq:144}
		\end{align} 
		Integrating both sides of (\ref{eq:144}) over $[x_0,y]$ gives
		\begin{align}
		v_{\beta}(y) \ge \frac{h}{\eta} + \frac{2}{\sigma^2}\left[\beta+\frac{\hat{\alpha}}{4}\left(\frac{h}{\eta}\right)^2\right](y-x_1),\quad y\ge x_0, \label{eq:145}
		\end{align}
		where the right hand side tends to $\infty$, completing the proof when $a\leq 0$.

    When $a>0$, we note from (\ref{eq:1}) that \begin{align}
        v_\beta'(y)+\frac{2a}{\sigma^2} v_\beta(y) &= \frac{2\beta}{\sigma^2}+\frac{\hat{\alpha}}{2\sigma^2} v_\beta^2(y)+\eta y\left(v_\beta(y)-\frac{h}{\eta}\right), \quad y\geq x_0.\label{eq:vBetaWhenAPositive}
    \end{align} We let $\epsilon=v_\beta(x_0)-h/\eta>0$ and write from (\ref{eq:vBetaWhenAPositive}) that
    \begin{align}
        v_\beta'(y)+\frac{2a}{\sigma^2} v_\beta(y) &= \frac{2\beta}{\sigma^2}+\frac{\hat{\alpha}}{2\sigma^2}\left(\frac{h}{\eta}\right)+\epsilon\eta y.
    \end{align} Multiplying both sides of this with the integrating factor $\exp\left\{\frac{2a}{\sigma^2}y\right\}$ yields: \[\left(v_\beta(y) \exp\left\{\frac{2a}{\sigma^2}y\right\}\right)'\geq C \exp \left\{\frac{2a}{\sigma^2}y\right\}+\epsilon\eta y \exp\left\{\frac{2a}{\sigma^2}y,\right\}\] where $C=\frac{2\beta}{\sigma^2}+\frac{\hat{\alpha}}{\sigma^2}(h/\eta)^2>0$. Integrating both sides of this on $[x_0, y]$ yields \[v_\beta(y)\geq v_\beta(x_0)+C\left(1-\exp\left\{-\frac{2a}{\sigma^2}(y-x_0)\right\}\right)+\epsilon\eta \frac{\sigma^4}{4a^2}\left(\frac{2a}{\sigma^2}y-1\right), \] where the right-hand side tends to $\infty$ as $y\rightarrow\infty$, completing the proof when $a>0$.
	\end{proof}
	
	\begin{lem}\label{lem:16}
		For $0\le \beta_1<\beta_2$, we have that $v_{\beta_1}(x)<v_{\beta_2}(x)$ for all $x>0$. That is, $v_{\beta}(x)$ is an increasing function of $\beta$ for each $x>0$.
	\end{lem}
	\begin{proof}
		Let $\beta_2>\beta_1\ge 0$. We argue by contradiction. Suppose $v_{\beta_1}(x) \ge v_{\beta_2}(x)$ for some $x>0$, and let
		\begin{align*}
		\hat{x}=\inf\left\{x>0: v_{\beta_1}(x) \ge v_{\beta_2}(x)\right\}.
		\end{align*}
		Then there exists a sequence $\left\{x_n\right\}$ that decreases to $\hat{x}$, i.e., $x_n\searrow\hat{x}$ as $n\rightarrow\infty$, such that $v_{\beta_1}(x_n)\ge v_{\beta_2}(x_n)$ for all $n$. Recall that $v_{\beta_1}(0)=v_{\beta_2}(0)=-r$ and $v_{\beta_2}^{\prime}(0)> v_{\beta_1}^{\prime}(0)$. Hence, $v_{\beta_2}>v_{\beta_1}$ in a neighborhood around zero. This and continuity of $v_{\beta_1}$ and $v_{\beta_2}$ imply that
		\begin{align}
		v_{\beta_1}(\hat{x}) = v_{\beta_2}(\hat{x}).\label{eq:146}
		\end{align}
		Consequently, we can write 
		\begin{align*}
		\frac{v_{\beta_1}(x_n) - v_{\beta_1}(\hat{x})}{x_n-\hat{x}}\ge \frac{v_{\beta_2}(x_n) - v_{\beta_2}(\hat{x})}{x_n-\hat{x}},\quad n\ge 1.
		\end{align*}
		Passing to the limit as $n\rightarrow\infty$, we conclude that
		\begin{align}
		v^{\prime}_{\beta_1}(\hat{x})\ge v^{\prime}_{\beta_2}(\hat{x}).\label{eq:147}
		\end{align}
		Note, however, from IVP($\beta$) that for $\beta=\beta_1,\beta_2$ we have
		\begin{align}
		\frac{\sigma^2}{2}v_{\beta_1}^{\prime}\left(\hat{x}\right) &= \beta_1 + \frac{\hat{\alpha}}{4}v_{\beta_1}^2(\hat{x}) + \eta\hat{x}\left(v_{\beta_1}(\hat{x})-\frac{h}{\eta}\right) -av_{\beta_1}(\hat{x}),\label{eq:155}\\
		\frac{\sigma^2}{2}v_{\beta_2}^{\prime}\left(\hat{x}\right) &= \beta_2 + \frac{\hat{\alpha}}{4}v_{\beta_2}^2(\hat{x}) + \eta\hat{x}\left(v_{\beta_2}(\hat{x})-\frac{h}{\eta}\right)-av_{\beta_1}(\hat{x}).\label{eq:156}
		\end{align}
		Subtracting (\ref{eq:155}) from (\ref{eq:156}) and using (\ref{eq:146}) yield
		\begin{align*}
		\frac{\sigma^2}{2}\left[v^{\prime}_{\beta_2}(\hat{x})-v^{\prime}_{\beta_1}(\hat{x})\right] = \beta_2-\beta_1>0,
		\end{align*}
		which contradicts (\ref{eq:147}). Thus, we conclude that $v_{\beta_2}(x)>v_{\beta_1}(x)$ for $x>0$.
	\end{proof}
	
	\begin{lem}\label{lem:17}
		For $x>0$, we have that $v_{\beta}(x)$ is continuous in $\beta$ on $[0,\infty)$. That is, for $x>0$, given $\beta\ge 0$ and $\ep>0$, there exists a $\delta>0$ such that $\vert v_{\beta}(x) - v_{\tilde{\beta}}(x)\vert <\epsilon$ for all $\tilde{\beta}\in \left(\beta-\delta,\beta+\delta\right)\cap [0,\infty)$.
	\end{lem}
	\begin{proof}
		Let $x>0$ and $\beta_2>\beta_1\ge 0$. Integrating IVP$(\beta)$ over $[0,x]$ for $\beta=\beta_1,\beta_2$, we arrive at the following two equations:
		\begin{align}
		\frac{\sigma^2}{2}v_{\beta_1}(x) &= -\frac{\sigma^2}{2}r+\beta_1 x + \frac{\hat{\alpha}}{4}\int_{0}^{x}v_{\beta_1}^2(y)\,dy + \eta\int_{0}^{x}y\left(v_{\beta_1}(y) - \frac{h}{\eta}\right)\,dy -\int_0^x a v_{\beta_1} (y) dy,\label{eq:157}\\
		\frac{\sigma^2}{2}v_{\beta_2}(x) &= -\frac{\sigma^2}{2}r+\beta_2 x + \frac{\hat{\alpha}}{4}\int_{0}^{x}v_{\beta_2}^2(y)\,dy + \eta\int_{0}^{x}y\left(v_{\beta_2}(y) - \frac{h}{\eta}\right)\,dy -\int_0^x a v_{\beta_2} (y) dy.\label{eq:158}
		\end{align}
		Subtracting (\ref{eq:157}) from (\ref{eq:158}) gives the following:
		\begin{align}
		\frac{\sigma^2}{2}\left[v_{\beta_2}(x) - v_{\beta_1}(x)\right] &= \left(\beta_2-\beta_1\right) x + \frac{\hat{\alpha}}{4}\int_{0}^{x}\left[v_{\beta_2}^2(y) - v_{\beta_1}^2(y)\right]\,dy \notag\\
        &\quad  + \eta\int_{0}^{x}y \left[v_{\beta_2}(x) - v_{\beta_1}(x)\right]\,dy -a\int_0^x \left[v_{\beta_1}(y)-v_{\beta_2} (y)\right]\, dy.\label{eq:159}
		\end{align}
		In order to facilitate the bound, let $\bar{\beta}>\beta_2>\beta_1\ge 0$ and note from Lemma \ref{lem:16} that
		\begin{align*}
		v_0(y)\le v_{\beta_1}(y)\le v_{\beta_2}(y) \le v_{\bar{\beta}}(y),\quad y\ge 0.
		\end{align*} 
		Hence for $y\ge 0$ we have that
		\begin{align*}
		2v_0(y) \le v_{\beta_1}(y) + v_{\beta_2}(y) \le 2v_{\bar{\beta}}(y),
		\end{align*}
		from which we conclude that 
		\begin{align*}
		\vert v_{\beta_1}(y) +v_{\beta_2}(y) \vert \le 2 \max\left(\vert v_0(y)\vert + \vert v_{\bar{\beta}}(y)\vert \right).
		\end{align*}
		Thus, letting
		\begin{align*}
		K\left(\bar{\beta}\right) = 2 \sup_{0\le y\le x}\left\{\max\left(\vert v_0(y)\vert + \vert v_{\bar{\beta}}(y)\vert \right)\right\},
		\end{align*}
		we arrive at the following for $y\in [0,x]$:
		\begin{align*}
		\vert v_{\beta_2}^2(y) - v_{\beta_1}^2(y) \vert= \vert v_{\beta_2}(y) + v_{\beta_1}(y)\vert \cdot \vert v_{\beta_2}(y) -v_{\beta_1}(y) \vert \le K\left(\bar{\beta}\right)\vert v_{\beta_2}(y) - v_{\beta_1}(y)\vert. 
		\end{align*}
		Combining this with (\ref{eq:159}) and letting
		\begin{align*}
		h(y) = \vert v_{\beta_2}(y) -v_{\beta_1}(y)\vert\quad \text{for}\quad y\in [0,x],
		\end{align*}
		yield the following inequality:
		\begin{align*}
		h(x) \le  \frac{2x}{\sigma^2}\vert \beta_2-\beta_1\vert +\left[\frac{\hat{\alpha}}{2\sigma^2}K\left(\bar{\beta}\right)+\eta x +|a|\right]\int_{0}^{x}h(y)\,dy.
		\end{align*}
		Then by Gronwall's inequality (e.g., see page 498 of \citet{EthierKurtz05}) 
		we conclude that 
		\begin{align*}
		h(x) \le \frac{2x}{\sigma^2}\vert \beta_2-\beta_1\vert \exp\left\{-\left(\eta x + \frac{\hat{\alpha}}{2\sigma^2}K\left(\bar{\beta}\right) +|a|\right)x\right\}.
		\end{align*}
		Thus, given $\ep>0$, we can let 
		\begin{align*}
		\delta = \frac{\ep\sigma^2}{2x}\exp\left\{-\left(\eta x + \frac{\hat{\alpha}}{2\sigma^2}K\left(\bar{\beta}\right) +|a| \right)x\right\},
		\end{align*}
		so that $\vert \beta_2-\beta_1\vert<\delta$ implies that $h(x) = \vert v_{\beta_2}(x) -v_{\beta_1}(x)\vert<\ep$. This concludes the proof.
	\end{proof}
	
	\begin{lem}\label{lem:18}
		Under Assumption \ref{ass:ParamAssumption}, we have the following: \begin{enumerate}[label=(\roman*)]
            \setlength{\itemsep}{0em}
		    \item In Case 1 of Assumption \ref{ass:ParamAssumption} (Assumption \ref{ass:ParamAssumption}(a)), for $0\leq \beta_1<\beta_2$, if $\beta_2\in {\cal D}_1$, then $\beta_1\in {\cal D}_1$. That is, $[0,\beta_2]\subseteq {\cal D}_1$ whenever $\beta_2\in {\cal D}_1$.
            \item In Case 2 of Assumption \ref{ass:ParamAssumption} (Assumption \ref{ass:ParamAssumption}(b)), for $\underline{\beta}_2 < \beta_1<\beta_2$, if $\beta_2\in {\cal D}_2$, then $\beta_1\in {\cal D}_2$. That is, $(\underline{\beta}_2,\beta_2]\subseteq {\cal D}_2$ whenever $\beta_2\in {\cal D}_2$.
		\end{enumerate}
	\end{lem}
	\begin{proof}
		Consider part (i), and let $\beta_2>\beta_1\ge 0$. Then by Lemma \ref{lem:14}, there exists $x_0>0$ such that $v_{\beta_2}(x_0)<-r$. In turn, by Lemma \ref{lem:16}, we have that
		\begin{align*}
		v_{\beta_1}(x_0)<v_{\beta_2}(x_0)<-r,
		\end{align*}
		Thus, $\beta_1\in\mathcal{D}_1$ by Lemma \ref{lem:14}. Proof of part (ii) follows similarly.
	\end{proof}
	
	\begin{lem}\label{lem:19}
		Under Assumption \ref{ass:ParamAssumption}, we have the following: \begin{enumerate}[label=(\roman*)]
            \setlength{\itemsep}{0em}
		    \item In Case 1 of Assumption \ref{ass:ParamAssumption} (Assumption \ref{ass:ParamAssumption}(a)), ${\cal D}_1\not=\emptyset$. In particular, $0\in {\cal D}_1$ and there exists a $\Tilde{\beta}_1>0$ such that $[0, \Tilde{\beta}]\subseteq {\cal D}_1$. 
            \item In Case 2 of Assumption \ref{ass:ParamAssumption} (Assumption \ref{ass:ParamAssumption}(b)), ${\cal D}_2\not=\emptyset$. In particular, there exists $\Tilde{\beta}_2>\underline{\beta}_2$ such that $(\underline{\beta}_2, \Tilde{\beta}_2] \subseteq {\cal D}_2$. 
		\end{enumerate}
	\end{lem}
	\begin{proof}
		Consider part (i). We first show $0\in {\cal D}_1$. Aiming for a contradiction,  suppose  $0\not\in\mathcal{D}_1$ so that $0\in\mathcal{I}_1$ by Corollary \ref{cor:SetsPartition}. We consider the following two cases:\vspace{-0.25em}
		\begin{itemize}
			\itemsep0em
			\item Case A: $v_0(y)\le 0$ for all $y>0$.
			\item Case B: $v_0(y)>0$ for some $y>0$.\vspace{-0.25em}
		\end{itemize} 
        Consider Case A.
		Because $0\in\mathcal{I}_1$, $v^{\prime}_0(y)\ge 0$ for all $y\ge 0$. Then, we have that 
		$-r\le v_0(y)\le 0$ for all $y\ge 0$.
		Substituting this into IVP($\beta$) for $\beta=0$, we consider the following two subcases of Case A: $a\geq 0$ and $a\in (-\frac{\alpha r}{4}$, 0).

        For $a\geq 0$, we conclude that
		\begin{align*}
		0\le \frac{\sigma^2}{2}v_0^{\prime}(y) \le \frac{\hat{\alpha}}{4}r^2 - hy + ar,
		\end{align*}
		where the right-hand side tends to $-\infty$. Thus, there exists $y>0$ such that $v_0^{\prime}(y)<0$, contradicting $0\in\mathcal{I}_1$.

        For $a\in (-\frac{\alpha r}{4}, 0)$, we conclude that \begin{align*}
		0\le \frac{\sigma^2}{2}v_0^{\prime}(y) \le \frac{\hat{\alpha}}{4}r^2 - hy,
		\end{align*} where the right-hand side tends to $-\infty$. Once again, there exists $y>0$ such that $v_0'(y)<0$, contradicting $0\in {\cal I}_1$. 
  
        Consider Case B. In this case, we let
		$y_0 = \inf\left\{y>0: v_0(y)>0\right\}$.
		By continuity of $v_0$ and $v_0(0)=-r<0$, we have that $v_0(y_0)=0$ and $y_0>0$. Substituting this into IVP($\beta$) for $\beta=0$ at $y=y_0$ gives 
		\begin{align*}
		\frac{\sigma^2}{2}v_0^{\prime}(y_0)=-hy_0<0.
		\end{align*}
		Thus, $0\in\mathcal{D}_1$ by Lemma \ref{lem:14}, a contradiction. Combining Cases A and B, we conclude that $0\in\mathcal{D}_1$. Then it follow from Lemma \ref{lem:14} that $v_0(y)\rightarrow -\infty$ as $y\rightarrow\infty$. Thus, there exists a $x_0>0$ such that $v_0(x_0)<-2r$. Then, by continuity of $v_{\beta}(x_0)$ in $\beta$ (see Lemma \ref{lem:17}), there exists a $\tilde{\beta}_1>0$ such that $v_{\tilde{\beta}}(x_0)<-r$. By Lemma \ref{lem:14}, we conclude $\Tilde{\beta}_1\in {\cal D}_1$. Then we conclude by Lemma \ref{lem:18} that $[0,\tilde{\beta}_1]\subseteq \mathcal{D}_1$.

  Consider part (ii). Recall that in Case 2 of Assumption \ref{ass:ParamAssumption}, $a\leq -\hat{\alpha}r/4$ and $\underline{\beta}_2=-ar-\hat{\alpha} r^2/4\geq 0$. Consider $v_{\beta_2}$ and note that $v_{\beta_2}(0)=-r$. It follows from (\ref{eq:1}) that $v_{\beta_2}'(0)=0$. Moreover, differentiating both sides of (\ref{eq:1}) and using $v_{\beta_2}'(0)=0$, we conclude that \[v_{\underline{\beta}_2}''(0) = -\frac{2\eta}{\sigma^2}\left(r+\frac{h}{\eta}\right)<0.\]
  Thus, $v_{\underline{\beta}_2}$ is decreasing and below $-r$ in a neighborhood of zero. Next, we argue that $v_{\underline{\beta}_2}(x)\leq -r$ for all $x>0$.
  
  Suppose not, and let $x_1={\rm inf}\{x>0:v_{\underline{\beta}_2}(x)-r\}$. By continuity of $v_\beta$, we have $v_{\underline{\beta}_2}(x_1)=-r$. We also have by its definition that $v_{\underline{\beta}_2}'(x_1)\geq 0$ and $x_1>0$. Then by combining these with (\ref{eq:1}), we write \begin{align*}
      0\leq v_{\underline{\beta}_2}'(x_1) &= -ar -\frac{\hat{\alpha}}{4}r^2+\frac{\hat{\alpha}}{4}r^2-\eta x_1\left(r+\frac{h}{\eta}\right)+ar   \\
      &= -\eta x_1 \left(r+\frac{h}{\eta}\right) < 0, 
  \end{align*} a contradiction. Thus, $v_{\underline{\beta}_2}(x)\leq -r$ for all $x\geq 0$. 
  
  Next, we argue that $\lim_{x\rightarrow\infty} v_{\underline{\beta}_2}(x)=-\infty$. Suppose not (Note that we can rule out oscillatory behavior following the same technique in the proof of Lemma \ref{lem:vIncreaseSup}). Then, there exists $k>r$ such that \[v_{\underline{\beta}_2}(x)\geq -k, \quad x\geq 0.\] But using (\ref{eq:1}), we conclude that \[\frac{\sigma^2}{2} v_{\underline{\beta}_2}'(x) \leq \beta_2+\frac{\hat{\alpha}}{4}K^2-\eta x\left(r+\frac{h}{\eta}\right)-ar.\] Integrating both sides from $0$ to $y$ yields \[\frac{\sigma^2}{2} v_{\underline{\beta}_2}(y) \leq -r\frac{\sigma^2}{2}+\left[\underline{\beta}_2-ar+\frac{\hat{\alpha}}{4}\right]y-\frac{\eta}{2}\left(r+\frac{h}{\eta}\right)\frac{y^2}{2},  \] where the right-hand side tends to $-\infty$ as $y\rightarrow\infty$. Thus, $v_{\underline{\beta}_2}(x)\rightarrow-\infty$ as $x\rightarrow\infty$ and there exists $x_2$ such that $v_{\underline{\beta}_2}(x_2) <-2r$. Then, by Lemma \ref{lem:16}, there exists $\Tilde{\beta}_2>\underline{\beta}_2$ such that $v_{\Tilde{\beta}_2}(x_2)<-r$. In particular, $\Tilde{\beta}_2\in {\cal D}_2$ by Lemma \ref{lem:14}. Then, by Lemma \ref{lem:18}, we conclude that $(\underline{\beta}_2, \Tilde{\beta}_2]\subset {\cal D}_2$.
	\end{proof}
	
	\begin{lem}\label{lem:20}
		Under Assumption \ref{ass:ParamAssumption}, we have ${\cal I}_i\not=\emptyset$ for $i=1,2$. In particular, \[\left(\frac{\sigma^2 h}{2\eta}+ 2\sigma\left(r+\frac{h}{\eta}\right)\sqrt{\frac{\eta}{\pi}} \exp\left\{-\frac{\sigma^2 a^2}{4\eta}\right\}, \infty\right)\subseteq \mathcal{I}_i, \quad i=1,2.\]
	\end{lem}
	\begin{proof}
		We establish the result by showing that $v_{\beta}(x)\rightarrow\infty$ as $x\rightarrow\infty$ for sufficiently large $\beta>0$. The result then follows from Corollary \ref{cor:SetsPartition} and Lemmas \ref{lem:14} and \ref{lem:16}. To that end, we rewrite IVP($\beta$) as follows: 
		\begin{align*}
		v_{\beta}^{\prime}(y) -\frac{2\eta}{\sigma^2}yv_{\beta}(y) +a v_{\beta}(y) = \frac{2\beta}{\sigma^2} + \frac{\hat{\alpha}}{2\sigma^2}v_{\beta}^2(y) - \frac{2h}{\sigma^2}y,\quad y\ge 0.
		\end{align*}
		Multiplying both sides with the integrating factor $\exp\left\{-\frac{\eta}{\sigma^2}y^2 +ay \right\}$ yields the following bound:
		\begin{align*}
		\left[\exp\left\{-\frac{\eta}{\sigma^2}y^2 +ay\right\}v_{\beta}(y)\right]^{\prime}\ge \frac{2\beta}{\sigma^2}\exp\left\{-\frac{\eta}{\sigma^2}y^2 + ay\right\} - \frac{2h}{\sigma^2}y\exp\left\{-\frac{\eta}{\sigma^2}y^2 + ay \right\}.
		\end{align*}
		Integrating both sides of the above inequality over $[0,x]$ and using $v_{\beta}(0)=-r$ gives:
		\begin{align}
		\exp\left\{-\frac{\eta}{\sigma^2}x^2 +ax\right\}v_{\beta}(x)\ge -r+\frac{2\beta}{\sigma^2}I_1 - \frac{2h}{\sigma^2} I_2,\label{eq:122*}
		\end{align} where \begin{align*}
		    I_1 &= \int_0^x \exp\left\{-\frac{\eta}{\sigma^2}y^2+ay\right\} dy \text{ and }
            I_2 = \int_0^x y \exp\left\{-\frac{\eta}{\sigma^2}y^2+ay\right\} dy.
		\end{align*}
		First, we consider $I_1$ and write \[I_1 = \exp\left\{\frac{\sigma^2}{4\eta} a^2\right\} \int_0^x \exp\left\{-\frac{\eta}{\sigma^2}\left(y-\frac{a\sigma^2}{2\eta}\right)^2\right\} dy.\] Applying the change of variable $u=\frac{\sqrt{2\eta}}{\sigma}(y-\frac{\sigma^2}{2\eta})$ yields
		\begin{align}
		I_1 &= \sqrt{\frac{\pi}{\eta}} \sigma \exp\left\{\frac{\sigma^2}{4\eta} a^2\right\}\int_{ -\frac{\sigma}{\sqrt{2\eta}}}^{\frac{\sqrt{2\eta}}{\sigma}(x-\frac{\sigma^2}{2\eta})} \frac{1}{\sqrt{2\pi}} \exp\left\{-\frac{u^2}{2}\right\}\,du \nonumber   \\
        &=\sqrt{\frac{\pi}{\eta}}\sigma \exp\left\{\frac{\sigma^2}{4\eta}a^2\right\} \left[\Phi\left(\frac{\sqrt{2\eta}}{\sigma}\left(x-\frac{\sigma^2}{2\eta}\right)\right)-\Phi \left(-\frac{\sigma}{\sqrt{2\eta}}\right)\right],\label{eq:I1}
		\end{align}
		where $\Phi$ is the CDF for the standard normal distribution. Next, we turn to $I_2$ and facilitate its derivative by first deriving \begin{align*}
		    I_3 &= \int_0^x \left(y-\frac{\sigma^2 a}{2\eta}\right) \exp\left\{-\frac{\eta}{\sigma^2}y^2+ay\right\} dy.
		\end{align*} Note that $I_3=I_2-\frac{\sigma^2 a}{2\eta}I_1$. Using the change of variable $u=-\frac{\eta}{\sigma^2}y^2+ay$, we write \begin{align*}
		    I_3 &= \int_0^{-\frac{\eta}{\sigma^2}x^2+ax} -\frac{\sigma^2}{2u}e^u du = \frac{\sigma^2}{2\eta} \left[1-\exp\left\{-\frac{\eta}{\sigma^2}x^2+ax\right\}\right].
		\end{align*} Then, using $I_2=I_3+\frac{a\sigma^2}{2\eta}I_1$, we arrive at 
		\begin{align} 
		I_2 =\frac{\sigma^2}{2\eta}-\frac{\sigma^2}{2\eta}\exp\left\{-\frac{\eta}{\sigma^2}x^2+ax\right\}+\frac{\sigma^2 a^2}{2\eta}\exp\left\{\frac{a^2\sigma^2}{4\eta}\right\}\sqrt{\frac{\pi}{\eta}}\sigma \left[\Phi\left(\frac{\sqrt{2\eta}}{\sigma}x-\frac{\sigma}{\sqrt{2\eta}}\right)-\Phi\left(-\frac{\sigma}{\sqrt{2\eta}}\right)\right].\label{eq:I2}
		\end{align} 
		Substituting (\ref{eq:I1})--(\ref{eq:I2}) into (\ref{eq:122*}) then gives 
		\begin{align}
		\exp\left\{-\frac{\eta}{\sigma^2}x^2 +ax\right\}v_{\beta}(x) &\geq -r-\frac{h}{\eta}+\frac{h}{\eta}\exp\left\{-\frac{\eta}{\sigma^2}x^2+ax\right\} \nonumber \\
        &\quad +\frac{2\beta}{\sigma^2}\sqrt{\frac{\pi}{\eta}} \exp \left\{\frac{\sigma^2}{4\eta}a^2\right\} \left[\Phi\left(\frac{\sqrt{2\eta}}{\sigma}x-\frac{\sigma}{\sqrt{2\eta}}\right)-\Phi \left(-\frac{\sigma}{\sqrt{2\eta}}\right)\right] \nonumber  \\
        &\quad -\sigma\frac{h}{\eta} \exp\left\{\frac{a^2 \sigma^2}{4\eta}\right\}\sqrt{\frac{\pi}{\eta}}\left[\Phi\left(\frac{\sqrt{2\eta}}{\sigma}x-\frac{\sigma}{\sqrt{2\eta}}\right)-\Phi\left(-\frac{\sigma}{\sqrt{2\eta}}\right)\right]. \label{eq:SubstituteIs}
		\end{align}
        Note that there exists $x_0>0$ large enough so that \begin{align} \label{eq:LargeX}
            \Phi\left(\frac{\sqrt{2\eta}}{\sigma}x-\frac{\sigma}{\sqrt{2\eta}}\right)-\Phi\left(-\frac{\sigma}{\sqrt{2\eta}}\right) &\geq \frac{1}{4}.
        \end{align} Then for $x\geq x_0$ and $\beta>\frac{\sigma^2 h}{2\eta}$, combining (\ref{eq:SubstituteIs}) and (\ref{eq:LargeX}), we write, \begin{align*}
            \exp\left\{-\frac{\eta}{\sigma^2}x^2+ax\right\}v_{\beta}(x) \geq -r+\left(\frac{2\beta}{\sigma}-\frac{\sigma h}{\eta}\right)\exp\left\{\frac{\sigma^2 a^2}{4\eta}\right\}\sqrt{\frac{\pi}{\eta}}\frac{1}{\eta}-\frac{h}{\eta} +\frac{h}{\eta} \exp\left\{-\frac{\eta}{\sigma^2}x^2+ax\right\}.
        \end{align*} 
		Thus, we have the following lower bound on $v_\beta(\cdot)$: 
		\begin{align}
		v_{\beta}(x)\ge\left[\frac{1}{4}\left(\frac{2\beta}{\sigma}-\frac{\sigma h}{\eta}\right)\exp\left\{\frac{\sigma^2 a^2}{4\eta}\right\}\sqrt{\frac{\pi}{\eta}} - \left(r+\frac{h}{\eta}\right)\right]\exp\left\{\frac{\eta}{\sigma^2}x^2 +ax\right\} +\frac{h}{\eta},\quad x>x_0.\label{eq:126*}
		\end{align}
		In particular, we note that for $\beta>\frac{\sigma^2 h}{2\eta}+2\sigma\left(r+\frac{h}{\eta}\right)\sqrt{\frac{\eta}{\pi}}\exp\{-\frac{\sigma^2 a^2}{4\eta}\}$, the right-hand side of (\ref{eq:126*}) tends to $\infty$ as $x\rightarrow\infty$. Thus, $\beta\in\mathcal{I}_i$ for $i=1,2$ whenever it is above $\frac{\sigma^2 h}{2\eta}+2\sigma (r+\frac{h}{\eta})\sqrt{\frac{\eta}{\pi}}\exp\{-\frac{\sigma^2 a^2}{4\eta}\}$, completing the proof.
	\end{proof}
	
	To facilitate the analysis, under Case $i$ of Assumption \ref{ass:ParamAssumption}, we define $\beta_i^*= \inf\mathcal{I}_i$ for $i=1,2$. The remaining results will prove that this $\beta_i^*$ along with its corresponding $v_{\beta_i^*}$, solve the Bellman equation in Case $i$ for $i=1,2$.
	
	\begin{lem}\label{lem:21}
		Under Case $i$ of Assumption \ref{ass:ParamAssumption}, we have that $\beta_i^*>0$ for $i=1,2$.
	\end{lem}
	\begin{proof}
		Recall from Lemma \ref{lem:19} that there exists a $\tilde{\beta}_i>0$ such that $\Tilde{\beta}_i\in {\cal D}_i$ for $i=1,2$. Clearly, we must have $\beta\ge \tilde{\beta}_i$ for $\beta\in\mathcal{I}_i$ and $i=1,2$. Thus, we conclude that $\beta_i^*=\inf\mathcal{I}_i\ge\tilde{\beta}_i>0$ for $i=1,2$.
	\end{proof}
	
	\begin{lem}\label{lem:22}
		Under Case $i$ of Assumption \ref{ass:ParamAssumption}, We have that $\beta_i^*\in\mathcal{I}_i$ and $v_{\beta_i^\ast}$ is bounded for $i=1,2$. 
	\end{lem}
	\begin{proof}
        Consider Case $i$ of Assumption \ref{ass:ParamAssumption} for $i=1,2$. We argue by contradiction. Suppose $\beta_i^*\not\in \mathcal{I}_i$. Then, by Corollary \ref{cor:SetsPartition}, $\beta_i^*\in\mathcal{D}_i$. In particular, by Lemma \ref{lem:14}, there exists a $x_0>0$ such that $v_{\beta_i^*}(x)<-r$. Because $v_{\beta}(x_0)$ is continuous in $\beta$ (see Lemma \ref{lem:17}), there exists a $\delta>0$ such that 
		\begin{align}
		v_{\beta}(x_0)<-r\quad\text{for}\quad \beta\in\left(\beta_i^*-\delta, \beta_i^*+\delta\right).\label{eq:122}
		\end{align}
		However, by definition of $\beta_i^*$, there exists a $\hat{\beta}_i\in \left(\beta_i^*, \beta_i^*+\delta\right)$ such that $\hat{\beta}_i\in\mathcal{I}_i$. Applying Lemma \ref{lem:14} again, it follows that $v_{\hat{\beta}_i}(x)\ge -r$ for all $x\ge 0$, contradicting (\ref{eq:122}). Thus, $\beta_i^*\in\mathcal{I}_i$. 
		
		We now prove that $v_{\beta_i^*}$ is bounded. Aiming for a contradiction, suppose it is not bounded. Then there exists a $x_0>0$ such that $v_{\beta_i^*}(x_0)>2h/\eta$. Then, because $v_{\beta}(x_0)$ is continuous in $\beta$ (by Lemma \ref{lem:17}) and $\beta_i^*>0$ (by Lemma \ref{lem:21}), there exists an $\ep>0$ such that $v_{\beta_i^*-\ep}(x_0)\ge h/\eta$. It follows that $v_{\beta_i^*-\ep}$ is unbounded by Lemma \ref{lem:vLimit}, which in turn implies that $\beta_i^*-\ep\in\mathcal{I}_i$ by Corollary \ref{cor:SetsPartition} and Lemma \ref{lem:13}. That $\beta_i^*-\ep\in\mathcal{I}_i$, however, contradicts the definition of $\beta_i^*$.
	\end{proof}
	
	\begin{lem}\label{lem:23}
        Under Assumption \ref{ass:ParamAssumption}, the following hold: \begin{enumerate}[label=(\roman*)]
            \item ${\cal D}_1=[0, \beta^*_1)$ and ${\cal I}_1=[\beta_1^*, \infty)$, 
            \item ${\cal D}_2=( \underline{\beta}_2, \beta_2^*)$ and ${\cal I}_2=[\beta_2^*, \infty)$.
        \end{enumerate}
	\end{lem}
	\begin{proof}
        Consider Case $i$ of Assumption \ref{ass:ParamAssumption}.
		Suppose that there exists a $\beta>\beta_i^*$ such that $\beta\in\mathcal{D}_i$. Then by Lemma \ref{lem:18} it follows that $\beta_i^*\in\mathcal{D}_i$, contradicting Lemma \ref{lem:22}. Hence, no such $\beta$ exists. Combining this with Lemma \ref{lem:22} and the definition of $\beta_i^*$ concludes the proof.
	\end{proof}
	
	\begin{lem}\label{lem:24}
        Under Case $i$ of Assumption \ref{ass:ParamAssumption}, 
		we have that $v_{\beta_i^*}$ is nondecreasing with $\lim\limits_{x\rightarrow\infty}v_{\beta_i^*}(x) = h/\eta$ for $i=1,2$.
	\end{lem}
	\begin{proof}
        Consider Case $i$ of Assumption \ref{ass:ParamAssumption} for $i=1,2$. 
		Because $\beta_i^*\in\mathcal{I}_i$ by Lemma \ref{lem:22}, $v_{\beta_i^*}$ is nondecreasing. Also, by Lemma \ref{lem:22} we have that $v_{\beta_i^*}$ is bounded. Consequently, by Lemma \ref{lem:vLimit}, we have that
		\begin{align*}
		v_{\beta_i^*}(x) \leq \frac{h}{\eta}\quad \text{for}\quad x\geq 0.
		\end{align*}
		Moreover, because $v_{\beta_i^*}$ is nondecreasing, its limit is well-defined and satisfies
		\begin{align*}
		\lim\limits_{x\rightarrow\infty}v_{\beta_i^*}(x) \le \frac{h}{\eta}.
		\end{align*}
		Now let $v = \lim\limits_{x\rightarrow\infty}v_{\beta_i^*}(x)$ and suppose that $v<\frac{h}{\eta}$. Consider IVP($\beta_i^*$):
		\begin{align*}
		\frac{\sigma^2}{2}v_{\beta_i^*}^{\prime}(y) = \beta_i^* + \frac{\hat{\alpha}}{4}v_{\beta^*}^2(y) +\eta y\left(v_{\beta_i^*}(y) - \frac{h}{\eta}\right) -a v_{\beta_i^*}(y), \quad y\geq 0.
		\end{align*}
		Passing to the limit on both sides and noting that $v<\frac{h}{\eta}$ gives the following:
		\begin{align*}
		\frac{\sigma^2}{2}\lim\limits_{y\rightarrow\infty}v_{\beta_i^*}^{\prime}(y) = \beta_i^*+ \frac{\hat{\alpha}}{4}v^2 -av+\lim\limits_{y\rightarrow\infty}\eta y\left(v_{\beta_i^*}(y) - \frac{h}{\eta}\right) = -\infty.
		\end{align*}
		Thus, there exists a $x_0>0$ such that $v_{\beta_i^*}^{\prime}(x_0)<0$. We conclude by Lemma \ref{lem:14} that $\beta_i^*\in\mathcal{D}_i$, a contradiction. Therefore, \begin{align*}
		    v = \lim\limits_{x\rightarrow\infty}v_{\beta_i^*}(x)=h/\eta.
		\end{align*}
	\end{proof}
	We conclude this section with a proof of Theorem \ref{thm:1}:
	
	\begin{proof}[Proof of Theorem \ref{thm:1}]
		First, consider the case $a>-\frac{\hat{\alpha}}{4}r$ that is covered by Case 1 of Assumption \ref{ass:ParamAssumption} (Assumption \ref{ass:ParamAssumption}(a)). In this case, $(\beta_1^*, v_{\beta_1^*})$ solves Equations (\ref{eq:1})--(\ref{eq:2}) and this solution in unique by Lemma \ref{lem:vUnique}. Moreover, by Lemma \ref{lem:24}, we have that $\lim\limits_{x\rightarrow\infty}v_{\beta_1^*}(x) = h/\eta$. Finally, by Lemma \ref{lem:21}, we have that $\beta_1^*>0$. Therefore, $\left(\beta_1^*, v_{\beta_1^*}\right)$ solves the Bellman equations (\ref{eq:6.5*})--(\ref{eq:6.6*}) in this case. When $a\leq -\frac{\hat{\alpha}}{4}r$, Case 2 of Assumption \ref{ass:ParamAssumption} applies, and the proof follows from the same steps as in the first case. 
	\end{proof}
	
	\section{Proposed Policy}\label{sec:8}
	In this section we propose a dynamic pricing and dispatch policy for the problem introduced in Section \ref{sec:3} by interpreting the solution of the equivalent workload formulation (\ref{eq:5.12*})--(\ref{eq:ProcessL}) in the context of the original control problem. To describe the policy, recall that we considered a sequence of systems indexed by the number of jobs $n$, whose formal limit was the Brownian control problem (\ref{eq:3.16})--(\ref{eq:3.22}) under diffusion scaling. To articulate the proposed policy, we fix the system parameter $n$ and use it to unscale processes of interest. We define the (unscaled) workload process $W^n=\left\{W^n(t),\,t\ge 0\right\}$ as follows: 
	$$W^n(t) = \sum_{i=1}^{I}Q^n_i(t)\quad\text{for}\quad t\ge 0.$$ 
		
	\noindent\textbf{Proposed Pricing Policy:} Given the workload process $W^n$, we choose the demand rates
	\begin{align*}
	\lambda_i^n(t) = n\lambda_i^* +\frac{\sqrt{n}}{2\alpha_i}v\left(\frac{W^n(t)}{\sqrt{n}}\right),\quad i=1,\dots, I,\quad t\ge 0,
	\end{align*}
	where $v$ is the solution to the Bellman equation (\ref{eq:6.5*})--(\ref{eq:6.6*}). This follows from Equations (\ref{eq:41*}) and (\ref{eq:6.7}), Lemma \ref{lem:5}, and Theorem \ref{thm:2}. The corresponding proposed pricing policy is given by 
	\begin{align}
	p_i^n(t) = \Lambda_i^{-1}\left(\lambda_i^*\right) + \frac{\left(\Lambda_i^{-1}\right)^{\prime}\left(\lambda_i^*\right)}{2\alpha_i\sqrt{n}}v\left(\frac{W^n(t)}{\sqrt{n}}\right),\quad i=1,\dots, I,\quad t\ge 0,\label{eq:138*}
	\end{align}
	where $\Lambda^{-1}_i$ is the inverse of the demand rate function for region $i$. Equation (\ref{eq:138*}) is derived in Appendix \ref{app:A}.\vspace{0.5em}
	
	\noindent\textbf{Proposed Dispatch Policy:} We propose two dispatch policies and refer to them as Dispatch Policy 1 (DP1) and Dispatch Policy 2 (DP2). Dispatch Policy 1 (DP1) is motivated by the following observation. In the Brownian control problem under the complete resource pooling assumption, we set all but one of the inventory levels to zero. (The buffer with nonzero inventory corresponds to the one with lowest holding cost.) However, as articulated in \citet{Harrison1996}, zero inventory in the Brownian control problem corresponds to small positive inventory levels in the original system. Thus, we put small safety stocks in the various buffers and only serve them when inventory levels are at or above the threshold. To that end, denote by $s_i$ the safety stock for buffer $i$. 
	
	To be more specific, letting $\bar{\mathcal{A}}_i = \mathcal{A}_i \cap \left\{1,\dots, b\right\}$ denote the set of basic activities undertaken by server $i$ and letting $\bar{\mathcal{C}}_i = \mathcal{C}_i\cap \left\{1,\dots, b\right\}$ denote the set of basic activities that serve buffer $i$, our proposed dispatch policy is as follows: If server $i$ becomes idle at time $t$, it serves a job from the buffer in $\left\{b(j):j\in\bar{\mathcal{A}}_i,\, Q^n_{b(j)}(t)\ge s_{b(j)} \right\}$ with largest holding cost $h_{b(j)}$. In words, when server $i$ becomes idle, it looks at all buffers it servers by means of basic activities and serves the buffer with largest holding cost that is above its safety stock. To complete the policy description, suppose that at time $t$ the inventory in buffer $i$ increases from $s_i -1$ to $s_i $, i.e., reaches the safety stock. The system manager serves buffer $i$ by an idle server in $\left\{s(j):j\in\bar{\mathcal{C}_i}\right\}$ with largest effective idling cost $c_{s(j)}/\lambda^*_{s(j)}$, see Equation (\ref{eq:76}). In words, when buffer $i$ reaches the safety stock, i.e., that buffer becomes eligible for service, the system manager selects an idle server with largest effective idling cost than can serve the buffer by means of a basic activity.

    Dispatch Policy 2 (DP2) is motivated by the maximum pressure policy, see for example \citet{Stolyar2004},  \citet{DaiLin2005}, \citet{DaiLin2008}, and \citet{AtaLin2008}. Under this policy, each server prioritizes his own (local) buffer. If his own buffer is empty, then he checks the other buffers that he can serve using basic activities. If there are multiple such buffers, the server works on the buffer with the largest queue length. If the server's own (local) buffer is empty and he cannot serve any other buffers using basic activities, then he considers all remaining buffers he can serve (using nonbasic activities) and works next on the buffer with the largest queue length.  
	
	\section{Simulation Study}\label{sec:SimulationStudy}

 This section presents a simulation study to illustrate the effectiveness of the proposed policy. The simulation setting and its parameters are motivated, albeit loosely, by the taxi market in Manhattan, see \citet{AtaBarjestehKumar_Empirical} and the references therein. We set the number of cars, i.e., the system parameter, as $n=10,000$. As done in \citet{AtaBarjestehKumar_Empirical}, we divide Manhattan into $I=4$ regions, see Figure \ref{fig:Manhattan}. 
 \begin{figure}[h!]
     \centering
     \includegraphics[scale=0.5]{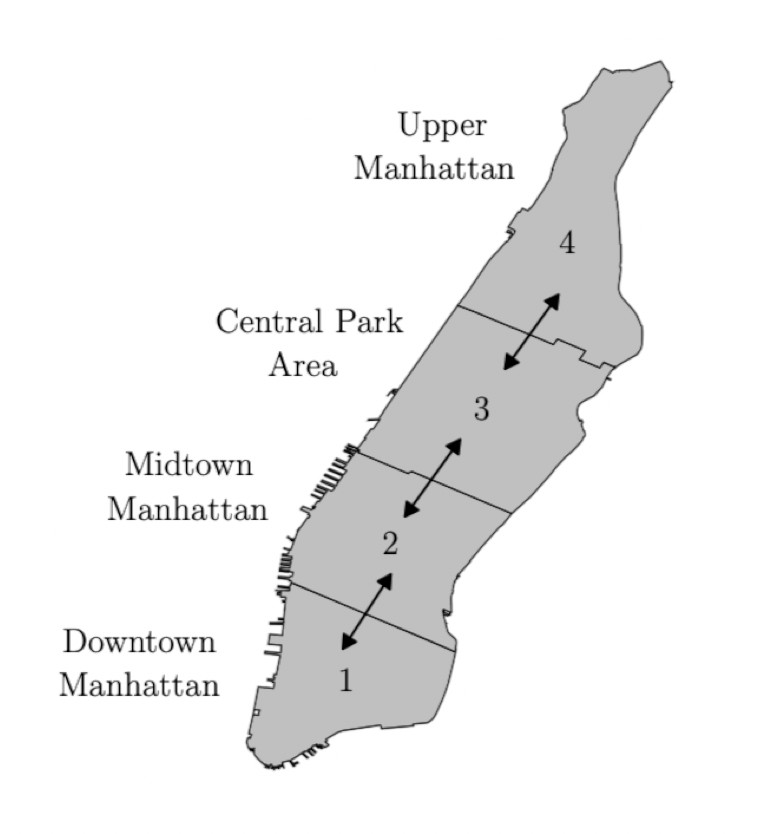}
     \caption{Manhattan area that are partitioned in four regions.}
     \label{fig:Manhattan}
 \end{figure}
 
 We assume cars can pick up customers in their own regions as well as from the neighboring regions. This gives rise to the following capacity consumption matrix: \begin{align*}
        A &=\begin{bmatrix}
	1 & 0 & 0 & 0 & 1 & 0 & 0 & 0 & 0 & 0 \\
	0 & 1 & 0 & 0 & 0 & 1 & 1 & 0 & 0 & 0 \\
	0 & 0 & 1 & 0 & 0 & 0 & 0 & 1 & 1 & 0 \\
    0 & 0 & 0 & 1 & 0 & 0 & 0 & 0 & 0 & 1
	\end{bmatrix}.
    \end{align*} Using the same dataset in \citet{AtaBarjestehKumar_Empirical}, we set\footnote{For simplicity, we use the preliminary results from \citet{AtaBarjestehKumar_Empirical} to estimate $\lambda^n$ and $q$ (based on a four-year dataset from January 2010 to December 2013). In doing so, we focus on the day shift of the non-holiday weekdays.} the demand rate (per hour) vector as follows: \[\lambda^n=(\lambda_1^n,\lambda_2^n,\lambda_3^n,\lambda_4^n)'=(3678,10723,6792,345)'.\] The corresponding limiting rate vector $\lambda^*$ is then computed as $\lambda^*=\lambda^n/n$, which yields 
    \begin{align}
        \lambda^*=(\lambda_1^*,\lambda_2^*,\lambda_3^*,\lambda_4^*)'=(0.367,1.072,0.679,0.0345)'. \label{eq:LimitingLambda}
    \end{align} Using this and Equation (\ref{eq:1.7}), we derive the input-output matrix $R$ as follows: \begin{align*}
        R = \begin{bmatrix}
	\lambda_1^* & 0 & 0 & 0 & 0 & \lambda_2^* & 0 & 0 & 0 & 0 \\
	0 & \lambda_2^* & 0 & 0 & \lambda_1^* & 0 & 0 & \lambda_3^* & 0 & 0 \\
	0 & 0 & \lambda_3^* & 0 & 0 & 0 & \lambda_2^* & 0 & 0 & \lambda_4^* \\
    0 & 0 & 0 & \lambda_4^* & 0 & 0 & 0 & 0 & \lambda_3^* & 0
	\end{bmatrix}.
    \end{align*} \citet{AtaBarjestehKumar_Empirical} reports the mean travel time as 13.2 minutes. To account for the pick up time and for other inefficiences that are not incorporated in our model, we inflate this by a factor of two, and set the mean trip time to 26.4 minutes. Thus $\eta^n=2.2727$ per hour. Moreover, because we study the system under the heavy traffic assumption (Assumption \ref{ass:3}), we set $\eta=e'\lambda^*=2.1539$. Therefore, we have that $\hat{\eta}=\sqrt{n}(\eta^n-\eta)=11.88$. 
    
    We estimate the routing probability vector $q$ from the data as  \[q=(q_1,q_2,q_3,q_4)'=(0.1647,0.5408,0.2724,0.0221)',\] which yields the limiting arrival rate vector $\nu$ to various buffers as follows: 
    \[\nu=\eta q=(0.3529,0.1159,0.5837,0.0474)'.\] Thus using the data $A, R$, and $\gamma$, one can compute the unique nominal processing plan $x^*$, referred to in Assumption \ref{ass:3}. It is displayed in Figure \ref{fig:TaxiNetwork_StaticSolution}.
    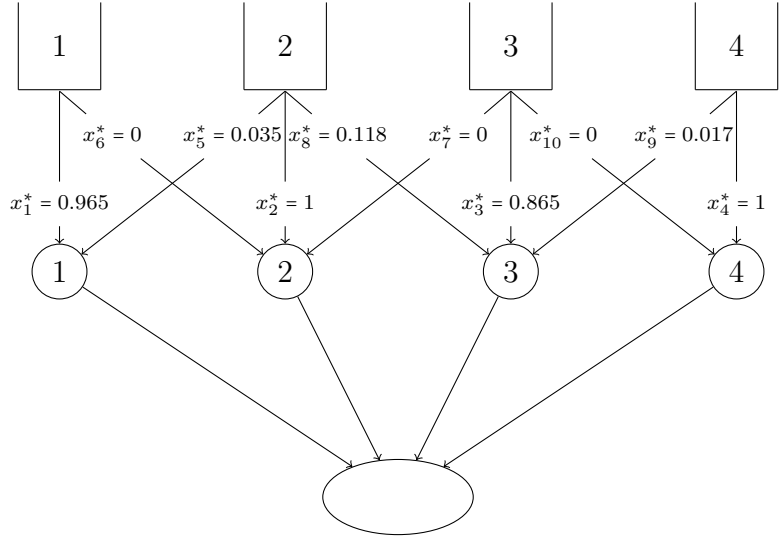
\begin{figure}[h!]
		\centering
        \begin{tikzpicture}
            \node[three sided] (b_1) at (0,3) {$1$}; 
            \node[three sided] (b_2) at (3,3) {$2$}; 
            \node[three sided] (b_3) at (6,3) {$3$}; 
            \node[three sided] (b_4) at (9,3) {$4$}; 
            \node[shape=circle, draw] (s_1) at (0,0) {$1$}; 
            \node[shape=circle, draw] (s_2) at (3,0) {$2$}; 
            \node[shape=circle, draw] (s_3) at (6,0) {$3$}; 
            \node[shape=circle, draw] (s_4) at (9,0) {$4$}; 
            \node[shape=ellipse, minimum width=2cm, minimum height=1cm, draw] (inf_s) at (4.5, -3) {$ $};
            \node (x1) at (0,0.9) {{\scriptsize $x_1^*=0.965$}};
            \node (x2) at (3,0.9) {{\scriptsize $x_2^*=1$}};
            \node (x3) at (6,0.9) {{\scriptsize $x_3^*=0.865$}};
            \node (x4) at (9,0.9) {{\scriptsize $x_4^*=1$}};
            \node (x5) at (2.3,1.8) {{\scriptsize $x_5^*=0.035$}};
            \node (x6) at (0.7,1.8) {{\scriptsize $x_6^*=0$}};
            \node (x7) at (5.3,1.8) {{\scriptsize $x_7^*=0$}};
            \node (x8) at (3.7,1.8) {{\scriptsize $x_8^*=0.118$}};
            \node (x9) at (8.3,1.8) {{\scriptsize $x_9^*=0.017$}};
            \node (x10) at (6.7,1.8) {{\scriptsize $x_{10}^*=0$}};
            \draw[->] (s_1) -- (inf_s);
            \draw[->] (s_2) -- (inf_s);
            \draw[->] (s_3) -- (inf_s);
            \draw[->] (s_4) -- (inf_s);
            \draw[->] ([yshift=-0.6cm]b_1.center) -- (x1) -- (s_1);
            \draw[->] ([yshift=-0.6cm]b_2.center) -- (x2) -- (s_2);
            \draw[->] ([yshift=-0.6cm]b_3.center) -- (x3) -- (s_3);
            \draw[->] ([yshift=-0.6cm]b_4.center) -- (x4) -- (s_4);
            \draw[->] ([yshift=-0.6cm]b_1.center) -- (x6) -- (s_2);
            \draw[->] ([yshift=-0.6cm]b_2.center) -- (x5) -- (s_1);
            \draw[->] ([yshift=-0.6cm]b_2.center) -- (x8) -- (s_3);
            \draw[->] ([yshift=-0.6cm]b_3.center) -- (x7) -- (s_2);
            \draw[->] ([yshift=-0.6cm]b_3.center) -- (x10) -- (s_4);
            \draw[->] ([yshift=-0.6cm]b_4.center) -- (x9) -- (s_3);
        \end{tikzpicture}
    \caption{Unique solution $x^*\in\mathbb{R}^{10}$ to the static problem from Equations (\ref{eq:1.13})--(\ref{eq:1.16}). We see that Activities 6,7, and 10 are nonbasic while the rest are basic.}
    \label{fig:TaxiNetwork_StaticSolution}
	\end{figure}
 Having characterized $x^*$, we next compute the drift parameter $a$ and the variance parameter $\sigma^2$ of the Brownian motion $\chi(\cdot)$, see Equation (\ref{eq:EWF_W}). To this end, first note that the drift vector $\gamma$ and the covariance matrix $\Sigma$ of the Brownian motion $B(\cdot)$ (see Equations (\ref{eq:3.17}), (\ref{eq:DriftCovariance}), and (\ref{eq:EWF})) are given as follows: \begin{align*}
     \gamma &= \hat{\eta}' q = (1.9566, 6.4247,3.2361,0.2625)', \text{ and}   \\
     \Sigma &= \begin{bmatrix}
         0.7097 & 0.1918 & 0.0966 & 0.0078   \\
         0.1918 & 2.3302 & 0.3173 & 0.0257  \\
         0.0966 & 0.3173 & 1.1742 & 0.0130   \\
         0.0078 & 0.0257 & 0.0130 &	0.0937   
     \end{bmatrix}.
 \end{align*}
 Thus, we have that $a=e'\gamma=11.88$ and $\sigma^2=e'\Sigma e=5.6125$. 

 Next, we describe the economic primitives of our example: the demand function, and its associated profit function, the holding cost rates and the cost of idleness. We assume that the demand function is linear. That is, \[\Lambda_i(p_i) = a_i-b_ip_i \quad \text{for}\quad p_i\in [0,\frac{a_i}{b_i}] \text{ and }i=1,\ldots,4,  \] where $a_i,b_i>0$ are constants. Also, its inverse is given by \begin{align*}
	\Lambda_i^{-1}(\lambda_i) =\frac{a_i-\lambda_i}{b_i},\quad \lambda_i\in [0,a_i],\quad i=1,\ldots,4.
\end{align*}
The profit function then follows from Equation (\ref{eq:1.10}) as follows: \begin{align*}
	\pi(\lambda) = \sum_{i=1}^{4} \frac{\lambda_i}{b_i}\left(a_i-\lambda_i\right) ,\quad \lambda_i\in [0, a_i],\quad i=1,\ldots,4.
	\end{align*}
 We set the optimal static price as $p_i^*=10$ for all region $i$, which is about the average price of a ride in the data, see \citet{AtaBarjestehKumar_Empirical}. Also, recall that the limiting demand rate vector $\lambda^*=(\lambda_1^*,\ldots, \lambda_4^*)$ is given by (\ref{eq:LimitingLambda}). We crucially assume that these are the optimal demand rate and the prices. This is equivalent to assuming $a_i=2\lambda_i^*$ and $b^*=\lambda_i^*/p_i$ for $i=1,\ldots,4$. Namely, we set \begin{align*}
     a &= 2\lambda^* = (0.7356,2.1446,1.3584,0.0691)',  \\
     b &= \lambda^*/p^* = (0.0367,0.1072,0.0679,0.0035)'.
 \end{align*} Given these we compute the parameter $\alpha_i$ as $\alpha_i=-(\Lambda_i^{-1})'(\lambda_i^*)-(\lambda_i^*/2)(\Lambda_i^{-1})''(\lambda_i^*)=1/b_i$ for $i=1,\ldots,4$. Thus, we obtain $\alpha=(27.18,9.32,14.72,289.55)$ and $\hat{\alpha}=\sum_{i=1}^4 1/\alpha_i=0.2154$.
    
     \citet{AtaBarjestehKumar_Empirical} suggest that the holding cost when taxis are traveling is $h_0^n=1$ dollars per hour (which can be derived from their fuel cost estimates). To estimate the holding cost rates for other buffers, we consider the driver's opportunity cost. A driver can complete about two trips per hour, resulting in approximately $2\times 10=20$ dollars per hour. Thus, we set $h_i^n=20$ for $i=1,\ldots,4$. Thus, we have $h^n=\min_{i=1,\ldots,4} h_i^n-h_0^n=19$. Upon scaling, we derive the limiting holding cost rate $h$ for the equivalent workload formulation as $h=\sqrt{n} h^n=1900$. The idleness costs parameters are set to equal the lost revenue. That is, $c_i^n=p_i^*=10$ for $i=1,\ldots,4$. Upon rescaling, the limiting idleness cost is $c_i=c_i^n/\sqrt{n}=0.1$. Thus, the cheapest server to idle as $k^*=\arg\min_{i=1,\ldots,4} c_i/\lambda_i^*=2$ with the idling cost $r=c_{k^*}/\lambda_{k^*}^*=0.0933$.

    Having computed the parameters $a, \sigma^2, h, r, \eta$, and $\hat{\alpha}$, we solve the Bellman equation numerically for the example. Using this solution, we next describe our proposed policy.

    \noindent\textbf{Pricing Policy.} It follows from Equation (\ref{eq:138*}) that \begin{align*}
        p_i^n(t) &= 10-\frac{1}{200} v\left(\frac{W^n(t)}{100}\right), \quad i=1,\ldots,4, \quad  t\geq 0.
    \end{align*}
    This corresponds to the following demand rates: \begin{align*}
        \lambda_i^n &= 10000\lambda_i^* +\frac{50}{\alpha_i} v\left(\frac{W^n(t)}{100}\right), \quad i=1,\ldots,4, \quad  t\geq 0.
    \end{align*}
    \noindent\textbf{Dispatch Policy.} As discussed in Section \ref{sec:8}, we propose two dispatch policies. Under the first proposed policy (Dispatch Policy 1), servers 2 and 4 work only on their own buffer throughout. Servers 1 and 3 prioritize their own buffers, but server 1 serves buffer 2 if buffer 1 is empty and buffer 2 exceeds threshold $s$. Similarly, server 3 serves buffers 2 or 4 only if buffer 3 is empty and buffer 2 or 4 exceeds threshold $s$. If both queues exceeds $s$, then server 3 serves the longest one. We determine the threshold $s$ by a brute-force search. In particular, we set $s=1$.  

     Under Dispatch Policy 2, each server prioritizes his own (local) buffer. If his own buffer is empty, then he checks the other buffers that he can serve using basic activities. If there are multiple such buffers, the server works on the buffer with the largest queue length. If the server's own (local) buffer is empty and he cannot serve any other buffers using basic activities, then he considers all remaining buffers he can serve (using nonbasic activities) and works next on the buffer with the largest queue length.

    In order to compare the performance of our policy, we calculate the total revenue by adding up the prices charged to each served customer. This also incorporates the cost of idleness. Also, we keep track of the holding costs incurred. Lastly, we use \[\Tilde{V}^n(t)=\left(n\pi(\lambda^*)-\sqrt{n}h_0\right)t=\left(n\sum_{i=1}^4\frac{\lambda_i^*}{b_i}(a_i-\lambda_i^*)-\sqrt{n}h_0\right)t, \text{ for } t\geq 0 \] (see Equation (\ref{eq:3.18*})) to compute the normalized cost $\hat{V}^n(t)$, see Equation (\ref{eq:44}).

    We compare our policy against the following benchmark policies that combine alternative pricing and dispatch policies. For pricing, in addition to our dynamic pricing policy, we also consider the static pricing policy which sets $p_i^n(t)=p_i^*=10$ for all $i=1,\ldots,4$ and $t\geq 0$. For dispatch, in addition to our two proposed policies, we consider (i) a static dispatch policy, and (ii) the closest driver policy as described next.
	
    \noindent\textbf{Static Dispatch Policy.} 
    Servers 2 and 4 always serve their own buffers. If both buffers 1 and 2 are nonempty, then server 1 works on buffer 1 with probability $x_1^*/(x_1^*+x_5^*)=0.965$ and it works on buffer 2 with probability $x_5^*/(x_1^*+x_5^*)=0.035$. If only one of the buffers 1 and 2 is nonempty, then server 1 works on that buffer. Server 3 splits its effort among buffers 2, 3, and 4 similarly, i.e., proportional to $x_3^*, x_8^*$, and $x_9^*$, respectively. 

    \noindent\textbf{Closest Driver Policy.} 
    We let $D$ be the distance matrix, i.e., $D_{ij}$ corresponds to the distance (in miles) between regions $i$ and $j$ when $i\not = j$ and $D_{ii}=0$. Using the data from \citet{AtaBarjestehKumar_Empirical}, we have \begin{align*}
        D &= \begin{bmatrix}
         0 & 2.6414 & 4.8132 & 8.2689   \\
         2.6414 & 0 & 1.9993 & 6.1969  \\
         4.8132 & 1.9993 & 0 & 3.9073  \\
         8.2689 & 6.1969 & 3.9073 &	0   
     \end{bmatrix}.
    \end{align*} 
    Server $i$ engages in activity $\arg\min_{j\in{\cal A}_i} D_{i b(j)}(t)$ at time $t$. 
    In other words, under the closest driver policy each server prioritizes the buffer that is closest to him.

    The result of the numerical study are given in Table \ref{tab:Simulation}. The simulated results are obtained based on a run-length of 1000 hours and the estimated average cost is computed by excluding the statistics from the first 200 hours warm-up period. The corresponding confidence intervals are calculated based on 10 macro-replications. We observe that the proposed dispatch policies (DP1, DP2) offer significant improvement (9.74\%-55.01\%) over the benchmark policies. More importantly, we observe that dynamic pricing can lead to significant improvement (30.96\%-61.73\%) for every dispatch policy considered. Among the policies considered, the dynamic pricing with Dispatch Policy 2 (DP2) has the best performance.  

\begin{table}[h!]
	\centering
	\caption{Estimated average cost along with the 95\% confidence interval based on 10 macro-replications. }
	\label{tab:Simulation}
	{\begin{tabular}{c||c|c}
			\toprule
			Dispatch policy & Static pricing policy & Dynamic pricing policy   \\  \midrule
			DP1 & 10075.23 $\pm$ 201.59 & 4302.59 $\pm$ 94.09  \\
            DP2 & 10607.19 $\pm$ 103.18 & 4059.35 $\pm$ 73.73 \\
            Static policy & 13066.83 $\pm$ 457.31 & 9021.89 $\pm$ 204.19 \\
            Closest driver policy & 12100.53 $\pm$ 193.57 & 4766.96 $\pm$ 122.19
			\\ \bottomrule
	\end{tabular}}
\end{table}

Unfortunately, we do not have any data to directly estimate the holding costs and the cost of idleness. For the former, the actual holding cost may be lower because the opportunity cost we estimate is likely an upper bound. On the other hand, the latter does not account for the loss of goodwill currently. Therefore, we conduct a sensitivity analysis that considers lower holding cost rates (Figure \ref{fig:HoldingCostSensitivity}) and another one that considers higher cost of idleness that incorporate the loss of goodwill\footnote{The estimated performance and the corresponding confidence interval for the sensitivity analysis is also based on 10 macro-replications where each replication has a run-length of 1000 hours (and the statistics of the first 200 hour are discarded as a warm-up period).} (Figure \ref{fig:IdleCostSensitivity}). 
These collectively show that the insights from Table \ref{tab:Simulation} are robust to changes in holding and idleness cost parameters. 

\begin{figure}[H]
	\centering
     \begin{subfigure}{.5\textwidth}
		\centering
		\includegraphics[width=\linewidth]{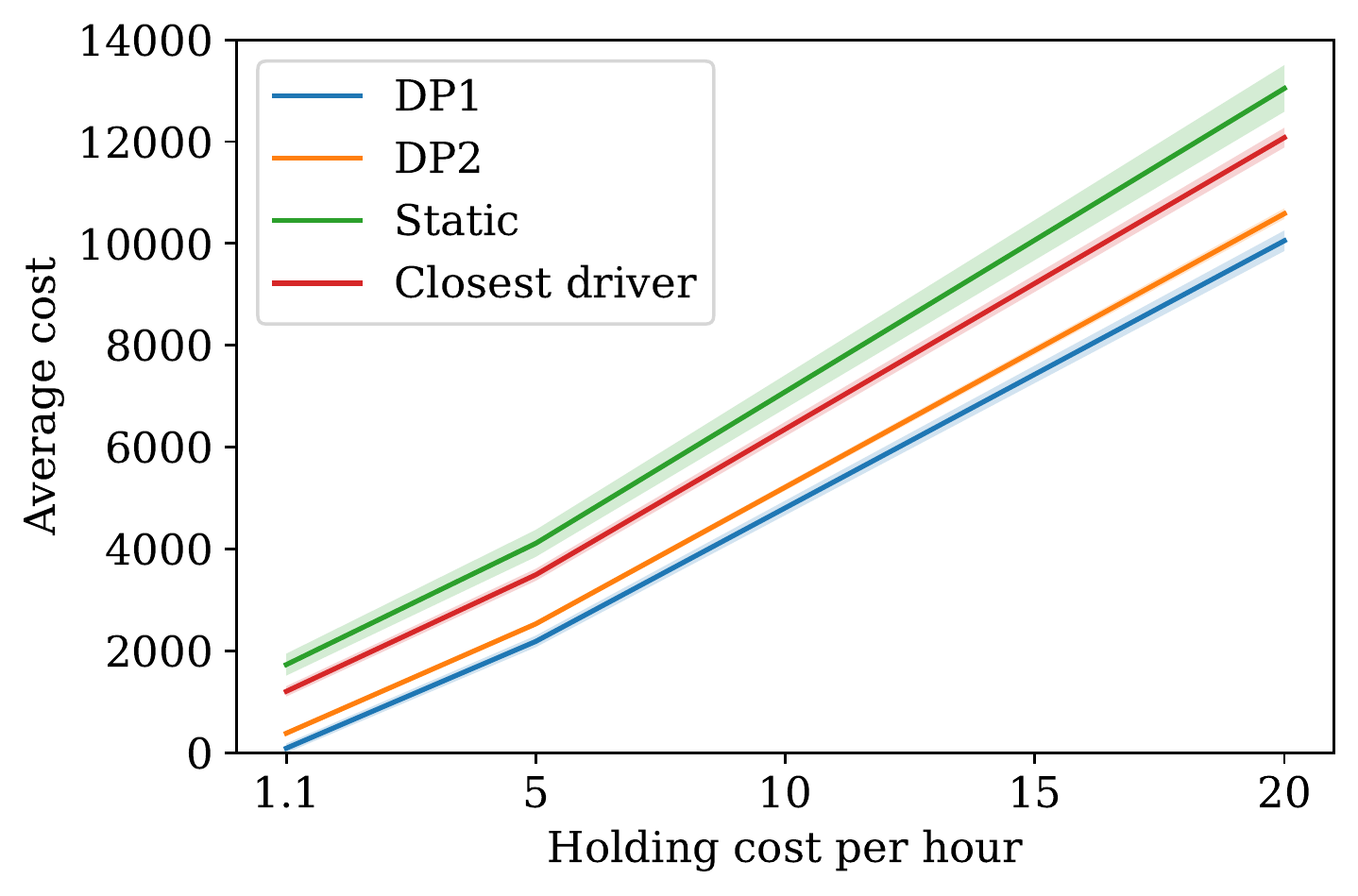}
		\caption{Static pricing}
		\label{fig:sub2}
	\end{subfigure}%
	\begin{subfigure}{.5\textwidth}
		\centering
		\includegraphics[width=\linewidth]{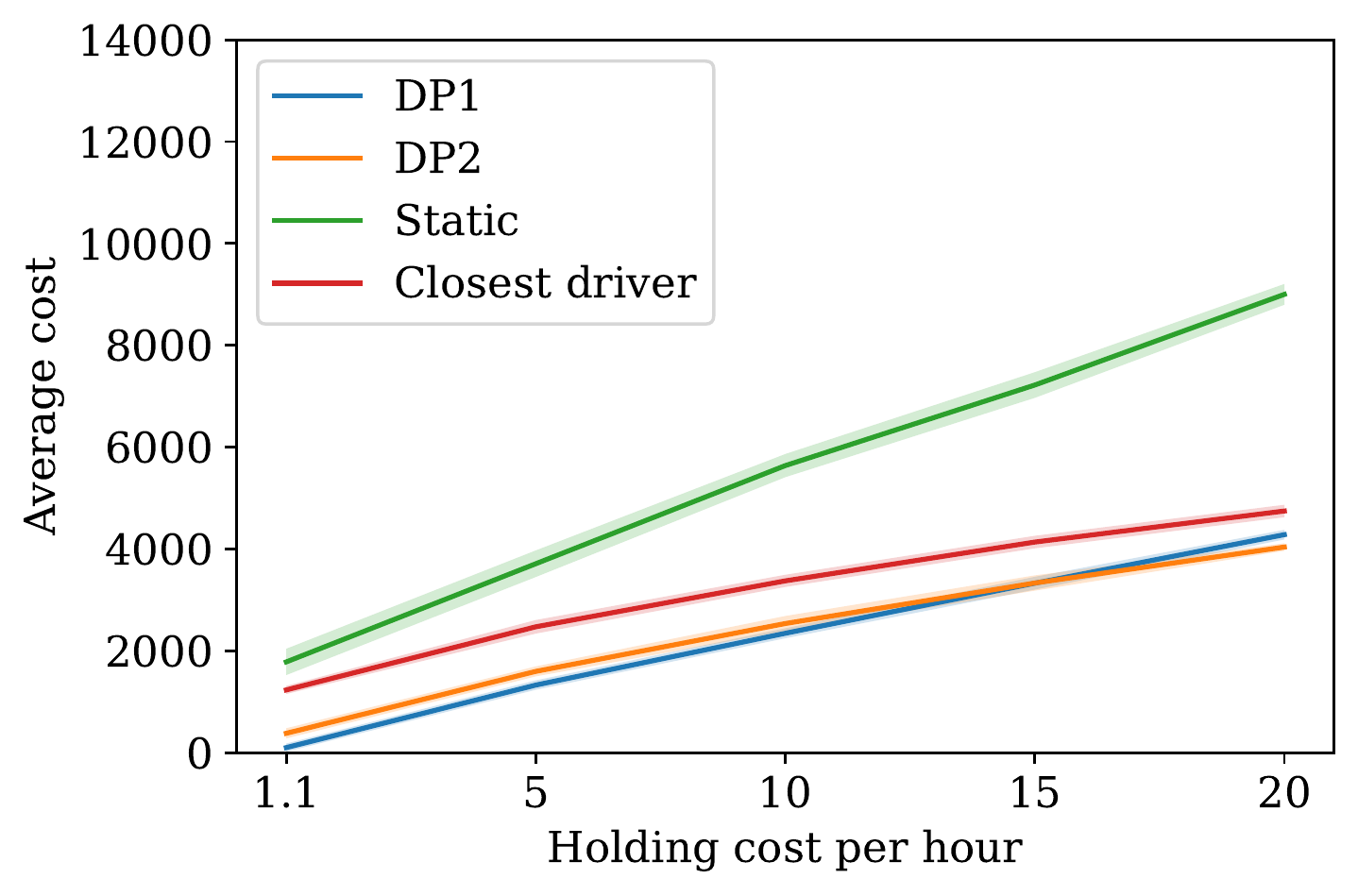}
		\caption{Dynamic pricing}
		\label{fig:sub1}
	\end{subfigure}
	\caption{Average cost with respect to varying holding cost. The shaded area along each line shows the 95\% confidence interval based on 10 macro-replications. }
	\label{fig:HoldingCostSensitivity}
\end{figure}

\begin{figure}[H]
	\centering
    \begin{subfigure}{.5\textwidth}
		\centering
		\includegraphics[width=\linewidth]{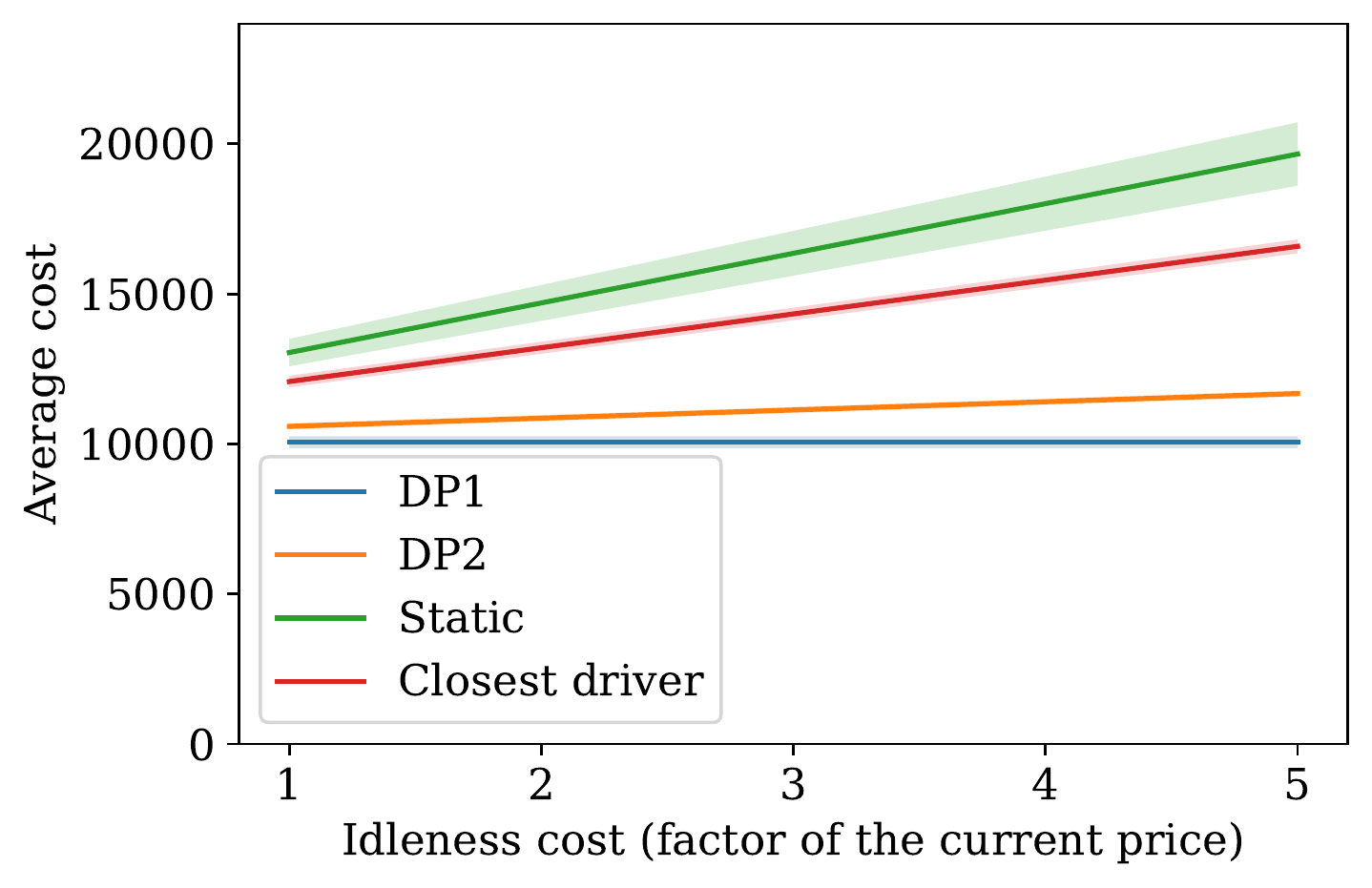}
		\caption{Static pricing}
		\label{fig:sub2}
	\end{subfigure}%
	\begin{subfigure}{.5\textwidth}
		\centering
		\includegraphics[width=\linewidth]{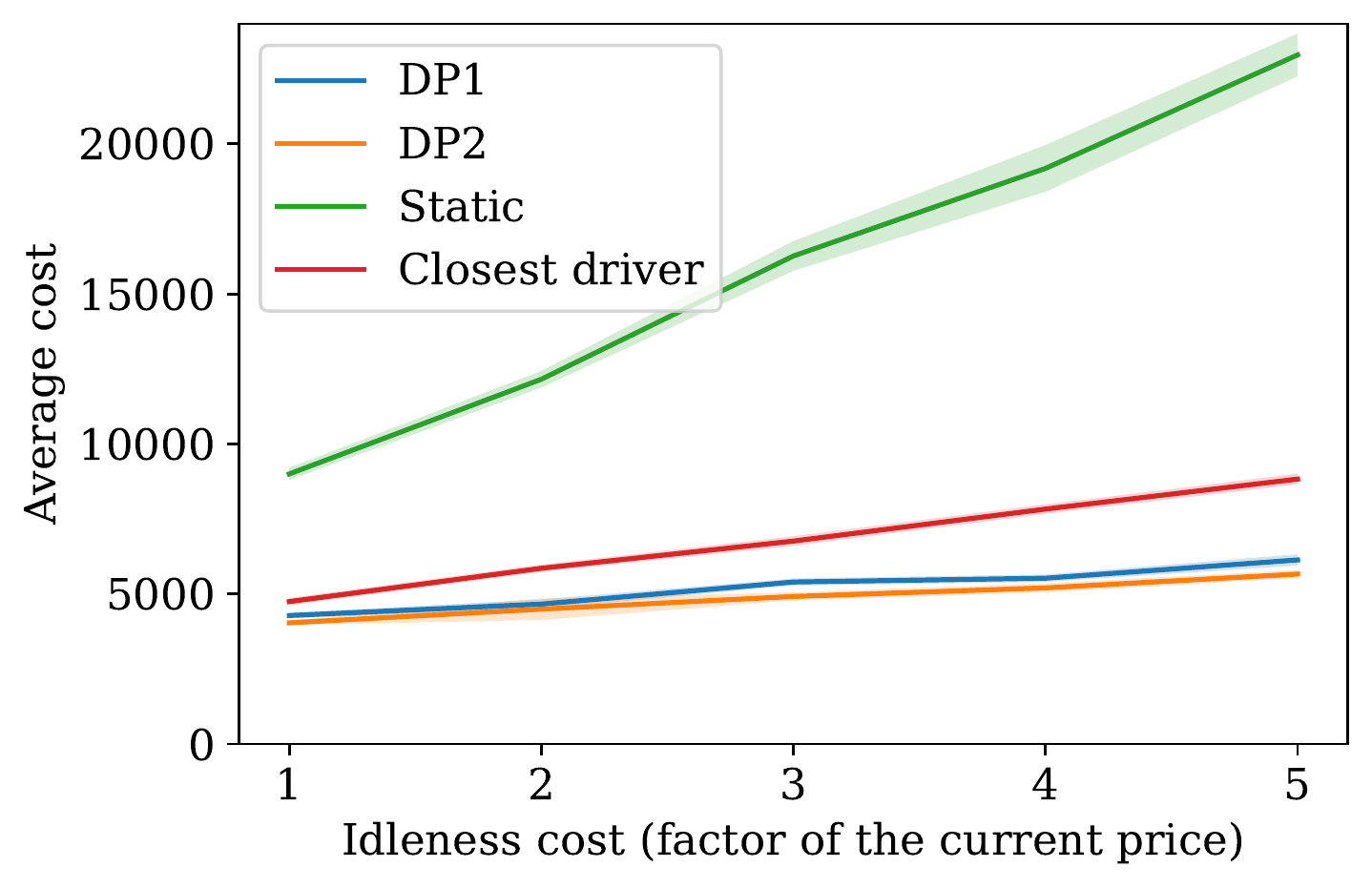}
		\caption{Dynamic pricing}
		\label{fig:sub1}
	\end{subfigure}
	\caption{Average cost with respect to varying idleness cost. The shaded area along each line shows the 95\% confidence interval based on 10 macro-replications.}
	\label{fig:IdleCostSensitivity}
\end{figure}

	\section{Concluding Remarks}\label{sec:10}
	We study a dynamic pricing and dispatch control problem motivated by ride-hailing systems. The novelty of our formulation is that it incorporates travel times. We solve this problem analytically in the heavy traffic regime under the complete resource pooling condition. Using this solution, we propose a closed form dynamic pricing policy as well as a dispatch policy. We compare the proposed policy against benchmarks in a simulation study and show that it is effective.

    Our formulation has some limitations too. Namely, we assume there is only one travel node and that the complete resource pooling condition holds. Interesting future research directions include relaxing these assumptions.

\clearpage
\newpage


\begin{thebibliography}{}
	
	\bibitem[{Abramowitz and Stegun(2003)}]{Abramowitz2003}
	Abramowitz, M. and Stegun, I.A. (2003), ``Handbook Mathematical Functions with Formulas, Graphs, and Mathematical Tables,'' Dover Publications, New York.\vspace{0.5em}
	
	\bibitem[{Adusumilli and Hasenbein(2010)}]{Adusumilli2010}
	Adusumilli, K.M. and Hasenbein, J.J. (2010), ``Dynamic Admission and Service Rate Control of a Queue,'' \textit{Queueing Systems}, \textbf{66} (2) 131--154.\vspace{0.5em}
	
	\bibitem[{Af\`{e}che et~al.(2018)}]{Afeche18}
	Af\`{e}che, P., Liu, Z., and Maglaras, C. (2018), ``Ride-Hailing Networks with Strategic Drivers: The Impact of Platform Control Capabilities on Performance,'' \textit{Working Paper}.\vspace{0.5em}
	
	\bibitem[{Af\`{e}che et~al.(2020)}]{Afeche20}
	Af\`{e}che, P., Liu, Z., and Maglaras, C. (2020), ``Surge Pricing an Dynamic Matching for Hotspot Demand Shock in Ride-hailing Networks,'' \textit{Working Paper}.\vspace{0.5em}
	
	
	\bibitem[{Ata(2005)}]{Ata2005}
	Ata, B. (2005) ``Dynamic Power Control in a Wireless Static Channel Subject to a Quality-of-Service Constraint,'' \textit{Operations Research}, \textbf{53} (5) 842--851.\vspace{0.5em}
	
	\bibitem[{Ata(2006)}]{Ata2006}
	Ata, B. (2006) ``Dynamic Control of a Multiclass Queue with Thin Arrival Streams,'' \textit{Operations Research}, \textbf{54} (5) 876--892.\vspace{0.5em}
	
	\bibitem[{Ata and Barjesteh(2020)}]{AtaBarjesteh2020}
	Ata, B. and Barjesteh, N. (2020), ``Dynamic Pricing of a Multiclass Make-to-Stock Queue,'' \textit{Working Paper}.\vspace{0.5em}
	
	\bibitem[{Ata et~al.(2019)}]{AtaBarjestehKumar_Empirical}
	Ata, B., Barjesteh, N., and Kumar, S. (2019), ``Spatial Pricing: An Empirical Analysis of Taxi Rides in New York City,'' \textit{Working Paper}.\vspace{0.5em}
	
	\bibitem[{Ata et~al.(2020)}]{AtaBarjestehKumar_SPN}
	Ata, B., Barjesteh, N., and Kumar, S. (2020), ``Dynamic Matching and Centralized Relocation in Ridesharing Platforms,'' \textit{Working Paper}.\vspace{0.5em}
	
	\bibitem[{Ata and Kumar(2005)}]{AtaKumar05}
	Ata, B. and Kumar, S. (2005), ``Heavy Traffic Analysis of Open Processing Systems with Complete Resource Pooling: Asymptotic Optimality of Discrete Review Policies,'' \textit{The Annals of Applied Probability}, \textbf{15} (1A) 331--391.\vspace{0.5em}
	
	\bibitem[{Ata et~al.(2005)}]{AtaHarrisonShepp2005}
	Ata, B., Harrison, J.M., and Shepp, L.A. (2005), ``Drift Rate Control of a Brownian Processing System,'' \textit{The Annals of Applied Probability}, \textbf{15} (2) 1145--1160.\vspace{0.5em}
	
	
	\bibitem[{Ata et~al.(2021)}]{Ata_Volunteer2}
	Ata, B., Field, J., Lee, D., and Tongarlak, M.H. (2021), ``A Dynamic Model for Managing Volunteer Engagement,'' \textit{Working Paper}.\vspace{0.5em}
	
	\bibitem[{Ata et~al.(2019)}]{Ata_et_al_2019}
	Ata, B., Lee, D., and S\"{o}nmez, E. (2019), ``Dynamic Volunteer Staffing in Multicrop Gleaning Operations,'' \textit{Operations Research}, \textbf{67} (2) 295--314.\vspace{0.5em}

    \bibitem[{Ata and Lin(2008)}]{AtaLin2008}
	Ata, B. and Lin, W. (2008), ``Heavy traffic analysis of maximum pressure policies for stochastic processing networks with multiple bottlenecks,'' \textit{Queueing System}, \textbf{59} 191-–235.\vspace{0.5em}
	
	\bibitem[{Ata and Olsen(2009)}]{AtaOlsen2009}
	Ata, B. and Olsen, T.L. (2009), ``Near-Optimal Dynamic Lead-Time Quotation and Scheduling Under Convex-Concave Customer Delays,'' \textit{Operations Research}, \textbf{57} (3) 753--768.\vspace{0.5em}
	
	\bibitem[{Ata and Olsen(2013)}]{AtaOlsen2013}
	Ata, B. and Olsen, T.L. (2013), ``Congestion-Based Leadtime Quotation and Pricing for Revenue Maximization with Heterogeneous Customers,'' \textit{Queueing Systems}, \textbf{73} (1) 35--78.\vspace{0.5em}
	
	\bibitem[{Ata and Shneorson(2006)}]{AtaShneorson2006}
	Ata, B. and Shneorson, S. (2006), ``Dynamic Control of an $M/M/1$ Service System with Adjustable Arrival and Service Rates,'' \textit{Management Science}, \textbf{52} (11) 1778--1791.\vspace{0.5em}
	
	\bibitem[{Ata and Tongarlak(2013)}]{AtaTongarlak2013}
	Ata, B. and Tongarlak, M.H. (2013), ``On Scheduling a Multiclass Queue with Abandonments under General Delay Costs,'' \textit{Queueing Systems}, \textbf{74} (1) 65--104.\vspace{0.5em}
	
	\bibitem[{Ata and Zachariadis(2007)}]{AtaZachariadis2007}
	Ata, B. and Zachariadis, K.E. (2007), ``Dynamic Power Control in a Fading Downlink Channel Subject to an Energy Constraint,'' \textit{Queueing Systems}, \textbf{55} (1) 41--69.\vspace{0.5em}
	
	\bibitem[{Banerjee et~al.(2021)}]{Freund}
	Banerjee, S., Freund, D., and Lykouris, T. (2021), ``Pricing and Optimization in Shared Vehicle Systems: An Approximation Framework,'' \textit{Operations Research}, forthcoming.\vspace{0.5em}
	
	\bibitem[{Banerjee et~al.(2019)}]{BanerjeeKanoriaQian2020}
	Banerjee, S., Kanoria, Y., and Qian, P. (2020), ``Dynamic Assignment Control of a Closed Queueing Network under Complete Resource Pooling,'' \textit{Working Paper}.\vspace{0.5em}
	
	\bibitem[{Banerjee et~al.(2015)}]{Banerjee15}		
	Banerjee, S., Riquelme, C., and Johari, R. (2015), ``Pricing in Ride-Sharing Platforms: A Queueing-Theoretic Approach,'' \textit{Proceedings of the Sixteenth ACM Conference on Economics and Computation}, 639--639.\vspace{0.5em}
	
	\bibitem[{Bateman and Erd\'{e}lyi(1953)}]{Bateman1953}
	Bateman, H. and Erd\'{e}lyi, A. (1953), ``Higher Transcendental Functions, Volume I,'' McGraw-Hill, New York.\vspace{0.5em}
	
	\bibitem[{Bell and Williams(2001)}]{BellWilliams2001}
	Bell, S.L. and Williams, R.J. (2001), ``Dynamic Scheduling of a System with Two Parallel Servers in Heavy Traffic with Resource Pooling: Asymptotic Optimality of a Threshold Policy,'' \textit{The Annals of Applied Probability}, \textbf{11} (3) 608--649.\vspace{0.5em}
	
	\bibitem[{Bell and Williams(2005)}]{BellWilliams2005}
	Bell, S.L. and Williams, R.J. (2005), ``Dynamic Scheduling of a Parallel Server System in Heavy Traffic with Complete Resource Pooling: Asymptotic Optimality of a Threshold Policy,'' \textit{Electronic Journal of Probability}, \textbf{10} 1044--1115.\vspace{0.5em}
	
	\bibitem[{Bertsimas et~al.(2019)}]{Bertsimas2019}
	Bertsimas, D., Jaillet, P., and Martin, S. (2019), ``Online Vehicle Routing: The Edge of Optimization in Large-Scale Applications,'' \textit{Operations Research}, \textbf{67} (1) 143--162.\vspace{0.5em}
	
	\bibitem[{Besbes et~al.(2021a)}]{Besbes2021a}
	Besbes, O., Castro, F., and Lobel, I. (2021), ``Surge Pricing and Its Spatial Supply Response,'' \textit{Management Science}, \textbf{67} (3) 1350--1367.\vspace{0.5em}
	
	\bibitem[{Besbes et~al.(2021b)}]{Besbes2021b}
	Besbes, O., Castro, F., and Lobel, I. (2021), ``Spatial Capacity Planning,'' \textit{Operations Research}, \textbf{70} (2) 1271--1291. \vspace{0.5em}
	
	\bibitem[{Billingsley(1999)}]{Billingsley}
	Billingsley, P. (1999), ``Convergence of Probability Measures (Second Edition),'' John Wiley \& Sons, Inc., New York, NY.\vspace{0.5em}
	
	\bibitem[{Bimpikis et~al.(2019)}]{Candogan}
	Bimpikis, K., Candogan, O., and Saban, D. (2019) ``Spatial Pricing in Ride-Sharing Networks,'' \textit{Operations Research}, \textbf{67} (3) 744--769.\vspace{0.5em}
	
	\bibitem[{Bramson and Dai(2001)}]{BramsonDai2001}
	Bramson, M. and Dai, J.G. (2001), ``Heavy Traffic Limits for some Queueing Networks,'' \textit{The Annals of Applied Probability}, \textbf{11} (1) 49--90.\vspace{0.5em}
	
	
	\bibitem[{Braverman et~al.(2019)}]{Braverman}
	Braverman, A., Dai, J.G., Liu, X., and Ying, L. (2019), ``Empty-Car Routing in Ridesharing Systems,'' \textit{Operations Research}, \textbf{67} (5) 1437--1452.\vspace{0.5em}
	
	\bibitem[{Browne and Whitt(1995)}]{BrowneWhitt1995}
	Browne, S. and Whitt, W. (1995), ``Piecewise-Linear Diffusion Processes,'' in \textit{Advances in Queueing: Theory, Methods, and Open Problems}, J.H. Dshalalow (Eds.), 463--480, CRC Press, Boca Raton, FL.\vspace{0.5em}
	
	\bibitem[{Budhiraja and Ghosh(2005)}]{BudhirajaRoss2005}
	Budhiraja, A. and Ghosh, A.P. (2005), ``A Large Deviations Approach to Asymptotically Optimal Control of Crisscross Network in Heavy Traffic,'' \textit{The Annals of Applied Probability}, \textbf{15} (3) 1887--1935.\vspace{0.5em}
	
	
	\bibitem[{Budhiraja et~al.(2016)}]{Budhiraja2016}
	Budhiraja, A., Liu, X., and Saha, S. (2018), ``Construction of Asymptotically Control for Crisscross Network from a Free Boundary Problem,'' \textit{Stochastic Systems}, \textbf{6} (2) 459--518.\vspace{0.5em}
	
	\bibitem[{Cachon et~al.(2017)}]{Cachon}
	Cachon, G., Daniels, K., and Lobel, R. (2017), ``The Role of Surge Pricing on a Service Platform with Self-Scheduling Capacity,'' \textit{Manufacturing \& Service Operations Management}, \textbf{19} (3) 337--507.\vspace{0.5em}
	
	\bibitem[{Castillo et~al.(2021)}]{Castillo}
	Castillo, J.C., Knoepfle, D., and Weyl, G. (2021), ``Matching in Ride Hailing: Wild Goose Chases and How to Solve Them,'' \textit{Working Paper}.\vspace{0.5em}
	
	\bibitem[{\c{C}elik and Maglaras(2008)}]{CelikMaglaras2008}
	\c{C}elik, S. and Maglaras, C. (2008), ``Dynamic Pricing and Lead-Time Quotation for a Multiclass Make-to-Order Queue,'' \textit{Management Science}, \textbf{54} (6) 1132--1146.\vspace{0.5em}

    \bibitem[{Chen et~al.(2020)}]{Chen2020}
	Chen, Q., Lei, Y., and Jasin, S. (2020), ``Real-time spatial-intertemporal dynamic pricing for balancing supply and demand in a network,'' \textit{Working Paper}. \vspace{0.5em}
	
	\bibitem[{Chen et~al.(1994)}]{Chen1994}
	Chen, H., Yang, P., and Yao, D.D. (1994), ``Control and Scheduling in a Two-Station Queueing Network: Optimal Policies and Heuristics,'' \textit{Queueing Systems}, \textbf{18} (3--4) 301--331.\vspace{0.5em}
	
	\bibitem[{Chen and Sheldon(2016)}]{Chen16}
	Chen, M.K. and Sheldon, M. (2016), ``Dynamic Pricing in a Labor Market: Surge Pricing and Flexible Work on the Uber Platform,'' \textit{Proceedings of the 2016 ACM Conference on Economics and Computation}.\vspace{0.5em}
	
%

    \bibitem[{Dai and Lin(2005)}]{DaiLin2005}
	Dai, J. G. and Lin, W. (2005), ``Maximum Pressure Policies in Stochastic Processing Networks,'' \textit{Operations Research}, \textbf{53} (2) 197-218\vspace{0.5em}

    \bibitem[{Dai and Lin(2008)}]{DaiLin2008}
	Dai, J. G. and Lin, W. (2008), ``Asymptotic optimality of maximum pressure policies in stochastic processing networks,'' \textit{The Annals of Applied Probability}, \textbf{18} (6)  2239--2299.\vspace{0.5em}
	
	\bibitem[{Ethier and Kurtz(2005)}]{EthierKurtz05}
	Ethier, S. and Kurtz, T. (2005), ``Markov Processes: Characterization and Convergences,'' John Wiley \& Sons, Inc., New York, NY.\vspace{0.5em}
	
	\bibitem[{Garg and Nazerzadeh(2019)}]{GargNazerzadeh2021}
	Garg, N. and Nazerzadeh, H. (2021), ``Driver Surge Pricing,'' \textit{Management Science}, \textbf{68} (5) 3219--3235.\vspace{0.5em}
	
	\bibitem[{George and Harrison(2001)}]{GeorgeHarrison2001}
	George, J.M. and Harrison, J.M. (2001), ``Dynamic Control of a Queue with Adjustable Service Rate,'' \textit{Operations Research}, \textbf{49} (5) 720--731.\vspace{0.5em}
	
	
	\bibitem[{Gokpinar and Selcuk(2019)}]{Gokpinar}
	Gokpinar, B. and Selcuk, C. (2019), ``The Selection of Prices and Commissions in a Spatial Model of Ride-hailing,'' \textit{Working Paper}.\vspace{0.5em}
	
	\bibitem[{Ghosh and Weerasinghe(2007)}]{Ghosh2007}
	Ghosh, A.P. and Weerasinghe, A.P. (2007), ``Optimal Buffer Size for a Stochastic Processing Network in Heavy Traffic,'' \textit{Queueing Systems}, \textbf{55} (3) 147--159.\vspace{0.5em}
	
	\bibitem[{Ghosh and Weerasinghe(2010)}]{Ghosh2010}
	Ghosh, A.P. and Weerasinghe, A.P. (2010), ``Optimal Buffer Size and Dynamic Rate Control for a Queueing System with Impatient Customers in Heavy Traffic,'' \textit{Stochastic Processes and Their Applications}, \textbf{120} (11) 2103--2141.\vspace{0.5em}
	
	\bibitem[{Guda and Subramanian(2019)}]{Guda2019}
	Guda, H. and Subramanian, U. (2019), ``Your Uber is Arriving: Managing On-Demand Workers Through Surge Pricing, Forecast Communication, and Worker Incentives,'' \textit{Management Science}, \textbf{65} (5) 1995--2014.\vspace{0.5em}
	
	\bibitem[{Harrison(1988)}]{Harrison1988}
	Harrison, J.M. (1988), ``Brownian Models of Queueing Networks with Heterogeneous Customer Populations,'' in \textit{Stochastic Differential Systems, Stochastic Control Theory and Applications}, W. Fleming and P.-L. Lions (Eds.), IMA Volumes in Mathematics and its Applications, \textbf{10} 147--186, Springer-Verlag, New York, NY.\vspace{0.5em}
	
	\bibitem[{Harrison(1996)}]{Harrison1996}
	Harrison, J.M. (1996), ``The BIGSTEP Approach to Flow Management in Stochastic Processing Networks,'' in \textit{Stochastic Networks: Theory and Applications}, F. P. Kelly, I. Ziedins and S. Zachary (Eds.), 57--90, Oxford University Press.\vspace{0.5em}
	
	\bibitem[{Harrison(1998)}]{Harrison1998}
	Harrison, J.M. (1998), ``Heavy Traffic Analysis of a System with Parallel Servers: Asymptotic Optimality of Discrete-Review Policies,'' \textit{The Annals of Applied Probability}, \textbf{8} (3) 822--848.\vspace{0.5em}
	
	\bibitem[{Harrison(2000)}]{Harrison2000}
	Harrison, J.M. (2000), ``Brownian Models of Open Processing Networks: Canonical Representation of Workload,'' \textit{The Annals of Applied Probability}, \textbf{10} (1) 75--103.\vspace{0.5em}
	
	\bibitem[{Harrison(2003)}]{Harrison2003}
	Harrison, J.M. (2003), ``A Broader View of Brownian Networks,'' \textit{The Annals of Applied Probability}, \textbf{13} (3) 1119--1150.\vspace{0.5em}
	
	\bibitem[{Harrison(2013)}]{Harrison2013}
	Harrison, J.M. (2013), ``Brownian Models of Performance and Control,'' Cambridge University Press, Cambridge, UK.\vspace{0.5em}
	
	\bibitem[{Harrison and Wein(1989)}]{HarrisonWein1989}
	Harrison, J.M. and Wein, L.M. (1989), ``Scheduling Networks of Queues: Heavy Traffic Analysis of a Simple Open Network,'' \textit{Queueing Systems}, \textbf{5} (4) 265--280.\vspace{0.5em}
	
	
	\bibitem[{Harrison and Van Mieghem(1997)}]{HarrisonVanMieghem1997}
	Harrison, J.M. and Van Mieghem, J.A. (1997), ``Dynamic Control of Brownian Networks: State Space Collapse and Equivalent Workload Formulation,'' \textit{The Annals of Applied Probability}, \textbf{7} (3) 747--771.\vspace{0.5em}
	
	\bibitem[{Harrison and L\'{o}pez(1999)}]{HarrisonLopez1999}
	Harrison, J.M. and L\'{o}pez, M.J. (1999), ``Heavy Traffic Resource Pooling in Parallel-Server Systems,'' \textit{Queueing Systems}, \textbf{33} (4) 339--368.\vspace{0.5em}
	
	\bibitem[{He et~al.(2020)}]{HeHuZhang2020}
	He, L., Hu, Z., and Zhang, M. (2020), ``Robust Repositioning for Vehicle Sharing,'' \textit{Manufacturing \& Service Operations Management}, \textbf{22} (2) 241--256.\vspace{0.5em}
	
	\bibitem[{Hosseini et~al.(2021)}]{Hosseini2021}
	Hosseini, M., Milner, J., and Romero, G. (2021), ``Dynamic Relocations in Car-Sharing Networks,'' \textit{Working Paper}.\vspace{0.5em}
	
	\bibitem[{Hu et~al.(2022)}]{HuHuZhu2022}
	Hu, B., Hu, M., and Zhu, H. (2022), ``Surge Pricing and Two-Sided Temporal Responses in Ride-Hailing,'' \textit{Management \& Service Operations Management}, \textbf{24} (1) 91--109.\vspace{0.5em}
	
	\bibitem[{Hu and Zhou(2021)}]{HuZhou2021}
	Hu, M. and Zhou, Y. (2021), ``Dynamic Type Matching,'' \textit{Management \& Service Operations Management}, \textbf{24} (1) 125--142.\vspace{0.5em}

    \bibitem[{Jacob, J. and Roet-Green, R(2021)}]{JacobRoetGreen2021}
	Jacob, J. and Roet-Green, R. (2021), ``Ride solo or pool: Designing Price-service Menus for a Ride-sharing
    Platform,'' \textit{European Journal of Operations Research}, \textbf{295} (3) 1008–-1024. \vspace{0.5em}
	
	\bibitem[{Karlin and Taylor(1981)}]{KarlinTaylor1981}
	Karlin, S. and Taylor, H.M. (1981), ``A Second Course in Stochastic Processes,'' Academic Press, New York.\vspace{0.5em}
	
	
	\bibitem[{Kim and Ward(2013)}]{KimWard2013}
	Kim, J. and Ward, A.R. (2013), ``Dynamic Scheduling of a $GI/GI/1+GI$ Queue with Multiple Customer Classes,'' \textit{Queueing Systems}, \textbf{75} (2--4) 339--384.\vspace{0.5em}
	
	\bibitem[{Kogan and Lipster(1993)}]{KoganLipster}
	Kogan, Y. and Lipster, R. (1993), ``Limit Non-Stationary Behavior of Large Closed Queueing Networks with Bottlenecks,'' \textit{Queueing Systems}, \textbf{14} (1--2) 33--55.\vspace{0.5em}
	
	\bibitem[{Kogan et~al.(1986)}]{KoganLipsterSmorodinskii}
	Kogan, Y., Liptser, R., and Smorodinskii, A.V. (1986), ``Gaussian Diffusion Approximation of Closed Markov Models of Computer Networks,'' \textit{Problems of Information Transmission}, \textbf{22} (1), 38--51.\vspace{0.5em}
	
	\bibitem[{Korolko et~al.(2020)}]{Korolko2020}
	Korolko, N., Woodard, D., Yan, C., and Zhu, H. (2020), ``Dynamic Pricing and Matching in Ride-Hailing Platforms,'' \textit{Naval Research Logistics}, \textbf{67} (8) 705--724.\vspace{0.5em}
	
	\bibitem[{Krichagina and Puhalskii(1997)}]{Krichagina}
	Krichagina, A.A. and Puhalskii, E.V. (1997), ``A Heavy-Traffic Analysis of a Closed Queueing System with a $GI/\infty$ Service Center,'' \textit{Queueing Systems}, \textbf{25} (1--4), 235--280.\vspace{0.5em}
	
	\bibitem[{Kumar(2000)}]{Kumar2000}
	Kumar, S. (2000), ``Two-Server Closed Networks in Heavy Traffic: Diffusion Limits and Asymptotic Optimality,'' \textit{The Annals of Applied Probability}, \textbf{10} (3) 930--961.\vspace{0.5em}
	
	\bibitem[{Kumar et~al.(2013)}]{Kumar_et_al_2013}
	Kumar, R., Lewis, M.E., and Topaloglu, H. (2013), ``Dynamic Service Rate Control for a Single-Server Queue with Markov-Modulated Arrivals,'' \textit{Naval Research Logistics}, \textbf{60} (8) 661--677.\vspace{0.5em}
	
	\bibitem[{Kushner and Martins(1996)}]{KushnerMartins1996}
	Kushner, H.J. and Martins, L.F. (1996), ``Heavy Traffic Analysis of a Controlled Multiclass Queueing Network via Weak Convergence Methods,'' \textit{SIAM J. Control and Optimization}, \textbf{34} (5) 1781--1797.\vspace{0.5em}
	
	
	\bibitem[{Lu et~al.(2018)}]{Lu2018}
	Lu, A., Frazier, P., and Kislev, O. (2018), ``Surge Pricing Moves Uber's Driver Partners,'' \textit{Proceedings of the 2018 ACM Conference on Economics and Computation}.\vspace{0.5em}
	
	\bibitem[{Mandl(1968)}]{Mandl1968}
	Mandl, P. (1968), ``Analytic Treatment of One-Dimensional Markov Processes,'' Springer-Verlag, New York.\vspace{0.5em}
	
	\bibitem[{Martins et~al.(1996)}]{Martins1996}
	Martins, L.F., Shreve, S.E., and Soner, H.M. (1996), ``Heavy Traffic Convergence of a Controlled Multiclass Queueing System,'' \textit{SIAM J. Control and Optimization}, \textbf{34} (6) 2133--2171.\vspace{0.5em}
	
	\bibitem[{\"{O}zkan(2020)}]{Ozkan2020}
	\"{O}zkan, E. (2020), ``Joint Pricing and Matching in Ride-Sharing Systems,'' \textit{European Journal of Operational Research}, \textbf{287} (3) 1149--1160.\vspace{0.5em}
	
	\bibitem[{\"{O}zkan and Ward(2020)}]{OzkanWard2020}
	\"{O}zkan, E. and Ward, A.R. (2020), ``Dynamic Matching for Real-time Ridesharing,'' \textit{Stochastic Systems}, \textbf{10} (1) 29--70.\vspace{0.5em}
	
	\bibitem[{Polyanin and Zaitsev(2003)}]{PolyaninZaitsev2003}
	Polyanin, A.D. and Zaitsev, V.F. (2003), ``Handbook of Exact Solutions for Ordinary Differential Equations (Second Edition),'' Chapman \& Hall/CRC, Boca Raton, FL.\vspace{0.5em}

    \bibitem[{Rubino and Ata(2009)}]{AtaRubino2009}
	Rubino, M. and Ata, B. (2009), ``Dynamic Control of a Make-to-Order, Parallel-Server System with Cancellations,'' \textit{Operations Research}, \textbf{57} (1) 94--108.\vspace{0.5em}
	
	\bibitem[{Smorodinskii(1986)}]{Smorodinskii}
	Smorodinskii, A.V. (1986), ``Asymptotic Distribution of the Queue Length of One Service System'' (in Russian), \textit{Avtomatika i Telemekhanika}, \textbf{2} 92-99.\vspace{0.5em}
	
	\bibitem[{Stidham and Weber(1989)}]{StidhamWeber1989}
	Stidham, S. and Weber, R.R. (1989), ``Monotonic and Insensitive Optimal Policies for Control of Queues with Undiscounted Costs,'' \textit{Operations Research}, \textbf{37} (4) 611--625.\vspace{0.5em}
	
	\bibitem[{Stolyar(2004)}]{Stolyar2004}
	Stolyar, A.L. (2004), ``MaxWeight Scheduling in a Generalized Switch: State Space Collapse and Workload Minimization in Heavy Traffic,'' \textit{The Annals of Applied Probability}, \textbf{14} (1) 1--53.\vspace{0.5em}
	
	
	\bibitem[{Talluri and van Ryzin(2004)}]{TalluriVanRyzin2004}
	Talluri, K. and van Ryzin, G. (2004), ``The Theory and Practice of Revenue Management,'' Springer, New York, NY.

    \bibitem[{Varma et~al.(2022)}]{Varma2022}
	Varma, S. M., Bumpensanti, P., Maguluri, S. T., and Wang, H. (2022), ``Dynamic Pricing and Matching for
    Two-sided Queues,'' \textit{Operations Research}. \vspace{0.5em} 

    \bibitem[{Wang et~al.(2017)}]{Wang17}
	Wang, X., Agatz, N., and Erera, A. (2017), ``Stable Matching for Dynamic Ride-Sharing Systems,'' \textit{Transportation Science}, \textbf{52} (4) 850--867.\vspace{0.5em}
	
	\bibitem[{Williams(1998)}]{Williams1998}
	Williams, R.J. (1998), ``Diffusion Approximations for Open Multiclass Queueing Networks: Sufficient Conditions involving State Space Collapse,'' \textit{Queueing Systems}, \textbf{30} (1--2) 27--88.\vspace{0.5em}

    \bibitem[{Yang et~at.(2018)}]{Yang2018}
	Yang, P., Iyer, K., and Frazier, P. (2018), ``Mean Field Equilibria for Resource Competition in Spatial Settings,'' \textit{Stochastic Systems}, \textbf{8} (4) 307--334.\vspace{0.5em}

    \bibitem[{Zhang and Pavone(2016)}]{ZhangPavone2016}
	Zhang, R. and Pavone, M. (2016), ``Control of Robotic Mobility-On-Demand Systems: A Queueing-Theoretical Perspective,'' \textit{The International Journal of Robotics Research}, \textbf{35} (1-3) 186--203.\vspace{0.5em}
	
\end{thebibliography}

\singlespacing

	\newpage
	\doublespacing
	\appendix
	\section*{Appendices}
	\section{Derivations}\label{app:A}
	\subsection{Formal Derivation of the Brownian Control Problem}
	 This section provides a formal derivation of the approximating Brownian control problem introduced in Section \ref{sec:4}. We do not provide a rigorous weak convergence limit theorem. However, the arguments given in support of the approximation can be viewed as a broad outline for such a proof; see \citet{Harrison1988,Harrison2000,Harrison2003} for similar derivations.
	
	We consider a sequence of systems indexed by the system parameter $n$ under the heavy traffic assumption. Then we center the various processes by their mean, scale them appropriately by the system parameter $n$, and finally pass to the limit as $n\rightarrow\infty$ formally. To that end, we first define the following (diffusion) scaled processes:
	\begin{alignat}{2}
	\hat{\Psi}_i^n (t) &= \frac{1}{\sqrt{n}}\left( \Psi_i\left(\lfloor nt\rfloor\right) - q_i nt\right), &&\quad t\ge 0 ,\quad i=1,\dots, I,\label{eq:6.2}\\
	\hat{N}_j^n(t) &= \frac{1}{\sqrt{n}}\left(N_j(nt) - nt\right), && \quad t\ge 0,\quad j=0,1,\dots, J,\label{eq:6.1}
	\end{alignat}
	where $\lfloor x\rfloor$ is the greatest integer less than or equal to $x$. We also define the following (fluid) scaled processes:
	\begin{alignat}{2}
	\bar{N}_0^n(t) &= \frac{1}{n}N_0(nt),\quad &&t\ge 0,\label{eq:6.3}\\
	\bar{Q}_0^n(t)&= \frac{1}{n}Q_0^n(t),\quad &&t\ge 0,\label{eq:6.4}\\
	\bar{\mu}_j^n(t) &= \frac{1}{n}\mu_j^n(t),\quad &&j=1,\dots, J,\quad t\ge 0.\label{eq:6.5}
	\end{alignat}
	By Donsker's theorem, the functional central limit theorem for renewal processes, and independence of the stochastic primitives, the processes $\hat{\Psi}_i^n$ and $\hat{N}_j^n$ converge weakly to independent standard Brownian motions, see \citet{Billingsley}. 
	
	As observed in \citet{KoganLipster}, under the heavy traffic assumption, we expect that the number of jobs in the infinite-server node will be $n$ to a first-order approximation. That is, we expect that $\bar{Q}^n_0(t)\approx 1$ for $t\ge 0$ as $n$ gets large. Similarly, we expect the queue lengths at buffers $1,\dots, I$ to be of order $\sqrt{n}$. As such, we expect the prices, or equivalently, the demand rates, to deviate from their nominal values only in the second order. That is, we expect $\lambda_i^n-\lambda^*_i n = O\left(\sqrt{n}\right)$. Because the demand rates determine the service rates (see Equation (\ref{eq:9})), we expect that $\bar{\mu}_j^n(t)\approx \mu_j^*$ for $t\ge 0$ as $n$ gets large.
	
	By combining Equations (\ref{eq:6.2})--(\ref{eq:6.5}) with Equations (\ref{eq:TravelTime})--(\ref{eq:42}), it is straightforward to derive the following scaled system dynamics equations for $i=1,\dots, I$:
	\begin{align}
	Z_i^n(t) 
	&= B_i^n(t) +q_i \eta^n\int_{0}^{t}Z_0^n(s)\,ds-\sum_{j\in\mathcal{C}_i}\int_{0}^{t}\kappa_j^n(s)\,dT_j^n(s)+ \sum_{j\in\mathcal{C}_i}\mu_j^* Y_j^n(t) + t\sqrt{n}\left[q_i\eta - \sum_{j\in\mathcal{C}_i}\mu_j^* x_j^*\right]\notag\\
	&= B_i^n(t) +q_i\eta^n\int_{0}^{t}Z_0^n(s)\,ds-\sum_{j\in\mathcal{C}_i}\int_{0}^{t}\kappa_j^n(s)\,dT_j^n(s)+ \sum_{j\in\mathcal{C}_i}\mu_j^* Y_j^n(t),\notag 
	\end{align}
	where the second equality holds by Assumption \ref{ass:3} and where the process $B_i^n$ is given by
	\begin{align*}
	B_i^n(t) &= Z_i^n(0) +q_i\hat{\eta}t + q_i \hat{N}_0^n\left(\eta^n\int_{0}^{t}\bar{Q}^n_0(s)\,ds\right) + \hat{\Psi}_i^n \left(\bar{N}^n_0\left(\eta^n\int_{0}^{t}\bar{Q}^n_0(s)\,ds\right)\right) \\ 
    &\quad -\sum_{j\in\mathcal{C}_i}\hat{N}_j^n\left(\int_{0}^{t}\bar{\mu}_j^n(s)\,dT^n_j(s)\right).
	\end{align*}
	Assuming that $Z_i^n(0)\approx Z_i(0)$ for large $n$, it is also straightforward to argue that $B_i^n$ can be approximated by a Brownian motion $B_i$ with starting state $Z_i(0)$ that has drift parameter $\gamma_i=\hat{\eta}q_i$ and variance parameter
	\begin{align*}
	\sigma_i^2 = \left[q_i^2 + q_i\left(1-q_i\right)\right]\eta + \sum_{j\in\mathcal{C}_i} \mu_j^*x_j^*
	= q_i \eta + \sum_{j\in\mathcal{C}_i} \mu_j^*x_j^*.
	\end{align*}
	Furthermore, the covariance between the limiting Brownian motion processes is given by
	\begin{align*}
	\text{Cov}\left(B_i,B_{i^{\prime}}\right) = q_i q_{i^{\prime}}\eta\quad \text{for}\quad i\ne i^{\prime}.
	\end{align*}Therefore, replacing $Z^n$, $Y^n$, and $\kappa^n$, by their formal limits $Z$, $Y$, and $\kappa$, we arrive at the following system dynamics equations in the approximating Brownian control problem for $i=1,\dots, I$:
	\begin{align*}
	Z_i(t)= B_i(t) +q_i\eta\int_{0}^{t}Z_0(s)\,ds-\sum_{j\in\mathcal{C}_i}\int_{0}^{t}x_j^*\kappa_j(s)\,ds+ \sum_{j\in\mathcal{C}_i}\mu_j^* Y_j(t),\quad t\ge 0.
	\end{align*}
	Equations (\ref{eq:2.11}) and (\ref{eq:38}) of the system state also imply that $Z_0^n(t) = -\sum_{i=1}^{I}Z_i^n(t)$ and that $Z_i^n(t)\ge 0$ for $i=1,\dots, I$ and $t\ge 0$. Thus, in the approximating BCP, the following relationships hold for $t\ge 0$:
	\begin{align*}
	Z_0(t) = -\sum_{i=1}^{I}Z_i(t)\quad \text{and}\quad Z_i(t)\ge 0\quad \text{for}\quad i=1\dots, I.
	\end{align*}
	Similarly, it is clear that Equations (\ref{eq:9}) and (\ref{eq:41*})--(\ref{eq:42}) give rise to Equation (\ref{eq:3.22}) in the BCP; Equations (\ref{eq:2.9}) and (\ref{eq:40}) give rise to Equation (\ref{eq:3.21}); and Equation (\ref{eq:41}) gives rise to Equation (\ref{eq:3.20}).
		
	To complete the formal derivation of the Brownian control problem, we argue that $\hat{V}^n\approx\xi$ for large $n$, where $\hat{V}^n$ and $\xi$ are given by Equations (\ref{eq:44}) and (\ref{eq:45}), respectively. First, observe that by Taylor's theorem we have
	\begin{align*}
	\pi\left(\lambda^* + \frac{1}{\sqrt{n}}\zeta^n(s)\right) &= \pi\left(\lambda^*\right) + \nabla \pi\left(\lambda^*\right)^{\prime}\frac{1}{\sqrt{n}}\zeta^n(s) + \frac{1}{2n}\zeta^n(s)^{\prime}\nabla^2\pi\left(\lambda^*\right)\zeta^n(s)+ R_{\lambda^*, 3}\left(\frac{1}{\sqrt{n}}\zeta^n(s)\right),
	\end{align*}
	where $R_{\lambda^*, 3}\left(\frac{1}{\sqrt{n}}\zeta^n(s)\right)=O(n^{-3/2})$ is a third-order remainder term.\footnote{In particular, the remainder term is given by \begin{align*}
		R_{\lambda^*, 3}\left(\frac{1}{\sqrt{n}}\zeta^n(s)\right)=\sum_{\substack{\alpha_1,\dots, \alpha_I\in \left\{0,1,2,3\right\}\\ \text{s.t. }\alpha_1+\cdots+ \alpha_I = 3}} \frac{\partial^3 \pi\left(\lambda^*  + \frac{C}{\sqrt{n}}\zeta^n(s)\right)}{\partial x_1^{\alpha_1}\partial x_2^{\alpha_2}\cdots \partial x_I^{\alpha_I}} \prod_{i=1}^{I}\frac{\left(\frac{1}{\sqrt{n}}\zeta_i^n(s)\right)^{\alpha_i}}{\alpha_i!}\quad\text{for some}\quad C\in (0,1).
		\end{align*}} 
	Moreover, note that the term $\nabla \pi\left(\lambda^*\right)^{\prime}\zeta^n(s)/\sqrt{n}$ vanishes because $\lambda^*$ is a maximizer of $\pi\left(\lambda\right)$ and is in the interior of the feasible region $\mathcal{L}$ (see Assumption \ref{ass:2}), implying that $\nabla\pi\left(\lambda^*\right)=0$. Therefore, we have that
	\begin{align*}
	\pi\left(\lambda^* + \frac{1}{\sqrt{n}}\zeta^n(s)\right)&=\pi\left(\lambda^*\right) - \frac{1}{n}\zeta^n(s)^{\prime}H\zeta^n(s)+ O\left(n^{-3/2}\right),
	\end{align*}
	where $H=-\frac{1}{2}\nabla^2\pi\left(\lambda^*\right)$.
	Using this and Equations (\ref{eq:3.4*}) and (\ref{eq:41*}), it follows that
	\begin{align}
	\pi^n\left(\lambda^n(s)\right)= n\pi\left(\lambda^*\right) -\zeta^n(s)^{\prime}H\zeta^n(s) + O\left(n^{-1/2}\right).\label{eq:8.27}
	\end{align}
	Finally, using Equations (\ref{eq:38}), (\ref{eq:40})--(\ref{eq:44}), and (\ref{eq:8.27}), it is straightforward to derive the following:
	\begin{align}
	\hat{V}^n(t) &= n\left(\pi\left(\lambda^*\right) - h_0^n\right)t- \left[\int_{0}^{t}\pi^n\left(\lambda^n(s)\right)\,ds-\int_{0}^{t}\sum_{i=0}^{I}h_i^nQ_i^n(s)\,ds - \left(c^n\right)^{\prime} I^n(t)\right]\notag\\
	&=\int_{0}^{t}\left[\zeta^n(s)^{\prime}H\zeta^n(s)+O\left(n^{-1/2}\right)\right]\,ds + \int_{0}^{t}\sum_{i=0}^{I}h_i Z_i^n(s)\,ds + c^{\prime}U^n(t).\notag
	\end{align}
	Therefore, replacing $\hat{V}^n$, $Z^n$, $\zeta^n$, and $U^n$ by their formal limits $\xi$, $Z$, $\zeta$, and $U$, we arrive at the following cost process of the approximating Brownian control problem:
	\begin{align*}
	\xi(t)= \int_{0}^{t}\zeta(s)^{\prime}H\zeta(s)\,ds + \int_{0}^{t}\sum_{i=0}^{I}h_i Z_i(s)\,ds + c^{\prime}U(t),\quad t\ge 0.
	\end{align*}
Note that it is a diagonal matrix, i.e., $H={\rm diag}(\alpha_1,\ldots,\alpha_I)$ where \[\alpha_i=-\left(\Lambda_i^{-1}\right)(\lambda_i^\ast)-\frac{\lambda_i^\ast}{2}\left(\Lambda_i^{-1}\right)'' (\lambda_i^\ast), \quad i=1,\ldots, I.\]
Using this we further simplify the limiting cost process $\xi(t)$ as follows: \[\xi(t) = \int_0^t \left[\sum_{i=1}^I \alpha_i \zeta_i^2(s)+\sum_{i=0}^I h_i Z_i(s) \right] ds + c' U(t), \quad t\geq 0. \]
 
	\subsection{Derivation of Equation (\ref{eq:138*})}
	Recall that the proposed chosen demand rates are 
	$$\lambda_i^n = n\lambda_i^* + \sqrt{n}q(t),\quad i=1,\dots, I,\quad t\ge 0,$$
	where $q(t) = \frac{1}{2\alpha_i}v\left(\frac{W^n(t)}{\sqrt{n}}\right)$. Therefore, the proposed pricing policy for region $i$ is given as follows:
	\begin{align*}
	p_i^n(t) &= \left(\Lambda_i^n\right)^{-1}\left(\lambda_i^n(t)\right) \\&= \left(\Lambda_i^n\right)^{-1}\left(n\lambda_i^* + \sqrt{n}q(t)\right) \\&= \Lambda_i^{-1}\left(\lambda_i^* + \frac{q(t)}{\sqrt{n}}\right),
	\end{align*}
	where the third equality follows from the fact that $\left(\Lambda_i^n\right)^{-1}(nx) = \Lambda_i^{-1}(x)$ for $x\in\mathcal{L}$. Then note that by Taylor's theorem we have
	\begin{align*}
	p_i^n(t) = \Lambda_i^{-1}\left(\lambda_i^*\right) + \left(\Lambda_i^{-1}\right)^{\prime}\left(\lambda_i^*\right)\frac{q(t)}{\sqrt{n}}+ \frac{1}{2}\left(\Lambda_i^{-1}\right)^{\prime\prime}\left(\lambda_i^* + \frac{c\cdot q(t)}{\sqrt{n}}\right)\left(\frac{q(t)}{\sqrt{n}}\right)^2\quad\text{for some}\quad c\in\left(0,1\right),
	\end{align*}
	which implies that
	\begin{align*}
	p_i^n(t) &= \Lambda_i^{-1}\left(\lambda_i^*\right) + \left(\Lambda_i^{-1}\right)^{\prime}\left(\lambda_i^*\right)\frac{q(t)}{\sqrt{n}} + O\left(\frac{1}{n}\right)\\
	&=\Lambda_i^{-1}\left(\lambda_i^*\right) + \frac{\left(\Lambda_i^{-1}\right)^{\prime}\left(\lambda_i^*\right)}{2\alpha_i\sqrt{n}}v\left(\frac{W^n(t)}{\sqrt{n}}\right) + O\left(\frac{1}{n}\right).
	\end{align*}
	As an aside, observe that by rearranging terms we have
	\begin{align*}
	\sqrt{n}\left(p_i^n(t) - \Lambda_i^{-1}\left(\lambda_i^*\right)\right) = \frac{\left(\Lambda_i^{-1}\right)^{\prime}\left(\lambda_i^*\right)}{2\alpha_i}v\left(\frac{W^n(t)}{\sqrt{n}}\right) + O\left(\frac{1}{\sqrt{n}}\right).
	\end{align*}
	This implies that our proposed dynamic pricing policy coincides with the static prices to a first-order approximation, but deviates from the static prices on the second order, i.e., order $1/\sqrt{n}$.

\section{Miscellaneous Proofs}\label{app:B}

	\begin{proof}[Proof of Lemma \ref{lem:2}]
		This proof follows in an almost identical fashion to Lemma 2 in \citet{AtaBarjestehKumar_SPN}, but we include it for completeness. It consists of four steps. We let $e^n$ denote the $n$th unit basis vector in a Euclidean space of appropriate dimension. That is, the $n$th component of $e^n$ is one, whereas its other components are zero. Moreover, recall from the discuss following Assumption \ref{ass:3} that for a vector $y\in\mathbb{R}^J$, we write $y = \left(y_B, y_N\right)$ where $y_B\in\mathbb{R}^b$ and $y_N\in\mathbb{R}^{J-b}$.
		
		\noindent\textbf{Step 1:} Consider the set of basic activity rates that do not cause any server idleness, i.e., $\left\{y\in\mathbb{R}^J: By_B=0,\,y_N=0\right\}$. First, we show that this set is the span of $\rttensortwo{C}$, defined next:
		\begin{align}
		\rttensortwo{C}&=\left\{e^j-e^{j^{\prime}}:\left(j,j^{\prime}\right)\in \bar{C},\,e^j,\,e^{j^{\prime}}\text{ are unit basis vectors in }\mathbb{R}^J\right\},\label{Cbarbar}
		\end{align}
		where $\bar{C}=\left\{\left(j,j^{\prime}\right): j,\,j^{\prime}\in\left\{1,\dots, b\right\}\text{ such that }s(j) = s(j^{\prime})\right\}$. That is, $\bar{C}$ is the set of all pairs of basic activities undertaken by the same server. Note that the difference $e^j - e^{j^{\prime}}$ in Equation (\ref{Cbarbar}) captures the trade-off server $s(j)$ makes by increasing the rate at which activity $j$ is undertaken (from its nominal value $x_j^*$) at the expense of decreasing the rate of activity $j^{\prime}$. By making such adjustments to the nominal basic activity rates $x_B^*$, the system manager can redistribute the workload between buffers $b(j)$ and $b(j^{\prime})$ without incurring any idleness. As such, we intuitively expect that taking linear combinations of such activity rates in $\rttensortwo{C}$ should yield the set of activity rates that do not result in any idleness, i.e., the set $\left\{y\in\mathbb{R}^J: By_B=0,\,y_N=0\right\}$. In summary, in Step 1 we prove that $$ \text{span}\left(\rttensortwo{C}\right)=\left\{y\in\mathbb{R}^J: By_B=0,\,y_N=0\right\}.$$
		To prove this, we show inclusions of both sets. First, let $y\in \left\{y\in\mathbb{R}^J: By_B= 0,\,y_N =0\right\}$. To prove that $y\in \text{span}\left(\rttensortwo{C}\right)$, we show that there exist constants $a_{jj^{\prime}}$, $(j,j^{\prime})\in\bar{C}$, such that $y = \sum_{(j,j^{\prime})\in\bar{C}}a_{jj^{\prime}}\left(e^{j}- e^{j^{\prime}}\right).$ To find these constants, it will be convenient to define the sets 
		\begin{align*}
		\bar{\mathcal{A}}_i = \mathcal{A}_i\cap \left\{1,\dots, b\right\},
		\end{align*}
		where $\mathcal{A}_i$ is the set of activities undertaken by server $i$, see Equation (\ref{eq:1.3}). To be more specific, $\bar{\mathcal{A}}_i$ consists of all basic activities undertaken by server $i$. After possibly relabeling, suppose that the basic activities are ordered so that
		\begin{align*}
		\bar{\mathcal{A}}_i = \left\{b_{i-1}+1,\dots, b_{i}\right\}\quad \text{for}\quad i=1,\dots, I,
		\end{align*}
		where $0=b_0<b_1<b_2<\cdots < b_{I}=b.$ We define constants $a_{jj^{\prime}}$ for $(j,j^{\prime})\in\bar{C}$ as follows: 
		\begin{align*}
		a_{jj^{\prime}}=\Bigg\{
		\begin{array}{ll}
		\sum_{l=b_{i-1}+1}^{k}y_l,&\text{if $\left(j,j^{\prime}\right)=(k,k+1)$ for $k=b_{i-1}+1,\dots, b_{i} -1$ and $i=1,\dots, I$},\\
		0,&\text{otherwise.}
		\end{array}
		\end{align*}
		Therefore, we have that
		\begin{align*}
		\sum_{(j,j^{\prime})\in\bar{C}}a_{jj^{\prime}}\left(e^j-e^{j^{\prime}}\right) &= \sum_{i=1}^{I}\sum_{k=b_{i-1}+1}^{b_{i}-1}a_{k,\,k+1}\left(e^{k}- e^{k+1}\right) \\&= \sum_{i=1}^{I}\sum_{k=b_{i-1}+1}^{b_{i} - 1}\left[\left(\sum_{l=b_{i-1}+1}^{k}y_{l}\right)\left(e^{k}- e^{k+1}\right)\right]\\
		&=\sum_{i=1}^{I}\Big[y_{b_{i-1}+1}\left(e^{b_{i-1}+1}-e^{b_{i-1}+2}\right)+\left(y_{b_{i-1}+1}+y_{b_{i-1}+2}\right)\left(e^{b_{i-1}+2}-e^{b_{i-1}+3}\right)\\
		&\qquad +\cdots +\left(\sum_{l=b_{i-1}+1}^{b_{i} -2}y_l\right)\left(e^{b_{i} -2}-e^{b_{i}-1}\right)+\left(\sum_{l=b_{i-1}+1}^{b_{i} -1}y_l\right)\left(e^{b_{i} -1}-e^{b_{i} }\right)\Big]\\
		&=\sum_{i=1}^{I}\Big[y_{b_{i-1}+1}e^{b_{i-1}+1}+\left(y_{b_{i-1}+2}+y_{b_{i-1}+1} - y_{b_{i-1}+1}\right)e^{b_{i-1}+2}\\
		&\qquad +\cdots +\left(\sum_{l=b_{i-1}+1}^{b_{i} -1}y_l-\sum_{l=b_{i-1}+1}^{b_{i} -2}y_l\right)e^{b_{i} -1}-\left(\sum_{l=b_{i-1}+1}^{b_{i} -1}y_l\right)e^{b_{i}}\Big]\\
		&= \sum_{i=1}^{I}\left[y_{b_{i-1}+1}e^{b_{i-1}+1}+y_{b_{i-1}+2}e^{b_{i-1}+2}+\cdots +y_{b_{i} -1}e^{ b_{i} -1}-\left(\sum_{l=b_{i-1}+1}^{b_{i} -1}y_l\right)e^{b_{i}}\right]\\
		&=\sum_{i=1}^{I}\left[y_{b_{i-1}+1}e^{b_{i-1}+1}+y_{b_{i-1}+2}e^{b_{i-1}+2}+\cdots +y_{b_{i} -1}e^{ b_{i} -1}-\left(-y_{b_{i}}\right)e^{b_{i}}\right]\\
		&=\sum_{i=1}^{I}\sum_{k=b_{i-1}+1}^{b_{i}}y_{k}e^{k}\\
		&=\sum_{j=1}^{J}y_j e^j,
		\end{align*}
		the first two equalities following from the definition of the $a_{jj^{\prime}}$, 
		the fourth equality from algebraic rearrangements, and the fifth equality from canceling terms. To derive the sixth equality note that $y$ satisfies $By_B=0$, which implies 
		\begin{align*}
		\sum_{l=b_{i-1}+1}^{b_i}y_l = 0\quad\text{for}\quad i=1,\dots, I.
		\end{align*}
		Equivalently, we have that
		\begin{align*}
		\sum_{l=b_{i-1}+1}^{b_i-1}y_l = -y_{b_i}\quad\text{for}\quad i=1,\dots, I.
		\end{align*}
		Substituting this for the last term of the fifth equality yields the sixth equality. Finally, the eighth equality from the facts that $y_N=0$ and that the sets $\bar{\mathcal{A}}_i$, $i=1,\dots, I$, partition the basic activities. Since $y = \sum_{j=1}^{J}y_j e^j$, we conclude that $y\in \text{span}\left(\rttensortwo{C}\right)$.
		
		Conversely, let $y\in \text{span}\left(\rttensortwo{C}\right)$. Then there are constants $a_{jj^{\prime}}$, $(j,j^{\prime})\in \bar{C}$, such that $$y = \sum_{(j,j^{\prime})\in\bar{C}}a_{jj^{\prime}}\left(e^{j}- e^{j^{\prime}}\right).$$ Since $\bar{C}$ consists only of pairs of basic activities, it follows that $y_N = 0$. Furthermore, for $(j,j^{\prime})\in\bar{C}$ and $i\in \left\{1,\dots, I\right\}$, we have
		\begin{align*}
		\left[A\left(e^j - e^{j^{\prime}}\right)\right]_i  = \sum_{l=1}^{b}A_{il}\left(e_l^j - e_l^{j^{\prime}}\right) = \sum_{l=1}^{b}\mathbf{1}_{\left\{s(l) = i\right\}}\left(e_l^j - e_l^{j^{\prime}}\right) = \mathbf{1}_{\left\{s(j) = i\right\}} - \mathbf{1}_{\left\{s(j^{\prime}) = i\right\}}=0,
		\end{align*}
		the second equality holding by Equation (\ref{eq:1.1}) and the fourth equality holding since $s(j) = s(j^{\prime})$. Therefore, $A\left(e^j - e^{j^{\prime}}\right)=0$ for all $(j,j^{\prime})\in \bar{C}$, implying that $Ay = 0$ by linearity. So, $y\in \left\{y\in\mathbb{R}^J: Ay= 0,\,y_N =0\right\}$.\vspace{0.5em}
		
		\noindent\textbf{Step 2:} In this step, we show that $\mathcal{N} = \text{span}\left(\tilde{C}\right)$, where $\tilde{C} = \left\{Ry: y\in \rttensortwo{C}\right\}.$
		To see this, recall that $\mathcal{N} = \left\{Hy_B: By_B=0,\, y_B\in\mathbb{R}^b\right\} = \left\{Ry: Ay=0,\, y_N=0\right\}.$
		Thus, it follows from Step 1 and the definition of $\tilde{C}$ that $\mathcal{N} = \text{span}\left(\tilde{C}\right)$.\vspace{0.5em}
		
		\noindent\textbf{Step 3:} In this step, we show that $\tilde{C} = \left\{\mu_j^*\left(e^{b(j)} - e^{b(j^{\prime})}\right): \left(j,j^{\prime}\right)\in\bar{C},\, e^{b(i)},\,e^{b(j^{\prime})}\in\mathbb{R}^I\right\}$. To see this, note that for $(j,j^{\prime})\in\bar{C}$ and $i\in\left\{1,\dots, I\right\}$, we have that
		\begin{align*}
		\left[R\left(e^j-e^{j^{\prime}}\right)\right]_i&=\sum_{l=1}^{J}R_{il}\left(e_l^j - e_l^{j^{\prime}}\right)=\sum_{l=1}^{J}\mu_{l}^*\mathbf{1}_{\left\{b(l) = i\right\}}\left(e_l^j - e_l^{j^{\prime}}\right)= \mu_{j}^*\mathbf{1}_{\left\{b(j) = i\right\}} - \mu_{j^{\prime}}^*\mathbf{1}_{\left\{b(j^{\prime}) = i\right\}}\\&= \mu_{j}^*\left(\mathbf{1}_{\left\{b(j) = i\right\}} -\mathbf{1}_{\left\{b(j^{\prime}) = i\right\}}\right)= \mu_j^* \left(e^{b(j)}_i - e^{b(j^{\prime})}_i\right),
		\end{align*}
		the second equality following from Equations (\ref{eq:1.2}) and (\ref{eq:1.7}) and the fourth equality following from the fact that $s(j) = s(j^{\prime})$ (since $(j,j^{\prime})\in \bar{C}$) and Equation (\ref{eq:1.12}). That is, 
		\begin{align}
		R\left(e^j - e^{j^{\prime}}\right) = \mu_j^*\left(e^{b(j)} - e^{b(j^{\prime})}\right)\quad \text{for}\quad (j,j^{\prime})\in\bar{C}.\label{eq:145*}
		\end{align}
		Then using the definition of $\rttensortwo{C}$, we write
		\begin{align*}
		\tilde{C}&=\left\{Ry: y\in\rttensortwo{C}\right\}\\
		&= \left\{Ry: y = e^j-e^{j^{\prime}}\text{ such that }(j,j^{\prime})\in\bar{C},\text{ }e^j,\,e^{j^{\prime}}\text{ are unit basis vectors}\right\}\\
		&=\left\{R\left(e^j - e^{j^{\prime}}\right):(j,j^{\prime})\in\bar{C},\text{ }e^j,\,e^{j^{\prime}}\text{ are unit basis vectors}\right\}\\
		&=\left\{\mu_j^*\left(e^{b(j)} - e^{b(j^{\prime})}\right):(j,j^{\prime})\in\bar{C},\text{ }e^j,\,e^{j^{\prime}}\text{ are unit basis vectors}\right\},
		\end{align*}
		where the last equality follows from Equation (\ref{eq:145*}). Hence, the result holds. In particular, by the definition of buffer communication, note that
		$$\tilde{C} = \left\{\mu_j^*\left(e^{i} - e^{i^{\prime}}\right): \text{buffers }i\text{ and }i^{\prime}\text{ communicate directly, } e^{i},\,e^{i^{\prime}}\in\mathbb{R}^I\right\}.$$
		
		\noindent\textbf{Step 4:} We consider the matrix $M$ defined in Lemma \ref{lem:2} (see Equation (\ref{eq:4.6*})) and show that its rows form a basis for $\mathcal{M}$. To that end, let $M^l$, $l=1,\dots, L$, be the rows of the matrix $M$ given in Equation (\ref{eq:4.6*}). Since the buffer pools partition the servers, the rows of $M$ are linearly independent. Thus, to complete the proof, it suffices to show that $\mathcal{M}=\text{span}\left(M^1,\dots, M^L\right)$. Recalling that $\mathcal{M}= \mathcal{N}^{\perp}$ and $\mathcal{N} = \text{span}\left(\tilde{C}\right)$, it follows that $a\in\mathcal{M}$ if and only if $a\cdot z = 0$ for all $z\in\tilde{C}$.
		Moreover, since $\mu_j^*>0$ for all $j\in \left\{1,\dots, b\right\}$, it follows from Step 3 that $$\mathcal{N} = \text{span}\left(\left\{e^{i} - e^{i^{\prime}}: \text{buffers }i\text{ and }i^{\prime}\text{ communicate directly, } e^{i},\,e^{i^{\prime}}\in\mathbb{R}^I\right\}\right).$$ Therefore, $a\in\mathcal{M}$ if and only if $a_{i} = a_{i^{\prime}}$ for all buffers $i$ and $i^{\prime}$ that communicate directly. 
		
		To prove that $\mathcal{M}=\text{span}\left(M^1,\dots, M^L\right)$ we show inclusions of both sets. On the one hand, let $a\in \mathcal{M}$. Then $a_i = a_{i^{\prime}}$ for all buffers $i$ and $i^{\prime}$ that communicate directly. By definition of buffer communication, it immediately follows that $a_i = a_{i^{\prime}}$ for all buffers $i$ and $i^{\prime}$ that communicate. That is, $a_i = a_{i^{\prime}}$ for all buffers $i$ and $i^{\prime}$ that are in the same buffer pool. Thus, $a\in \text{span}\left(M^1,\dots, M^L\right)$, implying that $\mathcal{M}\subseteq \text{span}\left(M^1,\dots, M^L\right)$. On the other hand, to show that $\text{span}\left(M^1,\dots, M^L\right)\subseteq \mathcal{M}$, it suffices to show that $M^l\in\mathcal{M}$ for each $l=1,\dots, L$. To that end, it is enough to show that $M^l_{i} = M^l_{i^{\prime}}$ for all buffers $i$ and $i^{\prime}$ that communicate directly. However, this trivially holds by Equation (\ref{eq:4.6*}), since buffers $i$ and $i^{\prime}$ that communicate directly are in the same buffer pool. Thus, $\text{span}\left(M^1,\dots, M^L\right)\subseteq \mathcal{M}.$
	\end{proof}
	
	\begin{proof}[Proof of Lemma \ref{lem:3}]
		It is enough to show that $(MR)_{lj}=(GA)_{lj}$ for all $l=1,\dots, L$ and $j=1,\dots, J$, where $G$ is given by Equation (\ref{eq:4.7*}). Indeed, by Equations (\ref{eq:1.2}), (\ref{eq:1.7}), and (\ref{eq:4.6*}),
		\begin{align}
		(MR)_{lj}&=\sum_{i=1}^{I}M_{li}R_{ij} = \sum_{i=1}^{I}\mathbf{1}_{\left\{i\in\mathcal{P}_l\right\}}\mu_j^*\mathbf{1}_{\left\{b(j)=i\right\}}=\mu_j^*\mathbf{1}_{\left\{b(j)\in\mathcal{P}_l\right\}}.\label{eq:7.17}
		\end{align}
		On the other hand, by Equations (\ref{eq:1.1}) and (\ref{eq:4.7*}),
		\begin{align}
		(GA)_{lj}&=\sum_{i=1}^{I}G_{li}A_{ij}=\sum_{i=1}^{I}\lambda_i^*\mathbf{1}_{\left\{i\in\mathcal{S}_l\right\}}\mathbf{1}_{\left\{s(j)=i\right\}}=\lambda^*_{s(j)}\mathbf{1}_{\left\{s(j)\in\mathcal{S}_l\right\}}.\label{eq:7.18}
		\end{align}
		Note that by Equation (\ref{eq:1.12}) we have $\mu_j^*=\lambda^*_{s(j)}$ and by Equation (\ref{eq:4.5}) we have that $b(j)\in\mathcal{P}_l$ if and only if $s(j)\in\mathcal{S}_l$. Thus, the desired result immediately follows by Equations (\ref{eq:7.17})--(\ref{eq:7.18}).
	\end{proof}
	
	\begin{proof}[Proof of Lemma \ref{lem:4}]
		When $L=1$, all buffers are in a single buffer pool. Thus, it follows immediately from Equation (\ref{eq:4.6*}) that $M= e^{\prime}$. Furthermore, by definition of buffer communication and Equation (\ref{eq:4.5}), there is a single server pool. It then follows from Equation (\ref{eq:4.7*}) that $G=\left(\lambda^*\right)^{\prime}$. 
		
		To prove the first relationship in Equation (\ref{eq:5.1*}), note that $$M\eta q = \eta Mq = \eta e^{\prime}q = \eta \sum_{i=1}^{I}q_i =\eta,$$ where the second equality follows from $M=e^{\prime}$ and where the fourth  equality follows from the fact that $q$ is a probability vector. To prove the second relationship in Equation (\ref{eq:5.1*}), first note that $MC=e^{\prime}\in\mathbb{R}^J$. This follows from $M=e^{\prime}\in\mathbb{R}^I$, the definition of $C$ in Equation (\ref{eq:1.2}), and the fact that $C$ has one nonzero element per column. Therefore, $$MC \text{diag}(x^*) A^{\prime} = e^{\prime}\text{diag}(x^*)A^{\prime} = (x^*)^{\prime}A^{\prime} = \left(Ax^*\right)^{\prime}=e^{\prime}\in\mathbb{R}^I,$$
		the fourth equality following from the heavy traffic assumption, see Equation (\ref{eq:1.15}).
	\end{proof}
	
	\begin{proof}[Proof of Lemma \ref{lem:5}]
		This is a straightforward convex optimization problem. Forming the Lagrangian \[{\cal L}(\zeta, \nu) = \sum_{i=1}^I \alpha_i \zeta_i^2 -\nu \sum_{i=1}^I \zeta_i + \nu x,\] where $\nu$ is the Lagrange multiplier, the necessary first-order conditions then give \[\zeta_i^* = \frac{\gamma}{2\alpha_i}, \quad i=1,\ldots, I.\]
        Substituting this into the constant $e'\zeta=x$ yields $\nu=2x/\hat{\alpha}$ and \begin{align}
            \zeta_i^*=\frac{x}{\alpha_i\hat{\alpha}}, \quad i=1,\ldots, I. \label{eq:Zeta}
        \end{align}
        The optimality of this solution follows from the convexity of the objective. Substituting (\ref{eq:Zeta}) in the objective function yields $c(x)=x^2/\hat{\alpha}$ as desired.
	\end{proof}
	
		\begin{proof}[Proof of Proposition \ref{prop:1}]
			Let $(Y,\zeta)$ be an admissible control for (\ref{eq:3.16})-(\ref{eq:3.21}) with the corresponding state process $Z$ and idleness process $U$. Letting $W(t)=MZ(t)$ for $t\geq 0$, (\ref{eq:3.16}) implies that (\ref{eq:RBCP_W}) holds, and (\ref{eq:RBCP_ReduceW}) holds by definition. Similarly, (\ref{eq:RBCP_ZPositive})-(\ref{eq:RBCP_UNondecreasing}) follow from (\ref{eq:3.19})-(\ref{eq:3.21}) whereas (\ref{eq:4.14}) follows from (\ref{eq:3.22}). Thus, $(Z, U, \zeta)$ of the BCP formulation (\ref{eq:3.16})-(\ref{eq:3.21}) is an admissible policy for the RBCP (\ref{eq:4.9})-(\ref{eq:4.14}). Because the two formulations have the same process $Z, U, \zeta$, they have the same cost. 

            The converse follows exactly as in (the second part of) the proof of Theorem 1 in \citet{HarrisonVanMieghem1997} (see pages 753-754) with the only substantive difference being (aside from the obvious notational differences) the process $X$ on their Equation (36) on page 755 is replaced with \[B(t) -\eta q\int_0^t e' Z(s) ds- C {\rm diag}(x^*)\int_0^t \kappa(s) ds \] in our setting. Then following the same steps in their proof shows that the analogy of the process $Y$ (in our setting) defined as in their Equation (35) and $\zeta$ is admissible for our BCP (\ref{eq:3.16})-(\ref{eq:3.21}). Moreover, because $(Y, \zeta)$ results in the same queue length process $Z$. Its cost is the same as that of the policy $(Z, U, \zeta)$ for RBCP (\ref{eq:4.9})-(\ref{eq:4.14}).   
		\end{proof}
		
		\begin{proof}[Proof of Proposition \ref{prop:2}]
			Given an admissible policy $\theta$ for EWF and the corresponding process $W, L$, we set $Z_{i^*}\equiv W$ and $Z_i\equiv 0$ for $i\not= i^*$ and $U_{k^*} \equiv L$ and $U_k\equiv 0$ for $k\not= k^*$. 
            Moveover, we set $\zeta_i(s)=\theta(s)/(\alpha_i \hat{\alpha}_i)$ for $i=1,\ldots, I$, which results in $\sum_{i=1}^I \alpha_i \zeta_i^2 (s) = c(\theta(s))$ by Lemma \ref{lem:5}. Then it follows from (\ref{eq:LowestHoldingCostBuffIndex})-(\ref{eq:76}) that $(Z, U, \zeta)$ has the same cost in RBCP as $\theta$ does in EWF. 

            To prove the converse, let $Z, U, \zeta$ be an admissible policy for RBCP, and let \begin{align*}
                \theta(s)=e'\zeta(s),\; W(s) = e' Z(s), \text{ and } \; L(s)=(\lambda^*)' U(s), \quad s\geq 0.
            \end{align*}
            It is easy to verify that $\theta(\cdot)$ is admissible for EWF. Moreover, $c(\theta(s))\leq \sum_{i=1}^I \alpha_i \zeta_i^2 (s)$ by Lemma \ref{lem:5} and that $h W(s) \leq \sum_{i=1}^I (h_i-h_0) Z_i(s)$ and $r L(s)\leq c' U(s)$ for $s\geq 0$. Thus, the cost of $\theta$ for the EWF is less than or equal to that of the policy $(Z, U, \zeta)$ for the RBCP. 
		\end{proof}

	\begin{proof}[Proof of Proposition \ref{prop:3}]
		Consider the auxiliary stationary reflected diffusion on $[0,\infty)$, denoted by $\left\{\tilde{W}(t),\,t\ge 0\right\}$, associated with the drift rate function $-\left(\eta y -a + \theta^*(y)\right)$ and variance parameter $\sigma^2$. As noted on pages 470--471 of \citet{BrowneWhitt1995} -- also see \citet{Mandl1968} and \citet{KarlinTaylor1981} -- its probability density function, denoted by $\varphi$, is given as follows:
		\begin{align}
		\varphi\left(x\right) = \frac{\displaystyle\exp\left\{-\int_{0}^{x}\frac{2}{\sigma^2}\left(\eta y -a + \theta^*(y)\right)\,dy\right\}}{\displaystyle\int_{0}^{\infty}\exp\left\{-\int_{0}^{y}\frac{2}{\sigma^2}\left(\eta s -a + \theta^*(s)\right)\,ds\right\}\,dy},\quad x\in [0,\infty)\label{eq:153}
		\end{align}
		provided all integrals are finite, which we verify next. To this end, let $k=\inf\left\{y\ge 0: v(y)\ge 0\right\}$ where $(v,\beta^*)$ solve the Bellman equation (\ref{eq:6.5*})--(\ref{eq:6.6*}), and note from Equation (\ref{eq:6.6*}) that $-r\le v(y)\le 0$ for $y\le k$ and $0\le v(y) \le h/\eta$ for $y\ge k$. In order to verify the integrals above are finite, using Equation (\ref{eq:6.7}) note that
		\begin{align}
		\exp\left\{-\int_{0}^{y}\frac{2}{\sigma^2}\left(\eta s -a + \theta^*(s)\right)\,ds\right\} & = \exp\left\{-\int_{0}^{y}\frac{2}{\sigma^2}\left(\eta s -a + \frac{\hat{\alpha}}{2}v(s)\right)\,ds\right\}\notag\\
		&=\exp\left\{-\frac{\eta y^2-ay}{\sigma^2}\right\}\exp\left\{-\frac{\hat{\alpha}}{\sigma^2}\left[\int_{0}^{k}v(s)\,ds + \int_{k}^{y}v(s)\,dy\right]\right\}\notag\\
		&\le \exp\left\{-\frac{\eta y^2-ay}{\sigma^2}\right\}\exp\left\{\frac{\hat{\alpha}}{\sigma^2}rk\right\},\label{eq:154}
		\end{align}
		from which we also deduce that the integral in the denominator of the right hand side of Equation (\ref{eq:153}) is finite. Moreover, it follows from Equation (\ref{eq:154}) that the stationary diffusion $\tilde{W}$ has finite moments. In particular, 
		\begin{align}
		E\left[\tilde{W}(0)\right] = E\left[\tilde{W}(t)\right]<\infty,\quad t<\infty.\label{eq:155*}
		\end{align}
		Next, we define another auxiliary stationary diffusion, denoted by $\tilde{W}^*$, as follows: 
		\begin{align*}
		\tilde{W}^*(t) = W^*(0)+\tilde{W}(t).
		\end{align*}
		Noting $W^*(0)<\tilde{W}^*(0)$ almost surely, we define the stopping time $\tau$ as follows:
		\begin{align*}
		\tau = \inf\left\{t\ge 0: W^*(t)\ge \tilde{W}^*(t)\right\}
		\end{align*}
		and introduce the following process:
		\begin{align*}
		\hat{W}^*(t) = \Bigg\{
		\begin{array}{ll}
		W^*(t),&t<\tau,\\
		\tilde{W}^*(t),&t>\tau.
		\end{array}
		\end{align*}
		By the strong Markov property of diffusions, $\hat{W}^*$ has the same distribution as $W^*$. Moreover, 
		\begin{align*}
		\hat{W}^*(t)\le \tilde{W}^*(t),\quad t\ge 0.
		\end{align*}
		Therefore, we conclude that
		\begin{align}
		E\left[W^*(t)\right]&=E\left[\hat{W}^*(t)\right]\notag\\&\le E\left[\tilde{W}^*(t)\right]\notag\\&= W^*(0)+E\left[\tilde{W}(t)\right]\notag\\&=W^*(0)+E\left[\tilde{W}(0)\right]\notag\\&<\infty,\label{eq:156*}
		\end{align}
		where the second equality follows from the definition of $\tilde{W}^*$, the third equality from the stationarity of $\tilde{W}$, and the last equality from Equation (\ref{eq:155*}). Thus, we conclude from $W^*(t)\ge 0$ for $t\ge 0$ and Equation (\ref{eq:156*}) that
		\begin{align*}
		\lim_{t\rightarrow \infty}\frac{E\left[W^*(t)\right]}{t} \le \lim_{t\rightarrow \infty}\frac{W^*(0) + E\left[\tilde{W}(0)\right]}{t}=0,
		\end{align*}
		as desired.
	\end{proof}
	
	The next lemma aids in the proof of Lemma \ref{lem:vUnique}. To state the result, it will be convenient to rewrite Equations (\ref{eq:1})--(\ref{eq:2}) as follows:
	\begin{alignat}{2} 
	&v^{\prime}(x) &&= q_2v^2(x) + q_1(x)v(x) + q_0(x),\quad x\ge 0,\label{eq:1*}\\
	&v(0)&&=-r,\label{eq:2*}
	\end{alignat}
	where $q_0(x)=\frac{2}{\sigma^2}\left(\beta - hx\right)$,  $q_1(x)=\frac{2}{\sigma^2}(\eta x-a)$, and $ q_2=\frac{\hat{\sigma}}{2\sigma^2}>0$ for $x\ge 0$.
	
	\begin{lem}\label{lem:7}
		For each $v\in C^1[0,\infty)$ satisfying Equations (\ref{eq:1*})--(\ref{eq:2*}), $y(x) = \exp\left\{-q_2\int_{0}^{x}v(t)\,dt\right\}$ satisfies 
		\begin{align}
		&y^{\prime\prime}(x) -q_1(x)y^{\prime}(x) +q_2q_0(x)y(x)=0,\quad x\ge 0,\label{eq:3*}\\
		&y(0)=1,\quad y^{\prime}(0) = rq_2. \label{eq:4*}
		\end{align}
		Conversely, for each $y\in C^2[0,\infty)$ satisfying Equations (\ref{eq:3*})--(\ref{eq:4*}), $v=-y^{\prime}/(q_2y)$ satisfies Equations (\ref{eq:1*})--(\ref{eq:2*}).
	\end{lem}
	\begin{proof}[Proof of Lemma \ref{lem:7}]
		Let $v\in C^1[0,\infty)$ satisfy Equations (\ref{eq:1*})--(\ref{eq:2*}) and let $y(x) = \exp\left\{-q_2\int_{0}^{x}v(t)\,dt\right\}$. Then it follows that
		\begin{align*}
		y^{\prime\prime}(x) &-q_1(x)y^{\prime}(x) +q_2q_0(x)y(x) \\&=\left[\exp\left\{-q_2\int_{0}^{x}v(t)\,dt\right\}\cdot \left(-q_2v(x)\right)\right]^{\prime}-q_1(x) \left[\exp\left\{-q_2\int_{0}^{x}v(t)\,dt\right\}\cdot \left(-q_2v(x)\right)\right]\\&\qquad\qquad+q_2 q_0(x) \exp\left\{-q_2\int_{0}^{x}v(t)\,dt\right\}\\
		&=\left[\exp\left\{-q_2\int_{0}^{x}v(t)\,dt\right\}\cdot \left(-q_2v(x)\right)^2+\exp\left\{-q_2\int_{0}^{x}v(t)\,dt\right\}\cdot \left(-q_2v^{\prime}(x)\right)\right]\\
		&\qquad\qquad +q_2q_1(x) v(x)\exp\left\{-q_2\int_{0}^{x}v(t)\,dt\right\}+q_2 q_0(x) \exp\left\{-q_2\int_{0}^{x}v(t)\,dt\right\}\\
		&= q_2\exp\left\{-q_2\int_{0}^{x}v(t)\,dt\right\}\left[q_2v^2(x) -v^{\prime}(x) +q_1(x) v(x) +q_0(x)\right]\\&=0.
		\end{align*}
		Moreover, $y(0) = \exp\left\{-q_2\cdot 0\right\} = 1$ and $y^{\prime}(0) = -q_2\exp\left\{-q_2\cdot 0\right\}v(0) = rq_2$. 
		
		On the other hand, let $y\in C^2[0,\infty)$ satisfy Equations (\ref{eq:3*})--(\ref{eq:4*}) and let $v=-y^{\prime}/(q_2y)$. Then it follows that
		\begin{align*}
		v^{\prime}(x) &= \left[-\frac{y^{\prime}(x)}{q_2y(x)}\right]^{\prime} = -\frac{1}{q_2}\left[\frac{y^{\prime\prime}(x)}{y(x)} - \left(\frac{y^{\prime}(x)}{y(x)}\right)^2\right] = -\frac{y^{\prime\prime}(x)}{q_2y(x)} + q_2v^2(x)\\
		&=-\frac{q_1(x)y^{\prime}(x)}{q_2y(x)}+q_0(x) + q_2v^2(x) = q_2v^2(x) +q_1(x) v(x) +q_0(x).
		\end{align*}
		Moreover, $v(0) = -y^{\prime}(0)/\left(q_2y(0)\right) = -\left(rq_2\right)/\left(q_2\cdot 1\right)=-r$. This completes the proof.
	\end{proof}
	
	\begin{proof}[Proof of Lemma \ref{lem:vUnique}]
		It is known that Equations (\ref{eq:3*})--(\ref{eq:4*}) can be transformed into a degenerate hypergeometric equation known as a Kummer's equation; see \citet{PolyaninZaitsev2003}. Such equations are known to have confluent hypergeometric function solutions; see \citet{Bateman1953} and \citet{Abramowitz2003}. It then follows from Lemma \ref{lem:7} that Equations (\ref{eq:1*})--(\ref{eq:2*}) have a solution $v$. To complete the proof, we must show that the solution $v$ to Equations (\ref{eq:1*})--(\ref{eq:2*}) is unique. To this end, define the function $f$ by 
		\begin{align*}
		f(x, u) = q_2v^2 + q_1(x) v + q_0(x),\quad (x,v)\in [0,\infty) \times (-\infty,\infty).
		\end{align*}
		To prove uniqueness, it is enough to show that $f$ is locally Lipschitz in $v$, i.e., that $f$ is Lipschitz in $v$ when restricted to the compact domain $[0,N]\times [-M, M]$ where $N,M>0$. More specifically, local Lipschitzness will demonstrate uniqueness on each compact interval, which can then be easily extended to the positive real line. To this end, for $x\in [0,N]$ and $v_1,v_2\in [-M,M]$ we have that
		\begin{align*}
		\vert f(x, v_1) - f(x, v_2)\vert &= \left\vert q_2v_1^2 +q_1(x) v_1 - q_2v_2^2-q_1(x) v_2\right\vert \\
		&\le q_2 \left\vert v_1^2 - v_2^2\right\vert + |q_1(x)| \left\vert v_1-v_2\right\vert\\
		&= \left[q_2\left\vert v_1 + v_2\right\vert + \frac{2}{\sigma^2}\left(\eta x+|a|\right)\right]\cdot \left\vert v_1- v_2\right\vert\\&\le \left[2Mq_2 + \frac{2}{\sigma^2}\left(\eta N+|a|\right)\right]\left\vert v_1-v_2\right\vert.
		\end{align*} 
		Thus, $f$ is locally Lipschitz in $v$. This completes the proof.
	\end{proof}
\end{document}